\documentclass[12pt]{amsart}%
\usepackage{amsfonts}
\usepackage{graphicx}
\usepackage{amscd}
\usepackage{amsmath}
\usepackage{amssymb}
\usepackage{amsfonts}
\newtheorem{theorem}{Theorem}[section]
\theoremstyle{plain}
\newtheorem{corollary}[theorem]{Corollary}
\newtheorem{definition}[theorem]{Definition}
\newtheorem{lemma}[theorem]{Lemma}
\newtheorem{proposition}[theorem]{Proposition}
\newtheorem*{propositionA1}{Proposition A.1}
\newtheorem*{propositionA2}{Proposition A.2}

\theoremstyle{definition}
\newtheorem{example}[theorem]{Example}
\newtheorem{remark}[theorem]{Remark}
\newtheorem{claim}[theorem]{Claim}
\newtheorem*{remarkA.1}{Remark A.1}

\numberwithin{equation}{section}

\begin{document}
\title[Tauberians for wavelet and non-wavelet transforms]{Multidimensional Tauberian
theorems for wavelet and non-wavelet transforms}
\author[S. Pilipovi\'{c}]{Stevan Pilipovi\'{c}}
\address{Department of Mathematics and Informatics\\ University of Novi Sad\\ Trg Dositeja Obradovi\'ca 4\\ 21000 Novi Sad \\ Serbia }
\email{pilipovic@dmi.uns.ac.rs}
\author[J. Vindas]{Jasson Vindas}
\address{Department of Mathematics, Ghent University, Krijgslaan 281 Gebouw S22, B 9000 Gent, Belgium}
\email{jvindas@cage.Ugent.be}
\thanks{The work of S. Pilipovi\'{c} is supported by the Serbian Ministry of Science}
\thanks{J. Vindas gratefully acknowledges support by a Postdoctoral Fellowship of the Research
Foundation--Flanders (FWO, Belgium)}


\begin{abstract}
We study several Tauberian properties of regularizing transforms of tempered distributions with values in Banach spaces, that
is, transforms of the form $M^{\mathbf{f}}_{\varphi}(x,y)=(\mathbf{f}\ast\varphi_{y})(x)$, where the kernel $\varphi$ is a test
function and $\varphi_{y}(\cdot)=y^{-n}\varphi(\cdot/y)$. If the zeroth moment of $\varphi$ vanishes,
it is a wavelet type transform; otherwise, we say it is a non-wavelet type transform.

The first aim of this work is to show that the scaling (weak) asymptotic properties of distributions are \emph{completely} 
determined by boundary 
asymptotics of the regularizing transform plus natural 
Tauberian hypotheses. Our second goal is to characterize the spaces of Banach space-valued tempered distributions in terms of the transform $M^{\mathbf{f}}_{\varphi}(x,y)$. We investigate conditions which ensure that a distribution that a priori takes values in locally convex space actually takes values in a narrower Banach space. Special attention is paid to find the \emph{optimal} class of kernels $\varphi$ for which these Tauberian results hold. 

We give various applications of our Tauberian theory in the pointwise and (micro-)local regularity analysis of Banach space-valued 
distributions, and develop a number of techniques which are specially useful when applied to scalar-valued functions and distributions. Among such applications, we obtain the full weak-asymptotic series 
expansion of the family of Riemann-type distributions $R_{\beta}(x)=\sum_{n=1}^{\infty}e^{i\pi xn^{2}}/n^{2\beta}$, $\beta\in\mathbb{C}$, at every rational point.
We also apply the results to regularity theory within generalized function algebras,
to the stabilization of solutions for a class of Cauchy problems, and to Tauberian theorems for the Laplace transform; in addition, we find a necessary and sufficient condition for the existence of $f(t_0,\xi)\in\mathcal{S}'({\mathbb R}^n_\xi),$ where $f(t,\xi)\in\mathcal{S}'(\mathbb R^n_t\times\mathbb R^n_\xi).$
\end{abstract}
\maketitle

\newpage

\tableofcontents

\newpage

\section{Introduction}
\label{wnwi}

In this article we investigate several Tauberian aspects of integral transforms arising from regularizations of distributions. The seminal work of Drozhzhinov and Zavialov \cite{drozhzhinov-z2,drozhzhinov-z3,drozhzhinov-z5} has been used as the starting point for our analysis.  We shall connect our work with those of Meyer, Jaffard, Estrada, Holschneider and collaborators. This will be done via the fine pointwise analysis of vector-valued functions and distributions.

Fix $\varphi\in\mathcal{S}(\mathbb{R}^{n})$ and set
$\varphi_{y}(\cdot)=y^{-n}\varphi(\cdot/y)$. To a
tempered distribution $f$, we associate the \emph{regularizing transform},
given by
\begin{equation}
\label{wnwieq1}
M^{f}_{\varphi}(x,y):=(f\ast\varphi_{y})(x)\ , \ \ (x,y)\in\mathbb{R}^{n}\times\mathbb{R}_{+}.
\end{equation}
Such a transform is extremely important and useful in mathematical analysis. If the integral of $\varphi$ vanishes, one obtains the widely studied wavelet transform. In contrary case, the transform is rather obtained through an old procedure in analysis, namely, convolution with an approximation of the unity; we therefore choose to make a distinction and call it the non-wavelet transform. 

The aims of this article are the following:
\begin{enumerate}
\item [\textbf{1.}]To obtain a complete characterization of scaling (weak) asymptotic properties of distributions in terms of the transform (\ref{wnwieq1}).  
\item [\textbf{2.}] To show that (\ref{wnwieq1}) intrinsically characterizes the spaces of tempered distributions with values in arbitrary Banach spaces.
\item [\textbf{3.}] To identify the \emph{optimal} class of kernels $\varphi$ suitable for the \emph{Tauberian} analysis of the two problems stated above.
\item [\textbf{4.}] To develop a number of techniques which can be used as standard devices for pointwise and (micro-)local analysis of functions and distributions.
\end{enumerate}

The study and determination of the local behavior of functions (or generalized functions) is a fundamental necessity in almost any area of mathematical analysis and its applications. In the case of Schwartz distributions, the problem is not easy to handle: Distributions are not pointwisely defined objects, so, how can one study their behavior at individual points? There are various different views of the problem. One approach consists in fixing a global space of functions, typically a H\"{o}lder or Besov type space, and measuring the \emph{local regularity} with respect to it: a distribution is said to be regular at a point if it coincides near the point with an element of the global space, see e.g., \cite{holschneider94,moritoh-yamada,triebel}. 

In several contexts, such as propagation of singularities of PDE \cite{bony1,bony2}, image or signal processing \cite{adelson}, or multifractal analysis \cite{jaffard97a,jaffard97b,jaffard2006a}, one is most interested in finer \emph{pointwise} measurements that allow one to distinguish special features of a distribution in an irregular background. Such a pointwise behavior may drastically and suddenly change from point to point, which makes the \emph{local} regularity approach sometimes inadequate for these purposes. Perhaps a quite representative example of a function exhibiting this phenomenon is Riemann's ``non-differentiable'' function,  
\begin{equation}
\label{wnwieqR}
\sum_{n=1}^{\infty}\frac{\sin(\pi n^{2}t)}{n^{2}},
\end{equation}
which have attracted for more than a century the attention of many researchers \cite{hardy1916,hardy-littlewood1914,gerver,duistermaat,holschneider-t,jaffard96,jaffard-m}. Its behavior at a point deeply depends on Diophantine approximation properties of the point, and that makes it change radically from point to point \cite{jaffard96}. Another instance is provided by the related family of Riemann type distributions 
\begin{equation}
\label{wnwieq2} R_{\beta}(t)=\sum_{n=1}^{\infty}\frac{e^{i\pi n^{2}t}}{n^{2\beta}}, \ \ \ \beta\in\mathbb{C},
\end{equation}
whose pointwise properties at rational points will be throughly investigated in this article. 

Jaffard and Meyer \cite{jaffard-m} showed that a better understanding of (\ref{wnwieqR}), and other functions, can be achieved via the transform (\ref{wnwieq1}) and a careful analysis of the fine scaling and oscillating properties of distributions. The key pointwise scaling notion involved in their analysis is that of $2$-microlocal spaces, introduced by Bony \cite{bony1,bony2} in the context of PDE and subsequently studied by many authors \cite{jaff,jaffard2006,jaffard-m,kempka,meyer,moritoh-yamada,seuret-vehel}.

Back in the early 1970's, Zavialov introduced a natural notion of measuring for scaling asymptotic properties of distribution, which is actually closely related to the 2-microlocal spaces. His notion was originated in the setting of quantum field theory \cite{vladimirov-z1,vladimirov-z2,zavialov} and it has been extensively investigated by him, Drozhzhinov, and Vladimirov, see e.g., \cite{drozhzhinov-z1,drozhzhinov-z3,drozhzhinov-z4,drozhzhinov-z6,vladimirov-d-z1}. Meyer rediscovered in \cite{meyer} a particular case of this scaling concept and used it to define the pointwise \emph{weak scaling exponent}.
The study of Zavialov scaling asymptotics is the core of problem \textbf{1} stated above.

In this article we will employ the name \emph{weak-asymptotics} for Zavialov's notion, though we should mention that the same concept is more often known in the literature as \emph{quasiasymptotics}. The idea of the weak-asymptotic behavior is to look for scaling self-similarity properties of a distribution $f\in\mathcal{S}'(\mathbb{R}^{n})$ at either small or large scale, so one is interested in asymptotic representations of the form $f(ht)\sim c(h)g(t)$, in the distributional sense, that is, holding after evaluation at test functions $\rho\in \mathcal{S}(\mathbb{R}^{n})$,
\begin{equation}
\label{wnwieqwa1}
\left\langle f(ht),\rho(t)\right\rangle\sim c(h) \left\langle g(t),\rho(t)\right\rangle.
\end{equation}
One quickly realizes \cite{estrada-kanwal2,pilipovic-stankovic-takaci,vladimirov-d-z1} that the comparison function $c(h)$ must be asymptotically self-similar, i.e., a Karamata regularly varying function \cite{bingham,seneta}. Familiar functions such as  $h^{\alpha}$, $h^{\alpha}\left|\log h\right|^{\beta},$ $ h^{\alpha}\left|\log\left|\log h\right|\right|^{\beta}$, ..., are regularly varying. Similarly, one can define $f(ht)=O(c(h))$ in the weak sense through
\begin{equation}
\label{wnwieqwa2}
\left\langle f(ht),\rho(t)\right\rangle=O (c(h)).
\end{equation}
By translating and looking at small scales, both notions provide a natural and qualitative measure of the \emph{pointwise regularity} of distributions. The connection between weak-asymptotics and 2-microlocal spaces is as follows. It was shown by Meyer in \cite{meyer}. A tempered distribution $f$ belongs (locally) to the 2-microlocal space \cite{meyer} $C^{\alpha,s}_{x_{0}}$ for some $s$ if and only if it satisfies (\ref{wnwieqwa2}) with $c(h)=h^{\alpha}$, as $h\to0^{+}$, for all test function in the Lizorkin \cite{holschneider,lizorkin,samko} space $\mathcal{S}_{0}(\mathbb{R}^{n})\subseteq\mathcal{S}(\mathbb{R}^{n})$, i.e., for all those test functions with all vanishing moments. Therefore, we may say that the weak-asymptotics have a microlocal character.

The weak-asymptotics admit a precise characterization in terms of (\ref{wnwieq1}). Our aim \textbf{1} is to establish such a characterization. For a suitable kernel $\varphi$, we shall show that (\ref{wnwieqwa2}) holds (up to a possible polynomial correction) for all $\rho\in\mathcal{S}(\mathbb{R}^{n})$ if and only if there is $k$ such that  
\begin{equation}
\label{wnwieq3} \left|M^{f}_\varphi(hx,hy)\right|\leq y^{-k}O(c(h)), \ \ \ \mbox{uniformly for } \left|x\right|^{2}+y^{2}=1.
\end{equation}
On the other hand, the weak-asymptotic behavior $f(ht)\sim c(h)g(t)$ holds if and only if (\ref{wnwieq1}) has a related angular asymptotic behavior plus the Tauberian estimate (\ref{wnwieq3}). Drozhzhinov and Zavialov made quite substantial progress toward the understanding of such a characterization \cite{drozhzhinov-z2,drozhzhinov-z3}. The problem was considered by Meyer as well \cite{meyer}. We will revisit the problem and obtain optimal results. 

Naturally, we cannot expect the characterization to hold for any kernel $\varphi$. Thus, we have partly arrived to the aim \textbf{3}: The determination of the biggest class of kernels that makes possible such a result. Drozhzhinov and Zavialov proved that the desired characterization is valid if $\hat{\varphi}$, the Fourier transform of $\varphi$, satisfies a non-degenerateness requirement; specifically, if it has a Taylor polynomial at the origin that is non-degenerate, in the sense that such a Taylor polynomial does not identically vanish on any ray through the origin. We will identify the biggest class of kernels associated to this Tauberian problem by finding a more general condition of non-degenerateness. It turns out that the structure of the Taylor polynomials does not play any role in our condition.  

In Wiener Tauberian theory \cite{korevaar1965,wiener} and its extensions \cite{feichtinger-s,korevaarbook,pilipovic-stankovicWTD} the Tauberian kernels are those whose Fourier transforms do not vanish at any point. In our theory the Tauberian kernels will be those $\varphi$ such that $\hat{\varphi}$ does not identically vanish on any ray through the origin. This is precisely our notion of non-degenerateness, which fully answers the problem \textbf{3}.

We now explain the problem \textbf{2} stated above. We claim that the mere knowledge of the transform (\ref{wnwieq1}) along with its growth properties suffice to conclude whether a vector-valued distribution takes values in a Banach space (up to some correction term). More precisely, suppose that the vector-valued tempered distribution $\mathbf{f}$ takes a priori values in a ``broad'' locally convex space which contains as a continuously embedded subspace the narrower Banach space $E$, and that $M^{\mathbf{f}}_\varphi(x,y)\in E$, for almost every value of $(x,y)$. 
If it is a priori known that $\mathbf{f}$ takes values in $E$, then one can directly verify that it satisfies the estimate
\begin{equation}
\label{wnwieq4} 
\left\|M^{\mathbf{f}}_{\varphi}(x,y)\right\|_{E}\leq
C \frac{(1+y)^{k}\left(1+\left|x\right|\right)^{l}}{y^{k}},
\end{equation}
for some $k$, $l$, and $C$. We call (\ref{wnwieq4}) a \emph{class estimate}. The problem of interest is the converse one: Up to what extend does the class estimate (\ref{wnwieq4}) allow one to conclude that $\mathbf{f}$ actually takes values in $E$? 

The aim \textbf{2} is to address this question. We shall show that if $\mathbf{f}$ satisfies (\ref{wnwieq4}), where $\varphi$ is non-degenerate in our sense, then $\mathbf{f}$ takes values in $E$ up to  vector-valued entire functions whose Fourier transforms have compact supports that are totally controlled by $\varphi$.

We point out that this problem was first raised and studied by Drozhzhinov and Zavialov \cite{drozhzhinov-z2,drozhzhinov-z3}. Their results are robust and they demonstrated their usefulness with several very interesting applications; for example, they discovered a general abstract class of one-dimensional Besov type spaces in \cite{drozhzhinov-z2}, and established in \cite{drozhzhinov-z5} useful norm estimates in various Banach spaces for solutions to the Schr\"{o}dinger equation. However, their theory excludes a wide number of kernels that are of great interest in harmonic analysis. For instance, their theory does not cover wavelets with all vanishing moments, $\varphi \in\mathcal{S}_{0}(\mathbb{R}^{n})$. Observe that the wavelets that are involved in (continuous) Littlewood-Paley decompositions \cite{hormander} are of this kind. Our results will incorporate those (and many other) important types of wavelets. 

The two type of Tauberian results described above are powerful tools, especially when they are combined together. It is important to emphasize that they provide a new perspective in the study of classical pointwise and local regularity properties of ordinary \emph{functions} as well. They even supply techniques to study global properties, since, as we will show, global regularity may be explicitly transfer into weak pointwise properties of vector-valued distributions. 

Our aim \textbf{4} is to illustrate these facts with several concrete applications. Such applications will be detailed in the next subsection. 
Perhaps the most representative of them  concerns the pointwise analysis of the Riemann type distributions (\ref{wnwieq2}). We show that for \emph{any} $\beta$  the distribution (\ref{wnwieq2}) admits a full weak-asymptotic series at \emph{every} rational point. Interestingly, the pointwise behavior of (\ref{wnwieq2}) is intimately related to the analytic continuation properties of the generalized zeta-type function
\begin{equation*}
\zeta_{r}(z):=\sum_{n=1}^{\infty}\frac{e^{i\pi r n^{2}}}{n^{z}}, \ \ \ \Re e\:z>1,
\end{equation*} 
provided in this article for $r\in\mathbb{Q}$.

\subsection{Outline of Contents} 
The paper is organized as follows.

Section \ref{wnwn} is of preliminary character, we recall there basic facts about Banach space-valued (Schwartz and Lizorkin \cite{samko}) distribution spaces and the notion of weak-asymptotics. Several special cases of weak-asymptotics will be introduced as examples in order to show how such a kind of asymptotic behavior is involved in familiar notions from distribution theory. Those particular instances are: \L ojasiewicz point values \cite{lojasiewicz,lojasiewicz2} and jump behavior of distributions \cite{estrada-vindas,vindas-estrada1,vindas-estrada4}, the Estrada-Kanwal moment asymptotic expansion \cite{estrada-kanwal90,estrada-kanwal2}, and Meyer's weak scaling exponents \cite{meyer}.

We discuss in Section \ref{wnwtd} properties of the regularizing transform and extend the wavelet analysis given in
\cite{holschneider} to Banach space-valued distributions. We first treat the basic properties that are common to both the wavelet and non-wavelet transforms. After a normalization, we will follow \cite{estrada-vindasI,vindas-estrada5} and use the terminology \emph{$\phi$-transform} for the non-wavelet case. We introduce here the crucial concept of \emph{non-degenerate} test functions, which is wider than that introduced by Drozhzhinov and Zavialov in \cite{drozhzhinov-z4} and called here \emph{strong non-degenerateness}. They will be indeed the Tauberian kernels of our theory. We mention that the same class of test functions has been already used by other authors in other contexts, see e.g., \cite{hyton-lutz}. 

The extension of the scalar distribution
wavelet analysis from \cite{holschneider} to
Banach space-valued Lizorkin distributions will be accomplished in Subsection \ref{waE} with the aid of the nuclearity of the Schwartz spaces \cite{treves}. The wavelet desingularization formula and the characterization of bounded sets of Banach space-valued Lizorkin distributions will prepare the ground for the vector-valued theory given in the next four sections.

The Abelian results from Section \ref{wnwa} are
essentially due to Drozhzhinov and Zavialov
\cite{drozhzhinov-z2,drozhzhinov-z3}, but we shall refine their results
by adding new information through uniformity conditions occurring
in the angular asymptotic behavior of the transform (\ref{wnwieq1}).

Section \ref{wnwtc} is devoted to the Tauberian characterization of weak-asymptotics in the space   $\mathcal S_0'(\mathbb R^n,E)$, i.e., the space of Lizorkin distributions with values in a Banach space $E$ (cf. Subsection \ref{wnwdsp}). The key condition in our theory will be
an $E$-norm version of the Tauberian estimate (\ref{wnwieq3}) on the half-sphere of $\mathbb{R}^{n}\times \mathbb{R}_{+}$. We have extended all of the one-dimensional results from \cite{vindas-pilipovic-rakic}. We highlight that the technique employed here, comparing it with that from \cite{vindas-pilipovic-rakic}, has been considerably simplified and refined.

The main results of the article are obtained in Section \ref{wnwtt} and Section \ref{wnwce}. In Section \ref{wnwtt} we undertake the Tauberian study of weak-asymptotics in $\mathcal{S}'(\mathbb{R}^{n},E)$, the space of tempered distributions with values in a Banach space $E$. In Subsection \ref{wnwtphi} we characterize the weak-asymptotics in terms of the $\phi-$transform, while in Subsection \ref{wnwwaw} we show that such asymptotic properties can be also fully characterized (up to a polynomial correction) via the wavelet transform. The estimate (\ref{wnwieq3}) will play again the role of the Tauberian hypothesis. An important achievement here is the analysis of the weak-asymptotics of critical degrees, namely, positive integral degrees, not present in \cite{drozhzhinov-z2,drozhzhinov-z3} nor in Meyer's work \cite{meyer}. In such a critical case the classes of associate homogeneous and homogeneously bounded functions \cite{vindas2,vindas4}, defined in Subsection \ref{aah}, will appear as natural terms in the polynomial correction. 

Section \ref{wnwce} deals with the characterization of $\mathcal{S}'(\mathbb{R}^{n},E)$ and $\mathcal{S}'_{0}(\mathbb{R}^{n},E)$, where $E$ is a Banach space, through 
 \emph{global and local class estimates}. We assume that $\mathbf{f}$ takes a priori values in a broad locally convex space which contains the narrower Banach space $E$ as a continuously embedded subspace, and that $M^{\mathbf{f}}_\phi(x,y)\in E$, for almost every value of $(x,y)$, satisfies the Tauberian class estimate (\ref{wnwieq4}). We show first that the global class estimate (\ref{wnwieq4}) determines $\mathcal{S}'_{0}(\mathbb{R}^{n},E)$ if the wavelet is non-degenerate. On the other hand, for $\mathcal{S}'(\mathbb{R}^{n},E)$, we prove the existence of $\mathbf{G}$, with values in the broad
space, such that $\operatorname*{supp}\hat{\mathbf{G}}\subseteq\left\{0\right\}$
and 
\begin{equation}
\label{wnwieq5}
\mathbf{f}-\mathbf{G}\in\mathcal{S}'(\mathbb{R}^{n},E).
\end{equation} In
case when the broad space is a normed one, $\mathbf{G}$ reduces
simply to a polynomial. The important notions of strong non-degenerateness and that of
degree of non-degenerateness, introduced in Subsections \ref{lce} and \ref{wnwces}, enable us to analyze a local version of the class estimate (\ref{wnwieq4}), that is, when it is assumed to hold just for $(x,y)\in\mathbb{R}^{n}\times(0,1]$. For strongly non-degenerate wavelet, the result described above remains unchanged if one replaces the global class estimate (\ref{wnwieq4}) by a local one. It is then shown that if $\varphi$ is non-degenerate, (\ref{wnwieq5}) remains true where the support of $\hat{\mathbf{G}}$ may not be any longer the origin but it is completely determined by the wavelet. For the $\phi-$transform, $\mathbf{G}$ does not occur. 
Our results are \emph{sharp}; we supply various examples to support this claim.

On the basis of previous the two sections, we develop in Section \ref{wnwap} several useful techniques which aim to be tools for the pointwise, asymptotic, and (micro-)local analysis of functions and distributions. It is not the scope of this section to investigate general problems. Our purpose is rather to explore the applicability of our Tauberian theorems in concrete situations which allow us to show the methods and relate them with the work of other authors. Nevertheless, many of the methods provided here seem most likely to be generalizable to other contexts (such as a more detailed analysis of the classical 2-microlocal spaces \cite{bony1,bony2,jaffard-m} or the generalized ones in the sense of Jaffard \cite{jaffard2006}).

We introduce in Section \ref{wnwla} a new class of pointwise spaces of Banach space-valued distributions, the \emph{pointwise weak H\"{o}lder space} with respect to regularly varying functions, and characterize them by growth estimates on the wavelet transform. These spaces are extended versions of Meyer's pointwise spaces $\mathcal{O}^{s}(x_{0})$ and $\Gamma^{s}(x_{0})$ from \cite{meyer}. Remarkably, although our spaces are \emph{pointwisely} defined, they are also very effective tools for studying \emph{global} regularity. The key point is to translate weak pointwise behavior of vector-valued distributions into global regularity for scalar-valued distributions. For instance, this approach will lead to a generalization of Holschneider's wavelet inverse theory for global regularity \cite{holschneider94,holschneider}. As a second application, we recover two Tauberian theorems of Meyer \cite{meyer} and Jaffard \cite{jaff} on the determination of pointwise H\"{o}lder exponents from size estimates of the wavelet transform.

In Subsection \ref{wnwPDE} we study sufficient conditions for asymptotic stabilization in time of solutions to the evolution equation $U_t-P(\partial/\partial x)U=0,$ $U_{|t=0}=f\in\mathcal{S}'$, where $P$ is a homogeneous polynomial such that $\Re e\:P(iu)$ is negative with respect to a cone. Subsection \ref{wnwCA} explores applications of our Tauberian techniques to regularity theory within generalized function algebra \cite{colombeau,hor,ober 001}; we show the regularity theorem from \cite{hor}: under certain growth condition for $(f*\varphi_\varepsilon)_\varepsilon$, the distribution $f$ must be smooth with all derivatives being polynomially bounded. Subsection \ref{wnwds} provides a wavelet characterization of Estrada's distributionally small distributions \cite{estrada98,estrada-kanwal2}, namely, the elements of the dual space of $\mathcal K$ (the space of symbols studied in \cite{gls}). The pointwise analysis of the family of Riemann's distributions (\ref{wnwieq2}) at the rationals will carried out in Subsection \ref{wnwWRfunction}. We discuss in Subsection \ref{wnwLT} Tauberian theorems for Laplace transforms of distributions with supports in cones. Such theorems will be obtained under milder conditions over the cone than those from \cite{drozhzhinov-z0,
vladimirov-d-z1}. We apply then these ideas to produce a new proof of
Littlewood's Tauberian theorem \cite{hardy,korevaarbook,littlewood}. As our last application, we
give in Subsection \ref{wnwDS} a new and very short proof of the fact that the
weak-asymptotics of $f\in\mathcal
S'(\mathbb R^n;E)$ over $\mathcal D(\mathbb R^n)$ is equivalent to
the weak-asymptotics over $\mathcal S(\mathbb R^n).$

In Section 9 we extend our considerations to a more general setting; we explain that all Abelian and Tauberian results from Sections \ref{wnwa}--\ref{wnwce}
remain true if we replace the Banach space by a DFS space \cite{silvas} (strong duals of Fr\'{e}chet-Schwartz spaces), or more generally, by a regular inductive limit of an increasing sequence of Banach spaces which additionally has the Montel property \cite{kom}. Most spaces of generalized functions, such as spaces of distributions, ultradistributions \cite{kom} and Fourier hyperfunctions \cite{kaneko}, are of either this kind or projective limits of this kind of spaces. As an illustration of these abstract considerations, we give a necessary and sufficient condition for a tempered distribution $f$ on $\mathbb R^n_t\times\mathbb R^m_\xi$ to have trace at $t=t_0,$ i.e., for the existence of $ f(t_0,\cdot)\in\mathcal S'(\mathbb R^m).$

Finally, the purpose of the Appendix is to establish the precise connection between weak-asymptotics in $\mathcal{S}'_{0}(\mathbb{R}^{n},E)$ and
$\mathcal{S}'(\mathbb{R}^{n},E)$. Such a relation showed to be a valuable part of the technique employed in Section \ref{wnwtt}. 

\newpage

\section{Notation and Preliminaries}
\label{wnwn}
We use the notation $\mathbb{H}^{n+1} = \mathbb{R}^{n} \times
\mathbb{R_+};$ $\mathbb{S}^{n-1}$ is the unit sphere; $\left|x\right|$ is the  Euclidean norm, $x\in\mathbb{R}^{n}$; $\left|m\right|=m_{1}+m_{2}+\dots+m_{n}$, for $m\in\mathbb{N}^{n}$, where $\mathbb{N}$ includes $0$; $\varphi^{(m)}=(\partial^{\left|m\right|}/\partial x^{m})\varphi$, $m\in\mathbb{N}^{n}$. 

The space $E$ always denotes a fixed, but arbitrary, Banach space with norm $\left\|\:\cdot\:\right\|$.

If $\mathbf{a}:I\mapsto E$ and $T:I\mapsto\mathbb{R}_{+}$, where $I=(0,A)$ (resp. $I=(A,\infty)$) we write $\mathbf{a}(y)=o(T(y))$ as $y\to0^{+}$ (resp. $y\to\infty$) if $\left\|\mathbf{a}(y)\right\|=o(T(y))$. We shall use a similar convention for the big $O$ Landau symbol. Let $\mathbf{v}\in E$, we write $\mathbf{a}(y)\sim T(y) \mathbf{v}$ if $\mathbf{a}(y)=T(y) \mathbf{v}+o(T(y))$.

\subsection{Spaces of $E$-valued Distributions} 
\label{wnwdsp} The Schwartz spaces \cite{schwartz1} of smooth compactly supported and rapidly decreasing test functions are denoted by $\mathcal{D}(\mathbb{R}^{n})$ and $\mathcal{S}(\mathbb{R}^{n})$; their dual spaces, the scalar spaces of distributions and tempered distributions, are
$\mathcal{D}'(\mathbb{R}^{n})$ and $\mathcal{S}'(\mathbb{R}^{n})$. We denote by $\mathcal{E}(\mathbb{R}^{n})$ the
space of $C^{\infty}$-functions, while $\mathcal{E}'(\mathbb{R}^{n})$ stands for the space of compactly supported
distributions.  We will use the Fourier transform
$$\hat{\varphi}(u)=\int_{\mathbb{R}^{n}}\varphi(t)e^{-iu \cdot t}\mathrm{d}t, \ \ \ u\in\mathbb{R}^{n},\;
\varphi\in\mathcal{S}(\mathbb{R}^{n}).$$ Following
\cite{holschneider}, the space
$\mathcal{S}_0(\mathbb{R}^{n})$ of highly time-frequency localized
functions over $\mathbb{R}^{n}$ is defined as the subspace of $\mathcal{S}(\mathbb{R}^{n})$ consisting of  those elements for which all their moments vanish,
i.e., $$\eta\in \mathcal{S}_0(\mathbb{R}^{n})\; \mbox{ if and only
if }\;  \int_{\mathbb{R}^{n}}t^m\eta(t)\mathrm{d}t=0, \;\mbox{ for
all }\; m\in\mathbb{N}^{n}.$$ It is provided with the relative
topology inhered from  $\mathcal{S}(\mathbb{R}^{n})$. This
space is also known as the \emph{Lizorkin space} of test functions
\cite{lizorkin,samko}, and it is invariant under Riesz potential
operators. We must emphasize that
$\mathcal{S}_{0}(\mathbb{R}^{n})$ is different from the one used
in \cite{drozhzhinov-z4}. The corresponding space of highly
localized function over $\mathbb H^{n+1}$ is denoted by
$\mathcal{S}(\mathbb {H}^{n+1})$. It consists of those $\Phi\in
C^{\infty}(\mathbb{H}^{n+1})$ for which
$$\sup_{(x,y)\in
\mathbb {H}^{n+1}}\,\left(y+\frac
{1}{y}\right)^{k_{1}}\left(1+\left|x\right|\right)^{k_{2}}\,\left|\frac{\partial^{l}}{\partial y^{l}}\frac{\partial^{m}}{\partial x^{m}}\Phi (x,y)\right|<\infty,$$
for all $k_{1},k_{2},l\in \mathbb{N}$ and $m\in\mathbb{N}^{n}$.
The canonical topology of this space is defined in the standard way \cite{holschneider}.

Let $\mathcal{A}(\Omega)$ be a topological vector space of test function over
 an open subset $\Omega\subseteq\mathbb{R}^{n}$.
 We denote by $\mathcal{A}'(\Omega,E)=L_{b}(\mathcal{A}(\Omega),E)$, the space of continuous
 linear mappings from $\mathcal{A}(\Omega)$ to $E$ with the topology of uniform convergence over
 bounded subsets of $\mathcal{A}(\Omega)$ \cite{treves}. We are mainly concerned with
 the spaces $\mathcal{D}'(\mathbb{R}^{n},E)$, $\mathcal{S}'(\mathbb{R}^{n},E)$, $\mathcal{S}'(\mathbb{H}^{n+1},E)$, and the Lizorkin type space $\mathcal{S}'_{0}(\mathbb{R}^{n},E)$; see \cite{silva} for vector-valued
 distributions. Let $\mathbf{f}$ be in one of these spaces of $E$-valued generalized functions and let $\varphi$
 be in the corresponding space of test functions; the value of $\mathbf{f}$ at $\varphi$ will be denoted by
$
\left\langle \mathbf{f},\varphi\right\rangle=\left\langle \mathbf{f}(t),\varphi(t)\right\rangle\in E.
$
If $f$ is a scalar generalized function and $\mathbf{v}\in E$, we denote by $f\mathbf{v}=\mathbf{v} f$ the $E$-valued
generalized function given
by $\left\langle f(t)\mathbf{v},\varphi(t)\right\rangle=\left\langle f,
\varphi\right\rangle\mathbf{v}$. The Fourier transform of $\mathbf{f}\in\mathcal{S}'(\mathbb{R}^{n},E)$
is defined in the usual way, 
$$\left\langle \hat{\mathbf{f}}(u),\varphi(u)\right\rangle=
\left\langle\mathbf{f}(t),\hat{\varphi}(t)\right\rangle, \ \ \
\varphi\in\mathcal{S}(\mathbb{R}^{n}). $$

Observe that we have a well defined continuous linear projector
from $\mathcal{S}'(\mathbb{R}^{n},E)$ onto $\mathcal{S}'_0(\mathbb{R}^{n},E)$ as
the restriction of $E$-valued tempered distributions to the closed subspace
$\mathcal{S}_0(\mathbb{R}^{n})$. It is clear that this map is surjective. However, it has no continuous right inverse \cite{estrada3}. We
do not want to introduce a notation for this map, so if
$\mathbf{f}\in\mathcal{S}'(\mathbb{R}^{n},E)$, we will keep calling by $\mathbf{f}$ its projection onto $\mathcal{S}'_0(\mathbb{R}^{n},E)$. Note also
that the kernel of this map is the space of polynomials over $\mathbb{R}^{n}$ with coefficients in $E$ ($E$-\emph{valued polynomials}). Therefore, $\mathcal{S}'_{0}(\mathbb{R}^{n},E)$ can be regarded as the quotient space of $\mathcal{S}'(\mathbb{R}^{n},E)$ by the space of $E$-valued polynomials.

If $\mathbf{f}$ is a continuous $E$-valued function of tempered growth on $\mathbb{R}^{n},$ we
make the usual identification with the element $\mathbf{f}\in\mathcal{S}'(\mathbb{R}^{n},E)$, that is,
$$\left\langle \mathbf{f}(t),\varphi\right(t)\rangle:=\int_{\mathbb{R}^{n}}\mathbf{f}(t)\varphi(t)\mathrm{d}t.$$ On the other hand, our convention is different for the space $\mathcal{S}'(\mathbb{H}^{n+1},E)$. Let $\mathbf{K}\in C(\mathbb{H}^{n+1},E)$, we say that it is of \emph{slow growth} on $\mathbb{H}^{n+1}$ if there exist $C>0$ and $k,l\in\mathbb{N}$ such that
\begin{equation*}
\left\|\mathbf{K}(x,y)\right\|\leq C \left(\frac{1}{y}+y\right)^{k}(1+\left|x\right|)^{l},\ \ \ (x,y)\in\mathbb{H}^{n+1};
\end{equation*}
we shall identify $\mathbf{K}\in\mathcal{S}'(\mathbb{H}^{n+1},E)$ by
\begin{equation*}
\left\langle \mathbf{K}(x,y),\Phi(x,y)\right\rangle:=\int^{\infty}_{0}\int_{\mathbb{R}^{n}}\mathbf{K}(x,y)\Phi(x,y)\frac{\mathrm{d}x\mathrm{d}y}{y}, \ \ \
\Phi\in \mathcal{S}(\mathbb{H}^{n+1}).
\end{equation*}
The choice of $y^{-1}\mathrm{d}x\mathrm{d}y$ instead of $\mathrm{d}x\mathrm{d}y$ will be clear in Subsection \ref{waE} below.
\subsection{Weak-asymptotics}
\label{wnwnq} The weak-asymptotics
measure the scaling asymptotic properties of a
distribution by asymptotic comparison with Karamata regularly
varying functions \cite{drozhzhinov-z3, drozhzhinov-z4,
estrada-kanwal2,
pilipovic-stankovic-takaci,vindas1,vindas2,vindas-pilipovic1,vladimirov-d-z1}. As previously mentioned in the introduction, ``quasiasymptotics'' is the usual name employed in the literature for this concept; however, we will use the name weak-asymptotics which better reflects the distributional nature of this notion.

Recall that a measurable real-valued function,
defined and positive on an interval $(0,A]$ (resp. $ [A, \infty)),
$ $A>0$, is called \emph{slowly varying} at the origin (resp. at
infinity) \cite{bingham,karamata,seneta} if
\begin{equation*}
\lim_{\varepsilon\rightarrow0^+}\frac{L(a\varepsilon)}{L(\varepsilon)}=1\ \ \
\left(\mbox{ resp. }
\lim_{\lambda\rightarrow\infty}\frac{L(a\lambda)}{L(\lambda)}=1\right), \ \ \ \mbox{for each }a>0.
\end{equation*}
Observe that slowly varying functions are very convenient objects
to be employed in wavelet analysis since they are asymptotic
invariant under rescaling at small scale (resp. large scale). Familiar functions such as $L(t)=\left|\log t\right|^{\gamma}, \left|\log\left|\log t\right|\right|^{\beta}, \dots$, are slowly varying at both the origin and infinity.

In the next definition $\mathcal{A}(\mathbb{R}^{n})$ is assumed to be a space of functions for which the dilations and translations are continuous operators; consequently, these two operations can be canonically defined on $\mathcal{A}'(\mathbb{R}^{n},E)$. Our interest is in $\mathcal{A}=\mathcal{D},\mathcal{S},\mathcal{S}_{0}$.

\begin{definition}
\label{wnwd2} Let $\mathbf{f}\in\mathcal{A}'(\mathbb{R}^{n},E)$ and let $L$ be slowly varying at the origin (resp. at infinity). We say that:
\begin{itemize}
\item [(i)] $\mathbf{f}$ is weak-asymptotically bounded of degree $\alpha\in\mathbb{R}$ at the point $x_{0}\in\mathbb{R}^{n}$ (resp. at infinity) with respect to $L$ in $\mathcal{A}'(\mathbb{R}^{n},E)$ if

\begin{equation*}
\sup_{0<\varepsilon\leq1}\frac{1}{\varepsilon^{\alpha}L(\varepsilon)}\left\|\left\langle
\mathbf{f}\left(x_0+\varepsilon t\right),\varphi(t)\right\rangle\right\|<\infty, \ \ \ \mbox{for each }\varphi\in\mathcal{A}(\mathbb{R}^{n})
\end{equation*}
$$\left(\mbox{resp. } \sup_{1\leq\lambda}\frac{1}{\lambda^{\alpha}L(\lambda)}\left\|\left\langle\mathbf{f}\left(\lambda t\right)
,\varphi(t)\right\rangle\right\|<\infty\right).
$$
In such a case we write,
\begin{equation*}
\mathbf{f}\left(x_0+\varepsilon t\right)=O\left(\varepsilon^{\alpha}L(\varepsilon)\right)\ \ \ \mbox{as}\  \varepsilon\to0^{+}  \ \ \mbox{in}\ \mathcal{A}'(\mathbb{R}^{n},E)
\end{equation*}
\begin{equation*}
\left(\mbox{resp. }\mathbf{f}\left(\lambda t\right)=O\left(\lambda^{\alpha}L(\lambda)\right) \ \ \ \mbox{as}\ \lambda\to\infty \ \ \mbox{in}\ \mathcal{A}'(\mathbb{R}^{n},E)\:\right).
\end{equation*}
\item [(ii)] $\mathbf{f}$ has weak-asymptotic behavior of degree $\alpha\in\mathbb{R}$ at the point $x_{0}\in\mathbb{R}^{n}$ (resp. at infinity) with respect to $L$ in $\mathcal{A}'(\mathbb{R}^{n},E)$ if there exists $\mathbf{g}\in\mathcal{A}'(\mathbb{R}^{n},E)$ such that for each $\varphi\in\mathcal{A}(\mathbb{R}^{n})$ the following limit holds, with respect to the norm of $E$,
\begin{equation*}
\lim_{\varepsilon\to0^{+}}\frac{1}{\varepsilon^{\alpha}L(\varepsilon)}\left\langle
\mathbf{f}\left(x_0+\varepsilon t\right),\varphi(t)\right\rangle=\left\langle \mathbf{g}(t),\varphi(t)\right\rangle \in E
\end{equation*}
$$\left(\mbox{resp. } \lim_{\lambda\to\infty}\frac{1}{\lambda^{\alpha}L(\lambda)}\left\langle
\mathbf{f}\left(\lambda t\right),\varphi(t)\right\rangle\right).
$$
We write
\begin{equation}
\label{wnweq2.1}
\mathbf{f}\left(x_0+\varepsilon t\right)\sim \varepsilon^{\alpha}L(\varepsilon)\mathbf{g}(t) \ \ \ \mbox{as}\  \varepsilon\to0^{+}  \ \ \mbox{in}\ \mathcal{A}'(\mathbb{R}^{n},E)
\end{equation}
\begin{equation*}
 \left(\mbox{resp. }\mathbf{f}\left(\lambda t\right)\sim\lambda^{\alpha}L(\lambda)\mathbf{g}(t) \ \ \ \mbox{as}\ \lambda\to\infty \ \ \mbox{in}\ \mathcal{A}'(\mathbb{R}^{n},E)\:\right) .
\end{equation*}
\end{itemize}
\end{definition}

We shall also employ the following notation for denoting the weak-asymptotic behavior (\ref{wnweq2.1})
$$
\mathbf{f}\left(x_0+\varepsilon t\right)= \varepsilon^{\alpha}L(\varepsilon)\mathbf{g}(t)+o\left(\varepsilon^{\alpha}L(\varepsilon)\right) \ \ \ \mbox{as}\  \varepsilon\to0^{+}  \ \ \mbox{in}\ \mathcal{A}'(\mathbb{R}^{n},E)
$$
$$
\left(\mbox{resp. }\mathbf{f}\left(\lambda t\right)=\lambda^{\alpha}L(\lambda)\mathbf{g}(t)+o\left(\lambda^{\alpha}L(\lambda)\right) \ \ \ \mbox{as}\ \lambda\to\infty \ \ \mbox{in}\ \mathcal{A}'(\mathbb{R}^{n},E)\:\right),
$$
which has certain advantage when considering (weak-)asymptotic expansions.

It is easy to show \cite{estrada-kanwal2,pilipovic-stankovic-takaci,vladimirov-d-z1} that $\mathbf{g}$ in (\ref{wnweq2.1}) must be homogeneous with degree of homogeneity $\alpha$ as a generalized function in $\mathcal{A}'(\mathbb{R}^{n},E)$, i.e., $\mathbf{g}(a t)=a^{\alpha}\mathbf{g}(t)$, for all $a\in\mathbb{R}_{+}$. We refer to \cite{drozhzhinov-z4} for an excellent presentation of the theory of multidimensional homogeneous distributions; such results are valid for $E$-valued distributions too.

Let $\mathbf{f}\in\mathcal{S}'(\mathbb{R}^{n},E)$ have weak-asymptotic behavior (resp. be weak-asymptotically bounded) in $\mathcal{S}'(\mathbb{R}^{n},E)$, it is trivial to see that $\mathbf{f}$ has the same weak-asymptotic properties when it is seen as an element of $\mathcal{S}'_{0}(\mathbb{R}^{n},E)$; however, the converse is in general not true. The precise relation between weak-asymptotics in $\mathcal{S}'_{0}(\mathbb{R}^{n},E)$ and $\mathcal{S}'(\mathbb{R}^{n},E)$ will be of vital importance for our further investigations, it will be studied in detail in the Appendix A (see Propositions A.1 and A.2).

Let us discuss some examples of weak-asymptotics.

\begin{example}\label{wnwex2.0}\emph{Pointwise weak scaling exponents.}
Meyer \cite{meyer} defined the weak scaling exponent of $f\in\mathcal{S}'(\mathbb{R}^{n})$ at $x_{0}\in\mathbb{R}^{n}$ as the supremum over all $\alpha$ such that
$$
f(x_{0}+\varepsilon t)=O(\varepsilon^{\alpha}) \ \ \ \mbox{as} \ \varepsilon\to\infty \mbox{ in } \mathcal{S}'_{0}(\mathbb{R}).
$$
The weak scaling exponent is connected with the (local) membership of $f$ to $2$-microlocal spaces \cite{meyer}. The motivation behind the introduction of this scaling exponent is to consider the smallest pointwise scaling exponent that is stable under any differentiation and bigger or equal to the usual H\"{o}lder exponent. The wavelet transform is an effective tool in the calculation of the pointwise weak scaling exponent.

The pointwise weak scaling exponent of a function may be rather
big without having the function to be even differentiable at the
point. This phenomenon is due to cancellation caused by rapid
oscillation. For example the function
$t^{\gamma}\sin(t^{-\beta})$, $t>0$, admits a regularization at
the origin which has weak scaling exponent $\infty$ at $x_{0}=0$,
this can be done for all $\gamma\in\mathbb{R}$ and $\beta>0$
\cite{lojasiewicz}. In \cite{saka}, Saka provided useful bounds
for pointwise weak scaling exponents of self-similar functions,
they allow one to relate such a weak exponent with the H\"{o}lder
one. For example, consider the Weierstrass function
$$
f(t)=\sum_{j=1}^{\infty}a^{j}\sin(2^{j}),
$$
where $0<a<1$; it is shown in \cite[p. 1044]{saka} that at any point $x_{0}\in\mathbb{R}$ the pointwise weak scaling and H\"{o}lder exponents of $f$ coincide, and actually, they equal the constant $-\log a/\log2$.

Meyer \cite{meyer} calls a point where the weak scaling and H\"{o}lder exponents of a function coincide (and are finite) a \emph{cusp singularity}; on the other hand, if the weak exponent is strictly bigger than the H\"older one, the point is said to be an \emph{oscillating singularity}. So, the Weierstrass function defined above has cusp singularities everywhere, while the functions $t^{\gamma}\sin(t^{-\beta})$, $\beta>0$, $\gamma\geq0$, have oscillating singularities at the origin.
\end{example}

\begin{example}
\label{wnwex2.1}\emph{\L ojasiewicz point values.} If $\mathcal{A}'(\mathbb{R}^{n},E)=\mathcal{D}'(\mathbb{R}^{n},E)$, $\alpha=0$ and $\mathbf{g}(t)=\mathbf{v}\in E$, a constant $E$-valued distribution, in (ii) of Definition \ref{wnwd2} (as $\varepsilon\to0^{+}$), then we say that the \emph{distributional point value} of $\mathbf{f}$ at $x_0$ is $\mathbf{v}$. We denote this by $\mathbf{f}(x_{0})=\mathbf{v},$ \emph{distributionally}. The notion of point values for distributions is due to \L ojasiewicz \cite{lojasiewicz,lojasiewicz2} (see also \cite{estrada-vindasI,vindas-estrada1,walterw}). It will be very important for us in the future.

\L ojasiewicz concept of point values is an average notion. In fact, it
can be shown \cite{lojasiewicz,lojasiewicz2} that $\mathbf{f}\left(x_{0}\right)  =\mathbf{v},$ distributionally,
if and only if there exists a multi-index $m_{0}\in\mathbb{N}^{n}$
such that for all multi-indices $m\geq m_{0}$ there exists an
$m$-primitive of $\mathbf{f},$ $\mathbf{G}$ with $\mathbf{G}^{(m)}=\mathbf{f},$ that is
a continuous function in a neighborhood of $x_{0}$ and satisfies%
\begin{equation}
\mathbf{G}\left(x\right)  =\frac{\left(  x-x_{0}\right)
^{m}}{m!}\mathbf{v}+o\left(  \left\vert x-x_{0}\right\vert
^{\left\vert m\right\vert }\right),\ \ \ \mbox{as }x
\rightarrow x_{0}. \label{pr.3}%
\end{equation}

In the case of periodic distributions over the real line, it was
shown in \cite{estrada96} that there is a characterization of the Fourier series at
points where the distributional point value exists; indeed, if $f\in\mathcal{S}'(\mathbb{R})$ has Fourier
series $\sum_{n=-\infty}^{\infty}c_{n}e^{int}$ and $x_{0}\in\mathbb{R}$, then
$f(x_{0})=\beta$, distributionally, if and only if there exists $k$ such that
\begin{equation*}
\lim_{x\rightarrow\infty}\sum_{-x\leq n\leq ax}c_{n}e^{inx_{0}}=\beta
\ \ \ (\mathrm{C},k),
\end{equation*}
for each $a>0$, where $(\mathrm{C},k)$ means in the Ces\`{a}ro sense. Remarkably, an analog result is true
for Fourier transforms of arbitrary tempered distributions \cite{vindas-estrada2}.
\end{example}
\begin{example}\label{wnwex2.3}\emph{Moment asymptotic expansions.} Let $\mathbf{f}\in \mathcal{E}'(\mathbb{R}^{n},E)$, a compactly supported $E$-valued distribution. Then $\mathbf{f}$ satisfies the Estrada-Kanwal moment asymptotic expansion \cite{estrada-kanwal90,estrada-kanwal2},
\begin{equation}
\label{wnweq2.2}
\mathbf{f}(\lambda t)\sim \sum_{\left|m\right|=0}^{\infty}\frac{(-1)^{\left|m\right|}}{m!\lambda^{\left|m\right|+n}} \delta^{(m)}(t)\mu_{m}(\mathbf{f})  \ \ \ \mbox{as}\ \lambda\to\infty \ \ \mbox{in}\ \mathcal{S}'(\mathbb{R}^{n},E),
\end{equation}
where $\mu_{m}(\mathbf{f})=\left\langle \mathbf{f}(t),t^{m}\right\rangle\in E$ are the moments of $\mathbf{f}$, in the sense that if $\varphi\in \mathcal{S}(\mathbb{R}^{n})$, then, for each $N\in\mathbb{N}$,
$$
\left\langle \mathbf{f}(\lambda t),\varphi(t)\right\rangle=\sum_{\left|m\right|\leq N} \frac{\varphi^{(m)}(0)}{m!\lambda^{\left|m\right|+n}}\mu_{m}(\mathbf{f})+O\left(\frac{1}{\lambda^{N+n+1}}\right) \ \ \ \mbox{as }\lambda\to\infty.
$$
Consequently, this shows that the weak-asymptotics of
distributions is not a local notion at infinity; in contrast with the case at finite points where the notion is actually local (cf. Subsection \ref{wnwDS}, where this well known \cite{vindas-pilipovic1,zavialov88} fact is reproved). The
moment asymptotic expansion is valid in many other important
distribution spaces \cite{estrada-kanwal2}. Distributions having
an expansion of the type (\ref{wnweq2.2}) are said to be
\emph{distributionally small at infinity}, we shall provide a
wavelet characterization of such distributions in Subsection
\ref{wnwds}. We refer to \cite{estrada-kanwal2} for the numerous
and interesting applications of the moment asymptotic expansions.
\end{example}
\begin{example}Let $f\in\mathcal{S}'(\mathbb{R})$ satisfy the moment
asymptotic expansion at infinity. Denote by $H$ the Heaviside
function, i.e., the characteristic function of $(0,\infty)$. Then,
it is easy to verify that the distributions
$$g_{\pm}(t)=\left|t\right|^{\gamma}f((t-x_{0})^{-\beta})H(\pm
(t-x_0)), \ \ \ t\neq x_{0},\;
\beta>0,\; \gamma\in\mathbb{R},$$ admit \emph{unique} extensions
to the whole $\mathbb{R}$ so that they are supported in
$(-\infty,x_{0}]$ and $[x_{0},\infty)$, respectively, and they
have weak scaling exponent equal to $\infty$ at the point $x_{0}$.
For instance, consider the Riemann ``non-differentiable function''
$$
w(t)=\sum_{n=1}^{\infty}\frac{\sin(\pi tn^{2})}{n^{2}}\: .
$$
Jaffard and Meyer showed \cite[Thm 7.1]{jaffard-m} that it has a
trigonometric chirp at each rational point of the form
$r=(2\nu+1)/(2j+1)$, $\nu,j\in\mathbb{Z}$; this gives in
particular that in a neighborhood of any such a point $w$ can be
written as
$$w(r+t)=\left|t\right|^{\frac{3}{2}}\left(f_{+}\left(t^{-1}\right)H(t)+f_{-}\left(t^{-1}\right)H(-t)\right)+G(t),$$
where $G\in C^{\infty}(\mathbb{R})$ and $f_{\pm}$ are distributionally small at infinity \cite[p. 146]{estrada-kanwal2}. Thus, we conclude that $w$ has pointwise weak scaling exponent equal to $\infty$ at those rational points. It can be also shown that the pointwise scaling H\"{o}lder exponents at those points is equal to $3/2$. Therefore, $w$ has oscillating singularities at the points $r=(2\nu+1)/(2j+1)$, $\nu,j\in\mathbb{Z}$.

We shall analyze in Subsection \ref{wnwWRfunction} the pointwise properties of the family of Riemann type distributions
$$R_{\beta}(t)=\sum_{n=1}^{\infty}\frac{e^{i\pi tn^{2}}}{n^{2\beta}}\: , \ \ \ \beta\in\mathbb{C},$$
and obtain their complete pointwise weak-asymptotic expansion at
\emph{every rational} point.
\end{example}
\begin{example}
\label{wnwexjb}\emph{Distributional jump behavior.} An useful extension of the notion of \L ojasiewicz point
values is that of pointwise jump behavior of distributions \cite{estrada-vindas,vindas-estrada1,vindas-estrada4}.
 Let $f\in\mathcal{S}'(\mathbb{R})$. Recall $H$ denotes the Heaviside function. We say $f$ has distributional jump
 behavior at $x_{0}$ if
\begin{equation}
\label{jlaeq1}
f\left(x_0+\varepsilon t\right)=\beta_{-}H(-t)+\beta_{+}H(t)+o(1) \ \ \ \mbox{as }\varepsilon\rightarrow0^{+} \mbox{ in }\mathcal{D}'(\mathbb{R}).
\end{equation}
for constants $\beta_{\pm}:=f(x_0^{\pm})$, its distributional lateral point values. Naturally, this notion can be
considered for $E$-valued distributions as well.

Suppose that $f$ is either a periodic distribution or an element
of $L^{p}(\mathbb{R})$, $p\in(1,\infty)$. Under this assumption we
can speak about its Hilbert transform, i.e., the distribution
$$\tilde{f}= \displaystyle\frac{1}{\pi}\mathrm{p.v.}\frac{1}{t}\ast f.$$ It can be shown
\cite{vindas-estrada4} that (\ref{jlaeq1}) implies the following
weak-asymptotic behavior for the Hilbert transform
\begin{equation}
\label{jumpconj}
\tilde{f}(x_0+\varepsilon t)\sim-\frac{1}{\pi}[f]_{x=x_0}\log(1/\varepsilon) \ \ \ \mbox{as }\varepsilon\rightarrow0^{+} \mbox{ in }\mathcal{S}'(\mathbb{R}),
\end{equation}
where $[f]_{x=x_0}=\beta_{+}-\beta_{-}$, the pointwise jump of $f$
at $x_{0}$. If $f$ is periodic with Fourier series
$$f(t)=\frac{a_0}{2}+\sum^{\infty}_{n=0}\left(a_n\cos nt+b_n\sin
nt\right),$$ then the weak-asymptotic behavior (\ref{jumpconj})
can be used to extend Zygmund \cite{zygmund} and M\'{o}ricz
\cite{moricz2} versions of the classical Luk\'{a}cs theorem, that
is, we can recover the jump from the Abel-Poisson means of the
conjugate series \cite{vindas-estrada4},
\begin{equation}
\label{jumpconj2}
\lim_{r\rightarrow 1^{-}}\frac{1}{\log(1-r)}\sum^{\infty}_{n=1}\left(a_n\sin nx_0-b_n\cos nx_0\right)r^n=\frac{1}{\pi}[f]_{x=x_0}\:.
\end{equation}
When $f\in L^{p}(\mathbb{R})$, $1<p<\infty$, the jump can be determined by its conjugate Poisson integral,
\begin{equation}
\label{jumpconj3}
\lim_{y\rightarrow 0^{+}}\frac{1}{\log y}\int_{-\infty}^{\infty}\frac{x_{0}-t}{(x_0-t)^{2}+y^{2}}\:f(t)\mathrm{d}t=[f]_{x=x_0}\:.
\end{equation}
Observe that we may single out a canonical Hilbert transform even in less restrictive situations, for example as in \cite{estrada-kanwal1}, and an appropriate version of (\ref{jumpconj2}) and (\ref{jumpconj3}) still holds  in those cases.
\end{example}

\newpage

\section{Wavelet and Non-Wavelet Transforms of $E$-valued Distributions}
\label{wnwtd}
We shall present in  Subsection \ref{wnwt} some basic properties of wavelet and non-wavelet transforms of $E$-valued tempered distributions. We then discuss examples in Subsection \ref{wnwex}. Section \ref{waE} deals with wavelet analysis on the space $\mathcal{S}'_{0} (\mathbb{R}^{n},E)$. For test functions we set $\check{\varphi}(\cdot)=\varphi(-\:\cdot\:)$ and $\varphi_{y}(\cdot)=y^{-n}\varphi( \cdot/y)$. The moments of $\varphi$ are denoted by $$\mu_{m}(\varphi)=\int_{\mathbb{R}^{n}}t^{m}\varphi(t)\mathrm{d}t, \ \ \ m\in\mathbb{N}^{n}.$$

\subsection{Regularizing Transforms}
\label{wnwt}
Let $\mathbf{f}\in\mathcal{S}'(\mathbb{R}^{n},E)$. We set, as in the introduction,

\begin{equation}
\label{wnwneq1}
M^{\mathbf{f}}_{\varphi}(x,y):=(\mathbf{f}\ast\varphi_{y})(x)\in E, \ \ \ (x,y)\in\mathbb{H}^{n+1},
\end{equation}
the \emph{regularizing transform} of $\mathbf{f}$ with respect to the test function $\varphi\in\mathcal{S}(\mathbb{R}^{n})$. Notice that $M^{\mathbf{f}}_{\varphi}\in C^{\infty}(\mathbb{H}^{n+1},E)$.

We shall distinguish two cases of the regularizing transform.

If $\mu_{0}(\varphi)\neq 0$, we say that (\ref{wnwneq1}) is a \emph{non-wavelet} type transform. Furthermore, let $\phi\in\mathcal{S}(\mathbb{R}^{n})$ be such that $\mu_{0}(\phi)=\int_{\mathbb{R}^{n}}\phi(t)\mathrm{d}t=1$. \emph{The $\phi-$transform} of $\mathbf{f}$ is
\begin{equation}
\label{wnwneq2}
F_{\phi}\mathbf{f}(x,y):=M^{\mathbf{f}}_{\check{\phi}}(x,y)=\left\langle \mathbf{f}(x+yt),\phi(t)\right\rangle\in E,  \ \ \ (x,y)\in\mathbb{H}^{n+1}.
\end{equation}
It should be observed that the $\phi-$transform essentially encloses all non-wavelet cases of (\ref{wnwneq1}) after a normalization. The terminology of $\phi-$transforms is from \cite{estrada-vindasI,vindas-estrada1,vindas-estrada5,vindas-estrada6}.

The second case of (\ref{wnwneq1}) is when $\mu_{0}(\varphi)=0$, namely, the wavelet one. Let $\psi\in\mathcal{S}(\mathbb{R}^{n})$ satisfy $\mu_{0}(\psi)=\int_{\mathbb{R}^{n}}\psi(t)\mathrm{d}t=0$, we then call $\psi$ a wavelet. \emph{The wavelet transform} of $\mathbf{f}$ with respect to $\psi$ is defined by
\begin{equation}
\label{wnwneq3}
\mathcal{W}_{\psi}\mathbf{f}(x,y):=M^{\mathbf{f}}_{\check{\bar{\psi}}} (x,y)=\left\langle \mathbf{f}(x+yt),\bar{\psi}(t)\right\rangle\in E, \ \ \ (x,y)\in\mathbb{H}^{n+1} .
\end{equation}
In the subsequent sections we shall restrict our attention to those wavelets which posses nice reconstruction properties (cf. Subsection \ref{waE} below).

\begin{definition}
\label{wnwd1}
We say that the test function $\varphi\in\mathcal{S}(\mathbb{R}^{n})$ is non-degenerate if for any $\omega\in\mathbb{S}^{n-1}$ the function of one variable $R_{\omega}(r)=\hat{\varphi}(r\omega)\in C^{\infty}[0,\infty)$ is not identically zero, that is,
$$
\operatorname*{supp} R_{\omega}\neq\emptyset , \ \ \ \text{for each }\omega\in\mathbb{S}^{n-1} .
$$
We say that $\psi\in\mathcal{S}(\mathbb{R}^{n})$ is a non-degenerate wavelet if it is a non-degenerate test function and additionally $\mu_{0}(\psi)=0$.
\end{definition}
Obviously, test functions for which $\mu_{0}(\varphi)\neq0$ are always non-degenerate. We will discuss particular important cases of non-degenerate wavelets in Example \ref{wnwex3.3}.

There is a remarkable difference between the wavelet and non-wavelet transforms. Indeed, the following proposition shows such a difference. We give a quick proof of it by using \L ojasiewicz point values (cf. Example \ref{wnwex2.1}); the argument is essentially the same as in \cite{vindas-estrada5}.

\begin{proposition}
\label{wnwp1} Let $\mathbf{f}\in\mathcal{S}'(\mathbb{R}^{n},E)$ and let $\varphi\in\mathcal{S}(\mathbb{R}^n)$, then
\begin{equation}
\label{wnwneq4}
\lim_{y\to0^{+}}M^{\mathbf{f}}_{\varphi}(\cdot,y)=\mu_{0}(\varphi) \mathbf{f}, \ \ \ \text{in}\ \mathcal{S}'(\mathbb{R}^{n},E).
\end{equation}
\end{proposition}
\begin{proof} Since $\mathcal{S}(\mathbb{R}^{n})$ is a Montel space \cite{treves}, the Banach-Steinhaus theorem implies that it is enough to show the convergence of (\ref{wnwneq4}) for the topology of pointwise convergence \cite{treves}. Let $\rho\in\mathcal{S}(\mathbb{R}^{n})$, we have
\begin{equation*}
\left\langle M^{\mathbf{f}}_{\varphi}\left(x,y\right)  ,\rho\left( x\right)
\right\rangle =\left\langle \mathbf{h}\left(  yt\right)  ,\varphi\left(t\right)  \right\rangle, \ \ \ 0<y<1 ,
\end{equation*}
where
$$\mathbf{h}\left( u\right)  =\left\langle \mathbf{f}\left(  x\right)
,\rho\left( x+u\right)  \right\rangle,\ \ \ u\in\mathbb{R}^{n},$$ is a
smooth $E$-valued function of slow growth. The \L ojasewicz point
value $\mathbf{h}\left( 0\right)$ exists and equals the ordinary
value and thus
\begin{equation*}
\lim_{y\rightarrow0^{+}}\left\langle \mathbf{h}\left(  yt\right)
,\varphi\left( t\right)  \right\rangle =\mathbf{h}\left( 0
\right)\int_{\mathbb{R}^{n}}\varphi(t)\mathrm{d}t  =\mu_{0}(\varphi)\left\langle \mathbf{f}\left(x\right)  ,\rho\left(  x
\right)  \right\rangle,
\end{equation*}
as required.
\end{proof}

Proposition \ref{wnwp1} therefore tells us that in the non-wavelet case we can always recover $\mathbf{f}$ as the distributional boundary values of its regularizing transform, while this is not any longer true for the wavelet transform.

Since for the $\phi-$transform
$$
\lim_{y\to0^{+}}F_{\phi}f(\cdot,y)=f, \ \ \ \text{in}\ \mathcal{S}'(\mathbb{R}^{n}),
$$
the Hahn-Banach theorem implies the ensuing important corollary.
\begin{corollary}\label{wnwc1} Let $\varphi\in\mathcal{S}(\mathbb{R}^{n})$ and let $\sigma>0$. Then, the linear span
of the set of the dilates (at scale less than $\sigma$) and
translates of $\varphi$, that is, $$\left\{\varphi((\:\cdot-x)/y):
(x,y)\in\mathbb{R}^{n}\times(0,\sigma)\right\},$$ is dense in
$\mathcal{S}(\mathbb{R}^{n})$ if and only if
$\mu_{0}(\varphi)\neq0$.
\end{corollary}

A property shared by the wavelet and non-wavelet transforms is the following one: They map continuously tempered distributions to smooth functions of slow growth on $\mathbb{H}^{n+1}$.
\begin{proposition}
\label{wnwp2} Let $\mathbf{f}\in\mathcal{S}'(\mathbb{R}^{n},E)$
and let $\varphi\in\mathcal{S}(\mathbb{R}^{n})$. Then,
$M^{\mathbf{f}}_{\varphi}\in C^{\infty}(\mathbb{H}^{n+1},E)$ is a
function of slow growth on $\mathbb{H}^{n+1}$. In addition, the
linear map  $$\mathbf{f}\in\mathcal{S}'(\mathbb{R}^{n},E)\mapsto
M^{\mathbf{f}}_{\varphi}\in\mathcal{S}'(\mathbb{H}^{n+1},E)$$ is
continuous for the topologies of uniform convergence over bounded
sets. Furthermore, if
$\mathfrak{B}\subset\mathcal{S}'(\mathbb{R}^{n},E)$ is bounded for
the topology of pointwise convergence, then there exist $k,l$ and
$C>0$ such that
\begin{equation}
\label{wnwneq5}
\left\|M_{\varphi}^\mathbf{f}(x,y)\right\|\leq C \left( \frac{1}{y}+y\right)^{k}\left(1+\left|x\right|\right)^{l} , \ \ \ \text{for all } \mathbf{f}\in\mathfrak{B}.
\end{equation}
\end{proposition}
\begin{proof}
Since $\mathcal{S}'(\mathbb{R}^{n},E)$ is the inductive limit of a
strictly increasing sequence of Banach spaces
\cite{treves,vladimirovbook}, it is bornological. Therefore, we
should show that this map takes bounded sets to bounded ones. Let
$\mathfrak{B}\subset\mathcal{S}'(\mathbb{R}^{n},E)$ be a bounded
set. The Banach-Steinhaus theorem implies that $\mathfrak{B}$ is
bounded for the topology of bounded convergence if and only if it
is bounded for the topology of pointwise convergence; it is also
an equicontinuous set, from where we obtain the existence of
$k_{1}\in\mathbb{N}$ and $C_{1}>0$ such that for all $\rho\in \mathcal{S}(\mathbb{R}^{n})$ and $\mathbf{f}\in\mathfrak{B}$,

$$
\left\|\left\langle \mathbf{f},\rho\right\rangle\right\|\leq C_{1} \sup_{t\in\mathbb{R}^{n}, \left|m\right|\leq k_{1}} \left(1+\left|t\right|\right)^{k_1}\left|\rho^{(m)}(t)\right| .
$$
Consequently,
\begin{align*}
\left\|M^{\mathbf{f}}_{\varphi}(x,y)\right\| & =\frac{1}{y^{n}}\left\|\left\langle \mathbf{f}(t),\varphi\left(\frac{x-t}{y}\right) \right\rangle\right\|
\\
&
\leq C_{1}\left(\frac{1}{y}+y\right)^{n+k_1} \sup_{u\in\mathbb{R}^{n}, \left|m\right|\leq k_1} \left(1+\left|x\right|+y\left|u\right|\right)^{k_1}\left|\varphi^{(m)}\left(u\right)\right|
\\
&
\leq
C\left(\frac{1}{y}+y\right)^{n+2k_{1}}(1+\left|x\right|)^{k_{1}}  , \ \ \mbox{ for all } \mathbf{f}\in\mathfrak{B},
\end{align*}
where $$C=C_{1}\sup_{u\in\mathbb{R}^{n}, \left|m\right|\leq k_1}
\left(1+\left|u\right|\right)^{k_1}\left|\varphi^{(m)}\left(u\right)\right|.$$
So, we obtain (\ref{wnwneq5}) with $k=n+2k_{1}$ and $l=k_{1}$. If
$\mathfrak{C}\subset\mathcal{S}(\mathbb{H}^{n+1})$ is a bounded
set of test functions, we have
\begin{align*}
\left\|\left\langle M^{\mathbf{f}}_{\varphi}(x,y),\Phi(x,y)\right\rangle\right\|&=\left\|\int^{\infty}_{0}\int_{\mathbb{R}^{n}}M^\mathbf{f}_{\varphi}(x,y)\Phi(x,y)\frac{\mathrm{d}x\mathrm{d}y}{y}\right\|
\\
&
\leq C \int^{\infty}_{0}\int_{\mathbb{R}^{n}}\left(\frac{1}{y}+y\right)^{k}(1+\left|x\right|)^{l}\left|\Phi(x,y)\right|\frac{\mathrm{d}x\mathrm{d}y}{y}\, ,
\end{align*}
which stays bounded as $\mathbf{f}\in\mathfrak{B}$ and $\Phi\in\mathfrak{C}$. Therefore, the set $$\left\{M^{\mathbf{f}}_{\varphi}: \; \mathbf{f}\in\mathfrak{B}\right\}\subset\mathcal{S}'(\mathbb{H}^{n+1},E)$$
is bounded; hence the map is continuous.
\end{proof}
The regularizing transforms enjoy excellent localization properties as shown by the following simple proposition. This holds for both the wavelet and non-wavelet cases.

\begin{proposition} Let $\mathbf{f}\in\mathcal{S}'(\mathbb{R}^{n},E)$ and let $\varphi\in\mathcal{S}(\mathbb{R}^{n})$. Suppose that $K\subset \mathbb{R}^{n}\setminus \operatorname*{supp} \mathbf{f}$ is a compact set. Then, for any positive integer $k\in\mathbb{N}$ there exists $C=C_{k}$ such that
\begin{equation}
\label{wnweqLoc}
\sup_{x\in K}\left\|M^{\mathbf{f}}_{\varphi}(x,y)\right\|\leq C y^{k}, \  \ \ \mbox{for all }0<y<1.
\end{equation}
\end{proposition}
\begin{proof} Define the $C(K,E)$-valued tempered distribution whose evaluation at $\rho\in\mathcal{S}(\mathbb{R}_{t}^{n})$ is given by
$$\left\langle \mathbf{G}(t),\rho(t)\right\rangle(\xi)=(\mathbf{f}\ast \rho)(\xi), \ \ \ \xi\in K.$$
Clearly, $\mathbf{G}\in\mathcal{S}'(\mathbb{R}_{t}^{n},C(K_{\xi},E))$. Then, since $K\subset \mathbb{R}^{n}\setminus \operatorname*{supp} \mathbf{f}$, we have that for each $\rho\in\mathcal{D}(\mathbb{R}^{n})$,
$$
\left\langle \mathbf{G}(\varepsilon t),\rho (t)\right\rangle=0,
$$
for sufficiently small $\varepsilon>0$. In particular, we obtain that, for a fixed $k\in\mathbb{N}$,
\begin{equation}
\label{wnweqpb1}
\mathbf{G}(\varepsilon t)=O(\varepsilon^{k}) \ \ \ \mbox{as }\varepsilon \to 0^{+}\ \mbox{ in }\mathcal{D}'(\mathbb{R}_{t}^{n},C(K,E)).
\end{equation}
Now, it is well know that the weak-asymptotic boundedness (\ref{wnweqpb1}) remains valid in the space $\mathcal{S}'(\mathbb{R}_{t}^{n},C(K,E))$. This fact is shown in \cite{vindas-pilipovic1,zavialov88}, but actually, we shall give a new proof in Corollary \ref{wnwDSc1} of Subsection \ref{wnwDS}. Thus, we have the right to evaluate the relation (\ref{wnweqpb1}) at $\varphi\in\mathcal{S}(\mathbb{R}^{n})$, which immediately yields
\begin{align*}
\sup_{x\in K} \left\|M^{\mathbf{f}}_{\varphi}(x,y)\right\|&= \left\|\left\langle \mathbf{G}(yt),\varphi(t)\right\rangle\right\|_{C(K,E)}
\\
&
\leq Cy^{k},
\end{align*}
for some $C>0$, as claimed.
\end{proof}
\subsection{Examples of Regularizing Transforms}
\label{wnwex}
Let us discuss some examples of regularizing transforms. We shall return to these examples in Section \ref{wnwap} where we will provide applications of the Tauberian theorems from Section \ref{wnwtt} and Section \ref{wnwce}.

Our first example is one-dimensional and shows how the $\phi-$transform is related to summability of numerical series.

\begin{example}
\label{wnwex3.1}\emph{The $\phi-$transform and summability of
series.} Let $\left\{ c_{n}\right\}^{\infty}_{n=0}$ be a sequence
of complex numbers and let $\rho\in\mathcal{S}(\mathbb{R})$ with
$\rho(0)=1$.  We say that the (possible divergent) series
$\sum_{n=0}^{\infty}c_{n}$ is $(\rho)$ summable to $\beta$
if
\begin{equation}
\label{wnweq3.1}
\sum_{n=0}^{\infty}c_{n} \rho(yn)\ \  \mbox{converges for all } y>0,
\end{equation}
and
\begin{equation}
\label{wnweq3.2}
\lim_{y\to0^{+}}\sum_{n=0}^{\infty}c_{n}\rho(yn) =\beta.
\end{equation}
One readily verifies that this summability method is regular \cite{hardy}, in the sense that it sums
convergent series to their actual values of convergence. Furthermore, different choices of the kernel $\rho$ lead to
 many familiar methods of summability. For example, if $\rho(u)=e^{-u}$ for $u>0$, one then recovers the well known
 Abel method \cite{hardy,korevaarbook}, in such a case one writes for Abel summable series
\begin{equation*}
\sum_{n=0}^{\infty}c_{n}=\beta \ \ \ (\mathrm{A}).
\end{equation*}
Another instance is provided by $\rho(u)=u/(e^{u}-1)$, $u>0$, the kernel of Lambert summability which is so important in number theory \cite{korevaarbook,wiener}.

Assume further that $\left\{c_{n}\right\}_{n=0}^{\infty}$ is of slow growth, i.e., there is $k\in\mathbb{N}$ such that $c_{n}=O\left(n^{k}\right)$. Obviously, (\ref{wnweq3.1}) is always fulfilled under this assumption. Define $f(t)=\sum_{n=0}^{\infty}c_{n}e^{itn}$, a periodic distribution over the real line. Moreover, set $\phi=(1/2\pi)\hat{\rho}$; thus, the $\phi-$transform of $f$ is precisely
$$
F_{\phi}f(x,y)=\frac{1}{2\pi}\left\langle e^{ixu}\hat{f}(u),\rho(yu)\right\rangle=\sum_{n=0}^{\infty}c_{n}e^{ixn} \rho(yn).
$$
Consequently, (\ref{wnweq3.2}) becomes equivalent to a statement on the (radial) boundary behavior of the $\phi-$transform at the origin, namely,
\begin{equation*}
\lim_{y\to0^{+}}F_{\phi}f(0,y)=\beta.
\end{equation*}
We shall use in Subsection \ref{wnwLT} these ideas to produce a new proof of Littlewood's Tauberian theorem for power series (Example \ref{wnwapexl}).
\end{example}
\begin{example}\label{wnwex3.2}
\emph{Embedding of distributions into generalized function  algebras.} The second
important example of $\phi-$transforms points out its relation with the theory of algebras of generalized functions \cite{colombeau,ober 001}.
If $\phi\in\mathcal{S}(\mathbb{R}^{n})$ is a \emph{mollifier} with all higher order vanishing moments, i.e., a test function such that
\begin{equation}
\label{wnweq3.3}
\mu_{0}(\phi)=1 \ \ \mbox{   and   }\ \ \mu_{m}(\phi)=0, \ \ \ \mbox{for all }\left|m\right|\geq1,
\end{equation}
then, for scalar distributions,
the $\phi-$transform is nothing but the standard embedding of $f\in\mathcal{S}'(\mathbb{R}^{n})$ into
the special Colombeau algebra $\mathcal{G}(\mathbb{R}^{n})$ of generalized functions (cf. Subsection \ref{wnwCA}), namely, the net
$$
f_{\varepsilon}(x)=F_{\phi}f(x,\varepsilon), \ \ \ 0<\varepsilon< 1, \ x\in\mathbb{R}^{n},
$$
which determines the class $[f_{\varepsilon}]\in\mathcal{G}(\mathbb{R}^{n})$. Likewise, the
$\phi-$transform also induces the embedding of $f\in\mathcal{S}'(\mathbb{R}^{n})$ into
the algebra $\mathcal{G}_{\tau}(\mathbb{R}^{n})$ of tempered generalized functions \cite{colombeau,ober 001}.
We will use this interpretation of the $\phi-$transform in
Subsection \ref{wnwCA} to give applications to
regularity theory within the framework of  algebras of generalized functions.
\end{example}

We now give examples of non-degenerate wavelets.
\begin{example}\label{wnwex3.3} \emph{Drozhzhinov-Zavialov wavelets.} We say that a polynomial $P$ is non-degenerate (at the origin) if
for each $\omega\in\mathbb{S}^{n-1}$ one has that
$$
P(r\omega)\not\equiv 0, \ \ \ r\in\mathbb{R}_{+}.$$
Drozhzhinov and Zavialov have considered the class of wavelets $\psi\in\mathcal{S}(\mathbb{R}^{n})$, $\mu_{0}(\psi)=0$,
for which there exists $N\in\mathbb{N}$ such that
$$T_{\hat{\psi}}^{N}(u)= \sum_{\left|m\right|\leq N}\frac{\hat{\psi}^{(m)}(0)u^{m}}{m!},$$ the Taylor polynomial of order $N$ at the origin, is non-degenerate; these wavelets were used in \cite{drozhzhinov-z3} to obtain Tauberian theorems for distributions. It should be noticed that this type of wavelets are included in Definition \ref{wnwd1}; naturally, Definition \ref{wnwd1} gives much more wavelets. For instance, any non-degenerate wavelet from $\mathcal{S}_{0}(\mathbb{R}^{n})$ obviously fails to be of this kind. An explicit example of a non-degenerate wavelet $\psi\in\mathcal{S}_{0}(\mathbb{R}^{n})$ is given in the Fourier side by
$$\hat{\psi}(u)=e^{-\left|u\right|-(1/\left|u\right|)},\ \ \ u \in
\mathbb{R}^{n} .$$ Furthermore, if
$\psi_{1}\in\mathcal{S}(\mathbb{R}^{n})$ satisfies
$$\hat{\psi}_{1}(u)=e^{-\left|u\right|-(1/\left|u\right|)}+u_{1}^{2},\ \ \ \mbox{for } \left|u\right|<1,$$ where $u=(u_{1},u_{2},\dots,u_{n})$,
then
$\psi_{1}\in\mathcal{S}(\mathbb{R}^{n})\setminus\mathcal{S}_{0}(\mathbb{R}^{n})$
is a non-degenerate wavelet but all its Taylor polynomials vanish
on the axis $u_{1}=0$.

Let $\phi\in\mathcal{S}(\mathbb{R}^{n})$ be a mollifier that satisfies (\ref{wnweq3.3}) (cf. Example \ref{wnwex3.2})
and let $P$ be a non-degenerate polynomial of degree $k$, then $\psi=P(-i\partial/\partial t)\phi$ is a wavelet of the type considered by Drozhzhinov and Zavialov;
indeed, $T^{k}_{\hat{\psi}}(u)=P(u)$. Wavelets of the form $\psi=\Delta^  d \phi$ were used in \cite{horman} to study H\"{o}lder-Zygmund regularity in algebras of generalized functions.
\end{example}

\begin{example}
\label{wnwex3.4}
\emph{The $\phi-$transform as solution to Cauchy problems}. When the test function
is of certain special form, the $\phi-$transform can become the solution to a PDE. We discuss a particular case
in this example. Let the set $\Gamma\subseteq\mathbb{R}^{n}$ be a closed convex cone with vertex at the origin.
In particular, we may have $\Gamma=\mathbb{R}^{n}$. Let $P$ be a homogeneous polynomial of degree $d$ such that
$$\Re e\:P(iu)<0 \ \ \  \mbox{ for all } \;u\in\Gamma\setminus\left\{0\right\}.$$ 
We denote \cite{vladimirovbook,vladimirov-d-z1} by $\mathcal{S}'_{\Gamma}\subseteq\mathcal{S}'(\mathbb{R}^{n})$ the subspace of distributions supported by $\Gamma$.

Consider the Cauchy problem
\begin{equation}\label{wnweq3.5}
\frac{\partial}{\partial t}U(x,t)=P\left(\frac{\partial}{\partial x}\right)U(x,t), \ \ \ \lim_{t\to0^{+}}U(x,t)=f(x)\ \ \ \mbox{in }\mathcal{S}'(\mathbb{R}^{n}_{x}),
\end{equation}
$$\operatorname*{supp}\hat{f}\subseteq\Gamma, \ \ \ (x,t)\in \mathbb{H}^{n+1},$$
within the class of functions of slow growth over $\mathbb{H}^{n+1}$, that is,
$$
\sup_{(x,t)\in \mathbb{H}^{n+1}} \left|U(x,t)\right|\left(t+\frac{1}{t}\right)^{-k_1}\left(1+\left|x\right|\right)^{-k_{1}}<\infty, \ \ \ \mbox{for some }k_{1},k_{2}\in\mathbb{N}.
$$
One readily verifies that (\ref{wnweq3.5}) has a unique solution. Indeed,
$$
U(x,t)=\frac{1}{(2\pi)^{n}}\left\langle \hat{f}(u), e^{ix\cdot u}e^{tP(iu)}\right\rangle=\frac{1}{(2\pi)^{n}}\left\langle \hat{f}(u), e^{ix\cdot u}e^{P\left(it^{1/d}u\right)}\right\rangle
$$
is the sought solution. We can find \cite{vladimirov-d-z1} a test
function $\eta\in\mathcal{S}(\mathbb{R}^{n})$ with the property
$$\eta(u)=e^{P(iu)},\; u\in\Gamma.$$
Setting $\phi=(2\pi)^{-n}\hat{\eta}$, we express $U$ as a
$\phi-$transform,
\begin{equation}
\label{wnweq3.8}U(x,t)=\left\langle f(\xi),\frac{1}{t^{n/d}}\phi\left(\frac{\xi-x}{t^{1/d}}\right)\right\rangle=F_{\phi}f(x,y), \ \ \ \mbox{with } y=t^{1/d}.
\end{equation}
If $d=2k$ is a positive even integer and $P(\xi)=(-1)^{k-1}\left|\xi\right|^{d}$, then we may take $\Gamma=\mathbb{R}^{n}$, the differential operator becomes $P(\partial/\partial x)=(-1)^{k-1}\Delta^{d}$, and $\phi$ is the Fourier inverse transform of $\eta(u)=e^{-\left|u\right|^{d}}$. In particular, when $d=2$, (\ref{wnweq3.5}) is the Cauchy problem for the heat equation and $\phi(\xi)=(2\sqrt{\pi})^{-n}e^{-\xi^{2}/4}$.

We will study in Subsection \ref{wnwPDE} sufficient conditions for the asymptotic stabilization in time of the solution $U$ to (\ref{wnweq3.5}).
\end{example}
\begin{example}\label{wnwex3.5}
\emph{Laplace transforms as  $\phi-$transforms.} Let $\Gamma$ be a
closed convex acute cone \cite{vladimirovbook,vladimirov-d-z1}
with vertex at the origin. Its conjugate cone is denoted by
$\Gamma^{\ast}$. The definition of an acute cone tells us that
$\Gamma^{\ast}$ has non-empty interior, set
$$C_{\Gamma}=\operatorname*{int} \Gamma^{\ast}\; \mbox{ and }
\;T^{C_{\Gamma}}=\mathbb{R}^{n}+iC_{\Gamma}.$$ We denote by
$\mathcal{S}'_{\Gamma}(E)$ the subspace of $E$-valued tempered
distributions supported by $\Gamma$. Given
$\mathbf{h}\in\mathcal{S}'_{\Gamma}(E)$, its \emph{Laplace
transform} \cite{vladimirovbook} is
$$
\mathcal{L}\left\{\mathbf{h};z\right\}=\left\langle \mathbf{h}(u),e^{iz\cdot u}\right\rangle, \ \ \ z\in T^{C_{\Gamma}};
$$
it is a holomorphic $E$-valued function on the tube domain
$T^{C_{\Gamma}}$. Fix $\omega\in C_{\Gamma}$. We may write
$\mathcal{L}\left\{\mathbf{h};x+i\sigma\omega\right\}$,
$x\in\mathbb{R}^{n}$, $\sigma>0$, as a $\phi-$transform. In fact,
choose $\eta_{\omega}\in\mathcal{S}(\mathbb{R}^{n})$ such that
$$\eta_{\omega}(u)=e^{- \omega\cdot u}, \ \ \  u\in\Gamma. $$
Then, with
$\phi_{\omega}=(2\pi)^{-n}\hat{\eta}_{\omega}$ and
$\hat{\mathbf{f}}=(2\pi)^{n}\mathbf{h}$,
\begin{equation}
\label{wnweq3.9}
\mathcal{L}\left\{\mathbf{h};x+i\sigma\omega\right\}=F_{\phi_{\omega}}\mathbf{f}(x,\sigma).
\end{equation}
Notice that this is a particular case of Example \ref{wnwex3.4} with $P_{\omega}(\xi)=i \omega\cdot\xi$.
\end{example}

\subsection{Wavelet Analysis on $\mathcal{S}'_{0}(\mathbb{R}^{n},E)$}
\label{waE}

In this subsection we indicate how to extend the scalar distribution wavelet analysis given in \cite{holschneider} to $E$-valued generalized functions. We complement the theory with some new results.

Although Proposition \ref{wnwp1} makes impossible to recover an $E$-valued tempered distribution as the boundary value of its wavelet transform, the non-degenerate wavelets from $\mathcal{S}_{0}(\mathbb{R}^{n})$ enjoy excellent reconstruction properties as long as we are interested in the projection of the tempered distribution onto $\mathcal{S}'_{0}(\mathbb{R}^{n},E)$. Observe that if the wavelet belongs to $\mathcal{S}_{0}(\mathbb{R}^{n})$, the wavelet transform with respect to this wavelet is continuous $\mathcal{S}'(\mathbb{R}^{n},E)\mapsto\mathcal{S}'(\mathbb{H}^{n+1},E)$, as can be inferred from Proposition \ref{wnwp2}; however it is not injective, since it maps every $E$-valued polynomial to $\mathbf{0}$, as follows from the moment vanishing properties of the wavelet. This fact makes necessary to work on $\mathcal{S}'_{0}(\mathbb{R}^{n},E)$ if one wishes to have reconstruction of distributions from their wavelet transforms.

Let $\psi\in\mathcal{S}_{0}(\mathbb{R}^{n})$. We have that \cite[Thm 19.0.1]{holschneider} $\mathcal{W}_{\psi}:\mathcal{S}_{0}(\mathbb{R}^{n})\mapsto\mathcal{S}(\mathbb{H}^{n+1})$ is a continuous linear map. We are interested in those wavelets for which $\mathcal{W}_{\psi}$ admits a left inverse. For wavelet-based reconstruction, we shall use the wavelet
synthesis operator \cite{holschneider}. Given $\Phi\in\mathcal{S}(\mathbb{H}^{n+1})$, we define the \textit{wavelet
synthesis operator} with respect to the wavelet $\psi$ as
\begin{equation}
\label{wnwneq6}
\mathcal{M}_{\psi}\Phi(t)=\int^{\infty}_{0}\int_{\mathbb{R}^{n}}\Phi(x,y)\frac{1}{y^{n}}\psi\left(\frac{t-x}{y}\right)\frac{\mathrm{d}x\mathrm{d}y}{y}\: ,
\ \ \ t \in \mathbb{R}^{n}.
\end{equation}
One can show that $\mathcal{M}_{\psi}:\mathcal{S}(\mathbb{H}^{n+1})\rightarrow
\mathcal{S}_0(\mathbb{R}^{n})$ is continuous \cite[p. 74]{holschneider}.

We shall say that the wavelet $\psi\in\mathcal{S}_{0}(\mathbb{R}^{n})$ admits a \textit{reconstruction wavelet} if there exists $\eta\in\mathcal{S}_{0}(\mathbb{R}^{n})$ such that
\begin{equation}
\label{wnwneq7}
c_{\psi,\eta}(\omega)=\int^{\infty}_{0}\overline{\hat{\psi}}(r\omega)\hat{\eta}(r\omega)\frac{\mathrm{d}r}{r}\: , \ \ \ \omega\in\mathbb{S}^{n-1} ,
\end{equation}
is independent of the direction $\omega$; in such a case we set $c_{\psi,\eta}:=c_{\psi,\eta}(\omega).$
The wavelet $\eta$ is
called a reconstruction wavelet for $\psi$.

It is easy to find explicit examples of wavelets admitting reconstruction wavelets; in fact, any non-trivial rotation invariant element of $\mathcal{S}_{0}(\mathbb{R}^{n})$ is itself its own reconstruction wavelet.

If $\psi$ is admits the reconstruction wavelet $\eta$, one has the reconstruction
formula \cite{holschneider} for the wavelet transform on $\mathcal{S}_{0}(\mathbb{R}^{n})$
\begin{equation}
\label{wnwneq8}
\mathrm{Id}_{\mathcal{S}_0(\mathbb{R}^{n})}=\frac{1}{c_{\psi,\eta}}\mathcal{M}_{\eta}\mathcal{W}_{\psi}.
\end{equation}

We now characterize those wavelets which have a reconstruction wavelet. Actually, the class of non-degenerate wavelets from $\mathcal{S}_{0}(\mathbb{R}^{n})$ (cf. Definition \ref{wnwd1}) coincides with the class of wavelets admitting reconstruction wavelets.

\begin{proposition}
\label{wnwp3}
Let $\psi\in\mathcal{S}_{0}(\mathbb{R}^{n})$. Then, $\psi$ admits a reconstruction wavelet if and only if it is non-degenerate.
\end{proposition}
\begin{proof} The necessity is clear, for if $\hat{\psi}(r w_{0})$ identically vanishes in the direction of $w_{0}\in\mathbb{S}^{n-1}$, then $c_{\psi,\eta}(w_{0})=0$ (cf. (\ref{wnwneq7})) for any $\eta\in\mathcal{S}_{0}(\mathbb{R}^{n})$.

Suppose now that $\psi$ is non-degenerate, we will construct a
reconstruction wavelet for it. As in (\ref{wnwneq7}), we write
$$c_{\psi,\psi}(\omega)=\int_{0}^{\infty}|\hat{\psi}(r\omega)|^{2}\frac{\mathrm{d}r}{r}>0, \ \ \
\omega\in\mathbb{S}^{n-1}.$$ Set

$$ \varrho(r,w)= \frac{\hat{\psi}(rw)}{c_{\psi,\psi}(w)}\ , \ \ \ (r,w)\in [0,\infty)\times\mathbb{S}^{n-1}; $$
obviously, if we prove that $\varrho(\left|u\right|, u/\left|u\right|)\in \mathcal{S}(\mathbb{R}^{n})$ and all its partial derivatives vanish at the origin, then $\eta$ given by $\hat{\eta}(u)=\varrho(\left|u\right|, u /\left|u\right|)$ will be a reconstruction wavelet for $\psi$; actually, $c_{\psi,\eta}=1$. By the characterization theorem  of test functions from  $\mathcal{S}(\mathbb{R}^{n})$ in polar coordinates \cite[Prop. 1.1]{drozhzhinov-z4}, the fact $\hat{\eta}\in\mathcal{S}(\mathbb{R}^{n})$ is a consequence of the relations

$$\left.\left(\frac{\partial}{\partial r}\right)^{k} \varrho(r,\omega)\right|_{r=0}=0,\ \ \ k=0,1,\dots \:; $$
the same relations show that all partial derivatives of $\hat{\eta}$ vanish at the origin, and hence $\eta\in\mathcal{S}_{0}(\mathbb{R}^{n})$.
\end{proof}

 In \cite{holschneider}, (\ref{wnwneq8}) was extended to $\mathcal{S}'_0(\mathbb{R}^{n})$ via duality arguments, the main step being the formula
$$
\int^{\infty}_{0}\int_{\mathbb{R}^{n}}\mathcal{W}_{\psi}f(x,y)\Phi(x,y)\frac{\mathrm{d}x\mathrm{d}y}{y}= \left\langle
f(t),\mathcal{M}_{\bar{\psi}}\Phi\:(t)\right\rangle ,
$$
valid for $\Phi \in
\mathcal{S}(\mathbb{H}^{n+1})$ and  $f\in\mathcal{S}'_{0}(\mathbb{R}^{n})$.
It can be easily extended to the $E$-valued case, as the next proposition shows.

\begin{proposition}
\label{wnwp3.4}
 Let $\mathbf{f}\in\mathcal{S}'_{0}(\mathbb{R}^{n},E)$ and $\psi\in\mathcal{S}_{0}(\mathbb{R}^{n})$. Then
\begin{equation}
\label{wnwneq9}
\int^{\infty}_{0}\int_{\mathbb{R}^{n}}\mathcal{W}_{\psi}\mathbf{f}(x,y)\Phi(x,y)\frac{\mathrm{d}x\mathrm{d}y}{y}= \left\langle
\mathbf{f}(t),\mathcal{M}_{\bar{\psi}}\Phi\:(t)\right\rangle,
\end{equation}
for all $\Phi \in
\mathcal{S}(\mathbb{H}^{n+1}).$
\end{proposition}
\begin{proof} The same argument used in Proposition \ref{wnwp2} shows that $$\mathcal{W}_{\psi}:\mathcal{S}'_{0}(\mathbb{R}^{n},E)\mapsto\mathcal{S}'(\mathbb{H}^{n+1},E)$$ is continuous. The linear map $T:\mathcal{S}'_{0}(\mathbb{R}^{n},E)\mapsto\mathcal{S}'(\mathbb{H}^{n+1},E)$ given by
$$
\left\langle (T\mathbf{f})(x,y),\Phi(x,y)\right\rangle=\left\langle \mathbf{f}(t),\mathcal{M}_{\bar{\psi}}\Phi\:(t)\right\rangle ,
$$
is continuous as well. Thus, if we show that $\mathcal{W}_{\psi}$ and $T$ coincide on a dense subset of $\mathcal{S}'_{0}(\mathbb{R}^{n},E)$, we would have (\ref{wnwneq9}). The nuclearity \cite{treves} of $\mathcal{S}'_{0}(\mathbb{R}^{n})$ implies that $\mathcal{S}'_{0}(\mathbb{R}^{n})\otimes E\subset\mathcal{S}'_{0}(\mathbb{R}^{n},E)$ is dense; thus, it is enough to verify (\ref{wnwneq9}) for $\mathbf{f}=f\mathbf{v}$, where $f\in\mathcal{S}'_{0}(\mathbb{R}^{n})$ and $\mathbf{v}\in E$. Now, the scalar case implies
\begin{align*}
\left\langle \mathcal{W}_{\psi}(f\mathbf{v})(x,y),\Phi(x,y)\right\rangle & =\left\langle \mathcal{W}_{\psi}f(x,y),\Phi(x,y)\right\rangle\mathbf{v}
\\
&
= \left\langle f(t),\mathcal{M}_{\bar{\psi}}\Phi\:(t)\right\rangle\mathbf{v}
\\
&
=\left\langle f(t)\mathbf{v},\mathcal{M}_{\bar{\psi}}\Phi\:(t)\right\rangle ,
\end{align*}
as required.
\end{proof}

We now extend the definition of the wavelet synthesis operator (\ref{wnwneq6}) to $\mathcal{S}'_{0}(\mathbb{H}^{n+1},E)$. Let $\mathbf{K}\in\mathcal{S}'_{0}(\mathbb{H}^{n+1},E)$. We define $\mathcal{M}_{\psi}:\mathcal{S}'_{0}(\mathbb{H}^{n+1},E)\mapsto\mathcal{S}'_{0}(\mathbb{R}^{n},E)$, a continuous linear map, as
\begin{equation*}
\left\langle
\mathcal{M}_{\psi}\mathbf{K}(t),\rho(t)\right\rangle=\left\langle
\mathbf{K}(x,y),\mathcal{W}_{\bar{\psi}}\rho(x,y)\right\rangle, \ \ \ \rho\in\mathcal{S}_{0}(\mathbb{R}^{n}).
\end{equation*}

 So, we have the ensuing reconstruction formula for the wavelet transform.
\begin{proposition}
\label{wnwth1}
Let $\psi\in\mathcal{S}_{0}(\mathbb{R}^{n})$ be non-degenerate and let $\eta\in\mathcal{S}_{0}(\mathbb{R}^{n})$ be a reconstruction wavelet for it. Then,
\begin{equation}
\label{wnwneq11}
\mathrm{Id}_{\mathcal{S}'_0(\mathbb{R}^{n},E)}=\frac{1}{c_{\psi,\eta}}\mathcal{M}_{\eta}\mathcal{W}_{\psi} .
\end{equation}
Furthermore, we have the desingularization formula,
\begin{equation}
\label{wnwneq12}
\left\langle
\mathbf{f}(t),\rho(t)\right\rangle=\frac{1}{c_{\psi,\eta}}\int^{\infty}_{0}\int_{\mathbb{R}^{n}}\mathcal{W}_{\psi}\mathbf{f}(x,y)\mathcal{W}_{\bar{\eta}}\rho(x,y)\frac{\mathrm{d}x\mathrm{d}y}{y}\: ,
\end{equation}
for all $\mathbf{f}\in\mathcal{S}'_{0}(\mathbb{R}^{n},E)$ and $\rho\in\mathcal{S}_{0}(\mathbb{R}^{n})$.
\end{proposition}
\begin{proof}
We apply the definition of $\mathcal{M}_{\eta}$, Proposition \ref{wnwp3.4}, (\ref{wnwneq8}), and use the fact that $c_{\psi,\eta}=c_{\bar{\eta},\bar{\psi}}$,
\begin{align*}\frac{1}{c_{\psi,\eta}}\left\langle \mathcal{M}_{\eta}\mathcal{W}_{\psi}\mathbf{f}(t),\rho(t)\right\rangle
&
=\frac{1}{c_{\psi,\eta}}\int^{\infty}_{0}\int_{\mathbb{R}^{n}}\mathcal{W}_{\psi}\mathbf{f}(x,y)\mathcal{W}_{\bar{\eta}}\rho(x,y)\frac{\mathrm{d}x\mathrm{d}y}{y}
\\
&
\\
&
=\frac{1}{c_{\psi,\eta}}\left\langle \mathcal{W}_{\psi}\mathbf{f}(x,y),\mathcal{W}_{\bar{\eta}}\rho(x,y)\right\rangle
\\
&
=\left\langle \mathbf{f}, \frac{1}{c_{\bar{\eta},\bar{\psi}}} \mathcal{M}_{\bar{\psi}}\mathcal{W}_{\bar{\eta}}\rho \right\rangle
\\
&
=\left\langle \mathbf{f}(t),\rho(t)\right\rangle,
\end{align*}
so both (\ref{wnwneq11}) and (\ref{wnwneq12}) have been established.
\end{proof}
The next result provides a second characterization of non-degenerate wavelets from $\mathcal{S}_{0}(\mathbb{R}^{n})$.

\begin{corollary}\label{wnwc2} Let $\psi\in\mathcal{S}_{0}(\mathbb{R}^{n})$. Then, the linear span of the set of dilates and translates of $\bar{\psi}$, $\left\{\bar{\psi}((\:\cdot-x)/y): (x,y)\in\mathbb{H}^{n+1}\right\}$, is dense in $\mathcal{S}_{0}(\mathbb{R}^{n})$ if and only if $\psi$ is a non-degenerate wavelet.
\end{corollary}
\begin{proof} The direct implication is a consequence of the Hahn-Banach theorem and the inversion formula (Proposition \ref{wnwth1}).
On the other hand, suppose that there is
$\omega_{0}\in\mathbb{S}^{n-1}$ such that
$\hat{\psi}(r\omega_{0})=0$ for all $r\in\mathbb{R}_{+}$. Let
$f\in\mathcal{S}'(\mathbb{R}^{n})$ be the distribution whose
Fourier transform is given by $$\langle
\hat{f},\rho\rangle=\int_{0}^{\infty}\rho(r\omega_{0})\mathrm{d}r,$$
then $\mathcal{W}_{\psi}f(x,y)=0,$ for all
$(x,y)\in\mathbb{H}^{n+1}$, which implies that $f$ identically
vanishes on the closure of the linear span of the dilates and
translates of $\bar{\psi}$. This yields the converse.
\end{proof}

In analogy to \cite[Thm. 28.0.1]{holschneider}, we can characterize the bounded sets of $\mathcal{S}'_{0}(\mathbb{R}^{n},E)$. One can also characterize some types of convergent nets. The next propositions will be very important for the subsequent sections.
\begin{proposition}
\label{wnwp4} Let $\psi\in\mathcal{S}_{0}(\mathbb{R}^{n})$ be a non-degenerate wavelet. A necessary and sufficient condition for a set $\mathfrak{B}\subset\mathcal{S}'_{0}(\mathbb{R}^{n},E)$ to be bounded for the topology of pointwise convergence (or bounded convergence) of $\mathcal{S}'_{0}(\mathbb{R}^{n},E)$ is the existence of $k,l\in\mathbb{N}$ and $C>0$ such that
\begin{equation}
\label{wnwneq13}
\left\|\mathcal{W}_{\psi}\mathbf{f}(x,y)\right\|\leq C \left( \frac{1}{y}+y\right)^{k}\left(1+\left|x\right|\right)^{l}, \ \ \ \text{for all } \mathbf{f}\in\mathfrak{B}.
\end{equation}
\end{proposition}
\begin{proof}
The necessity can be established as in the proof of Proposition \ref{wnwp2}. For the sufficiency, we only need to show the boundedness of $\mathfrak{B}$ for the topology of pointwise convergence \cite{treves}, in view of the Banach-Steinhaus theorem. Let $\eta$ be a reconstruction wavelet for $\psi$. Let $\rho\in\mathcal{S}_{0}(\mathbb{R}^{n})$, by the wavelet desigularization formula (cf. Proposition \ref{wnwth1}) and (\ref{wnwneq13}),
$$
\left\|\left\langle \mathbf{f},\rho\right\rangle\right\|\leq  \frac{C}{c_{\psi,\eta}}\int^{\infty}_{0}\int_{\mathbb{R}^{n}}\left( \frac{1}{y}+y\right)^{k}\left(1+\left|x\right|\right)^{l}\left|\mathcal{W}_{\bar{\eta}}\rho(x,y)\right|\frac{\mathrm{d}x\mathrm{d}y}{y}\: ,
$$
and the last quantity is uniformly bounded for $\mathbf{f}\in\mathfrak{B}$ since $\mathcal{W}_{\bar{\eta}}\rho\in\mathcal{S}(\mathbb{H}^{n+1})$. This completes the proof.
\end{proof}

\begin{proposition}
\label{wnwp5} Let $\psi\in\mathcal{S}_{0}(\mathbb{R}^{n})$ be a non-degenerate wavelet. Necessary and sufficient conditions for the net $\left\{\mathbf{f}_{\lambda}\right\}_{\lambda\in\mathbb{R}_{+}}$  to be convergent ($\lambda\to\infty$), for the topology of pointwise convergence (or bounded convergence) of $\mathcal{S}'_{0}(\mathbb{R}^{n},E)$, are the existence of the limits (with respect to the norm of $E$)
\begin{equation}
\label{wnwneq14} \lim_{\lambda\to\infty}\mathcal{W}_{\psi}\mathbf{f}_{\lambda}(x,y) , \ \ \ \text{for each } (x,y)\in\mathbb{H}^{n+1},
\end{equation}
and the existence of $k,l\in\mathbb{N}$ and $C,\lambda_{0}>0$ such that
\begin{equation}
\label{wnwneq15}
\left\|\mathcal{W}_{\psi}\mathbf{f}_{\lambda}(x,y)\right\|\leq C \left( \frac{1}{y}+y\right)^{k}\left(1+\left|x\right|\right)^{l} , \ \ \ \text{for all } \lambda_{0}\leq\lambda.
\end{equation}
In such a case, the limit generalized function $\mathbf{h}=\lim_{\lambda\to\infty}\mathbf{f}_{\lambda}$ satisfies $$\mathcal{W}_{\psi}\mathbf{h}(x,y)=\lim_{\lambda\to\infty}\mathcal{W}_{\psi}\mathbf{f}_{\lambda}(x,y),$$
uniformly over compact subsets of $\mathbb{H}^{n+1}$.
\end{proposition}
\begin{proof}
By Proposition \ref{wnwp4}, (\ref{wnwneq15}) is itself equivalent
to the boundedness of $\left\{\mathbf{f}_{\lambda}\right\}$ for
large values of $\lambda$, which in turn is equivalent to the
equicontinuity of the set for large values of $\lambda$
(Banach-Steinhaus theorem). Because of the standard result
\cite[p. 356]{treves}, the convergence of
$\left\{\mathbf{f}_{\lambda}\right\}_{\lambda\in\mathbb{R}_{+}}$
is then equivalent to the pointwise convergence of the net of
linear mappings over a dense subset of
$\mathcal{S}_{0}(\mathbb{R}^{n})$. But (\ref{wnwneq14}) gives
precisely this convergence over the linear span of
$$
\left\{\bar{\psi}((\:\cdot-x)/y):
(x,y)\in\mathbb{H}^{n+1}\right\},
$$ which is actually dense
(Corollary \ref{wnwc2}). The last property follows by the
definition of convergence in $\mathcal{S}'_{0}(\mathbb{R}^{n},E)$,
since if $K\subset\mathbb{H}^{n+1}$ is a compact set, then
$\left\{y^{-n}\bar{\psi}((\:\cdot-x)/y): (x,y)\in K\right\}$ is
compact in $\mathcal{S}_{0}(\mathbb{R}^{n})$.
\end{proof}

\newpage

\section{Abelian Results}
\label{wnwa}
We present in this section an Abelian proposition for the transform $M^{\mathbf{f}}_{\varphi}$. Its Tauberian counterparts will be the main subject of the next two sections. This Abelian result is essentially due to Drozhzhinov and Zavialov \cite{drozhzhinov-z2,drozhzhinov-z3} (cf. \cite{vindas-estrada6,vindas-pilipovic-rakic}), but we refine their results by adding some information about uniformity in the asymptotics. Let $x_{0}\in\mathbb{R}^{n}$ and $0\leq\vartheta<\pi/2$, we denote by $C_{x_{0},\vartheta}$ the cone of angle
$\vartheta$ in $\mathbb{H}^{n+1}$ with vertex at $x_{0}$, namely,
$$
C_{x_{0},\vartheta}=\left\{  (x,y)\in\mathbb{H}^{n+1}:\left\vert x-x_{0}\right\vert \leq(\tan\vartheta)y\right\}=(x_{0},0)+C_{0,\vartheta}.
$$

\begin{proposition}
\label{wnwth2}
Let $L$ be slowly varying at the origin (resp. at infinity) and let $\mathbf{f}\in\mathcal{S}'(\mathbb{R}^{n},E)$.
\begin{itemize}
\item [(i)] Assume that $\mathbf{f}$ is weak-asymptotically bounded of degree $\alpha$ at the point $x_{0}$ (resp. at infinity) with respect to $L$ in $\mathcal{S}'(\mathbb{R}^{n},E)$. Then, there exist $k,l\in\mathbb{N}$, $C>0$ and $\varepsilon_{0}>0$ (resp. $\lambda_{0}>0$) such that for all $(x,y)\in\mathbb{H}^{n+1}$
\begin{equation}
\label{wnwaeq1}
\left\|M_{\varphi}^\mathbf{f}(x_{0}+\varepsilon x,\varepsilon y)\right\|\leq C \varepsilon^{\alpha}L(\varepsilon) \left( \frac{1}{y}+y\right)^{k}\left(1+\left|x\right|\right)^{l} , \ \ 0<\varepsilon\leq\varepsilon_{0},
\end{equation}
\begin{equation*}
\left(\text{resp. }\left\|M_{\varphi}^\mathbf{f}(\lambda x,\lambda y)\right\|\leq C \lambda^{\alpha}L(\lambda) \left( \frac{1}{y}+y\right)^{k}\left(1+\left|x\right|\right)^{l} , \ \ \lambda_{0}\leq\lambda\  \right).
\end{equation*}
\item [(ii)] If $\mathbf{f}\in\mathcal{S}'(\mathbb{R}^{n},E)$ has the weak-asymptotic behavior
$\mathbf{f}\left(x_0+\varepsilon t\right)\sim \varepsilon^{\alpha}L(\varepsilon)\mathbf{g}(t)$ as  $\varepsilon\to0^{+}$
(resp. $\mathbf{f}\left(\lambda t\right)\sim\lambda^{\alpha}L(\lambda)\mathbf{g}(t)$ as $\lambda\to\infty$)
in $\mathcal{S}'(\mathbb{R}^{n},E)$, and if $0\leq\vartheta<\pi/2$,
then
\begin{equation}
\label{wnwaeq2}
\lim
_{
_{\substack{\left( x,y\right)\rightarrow\left( 0,0\right)
\\ \left( x,y\right)  \in C_{0,\vartheta} }
}}
\left|(x,y)\right|^{-\alpha}\left\|\frac{1}{L\left(\left|(x,y)\right|\right)}
M_{\varphi}^{\mathbf{f}}(x_{0}+x,y)-M^{\mathbf{g}}_{\varphi}\left(x,y\right)\right\|=0
\end{equation}
$$\left(\mbox{resp.}\:
\lim _{ _{\substack{\left|\left(
x,y\right)\right|\rightarrow\infty \\ \left( x,y\right)  \in
C_{0,\vartheta} } }}
\left|(x,y)\right|^{-\alpha}\left\|\frac{1}{L\left(\left|(x,y)\right|\right)}
M_{\varphi}^{\mathbf{f}}(x,y)-M^{\mathbf{g}}_{\varphi}\left(x,y\right)\right\|=0\right)
;
$$
in particular, for each fixed $(x,y)\in\mathbb{H}^{n+1}$,
\begin{equation}
\label{wnwaeq3} \lim_{\varepsilon\to 0^{+}}
\frac{1}{\varepsilon^{\alpha}L\left(\varepsilon\right)}M_{\varphi}^{\mathbf{f}}(x_{0}+\varepsilon
x,\varepsilon y)=M^{\mathbf{g}}_{\varphi}\left(x,y\right) \ \ \
\text{in } E
\end{equation}
\begin{equation*}
\left(\mbox{resp. } \lim_{\lambda\to\infty} \frac{1}{\lambda^{\alpha}L\left(\lambda\right)}M_{\varphi}^{\mathbf{f}}(\lambda x,\lambda y)=M^{\mathbf{g}}_{\varphi}(x,y)\right).
\end{equation*}
\end{itemize}

\end{proposition}
\begin{proof} The estimate (\ref{wnwaeq1}) from Part (i) follows immediately from Proposition \ref{wnwp2} by considering the bounded set $$\left\{\frac{1}{\varepsilon^{\alpha}L(\varepsilon)}\mathbf{f}(x_{0}+\varepsilon\: \cdot \:):0<\varepsilon\leq 1\right\}\ \
\left(\mbox{resp. }
\left\{\frac{1}{\lambda^{\alpha}/L(\lambda)}\:\mathbf{f}(\lambda\:
\cdot\:): 1\leq\lambda\right\}\right).$$ For (ii), we may assume
that $x_{0}=0$. Next, observe that $(x,y)\in C_{0,\vartheta}$ can
be written as $x=r \xi \mbox{ and } y=r \cos\theta,$ where
$r>0$, $\xi\in\mathbb{R}^{n}$, $\left|\xi\right|=\sin \theta$ and
$0\leq\theta\leq \vartheta$. So,
$$
M^{\mathbf{f}}_{\varphi}(r\xi,r\cos\theta)= \left\langle
\mathbf{f}(rt),\frac{1}{(\cos \theta)^{n}} \varphi\left(\frac{\xi-t}{\cos\theta}\right)
\right\rangle.
$$
Since  
$$
\mathfrak{C}=\left\{ \frac{1}{\cos\theta}\varphi\left(\frac{\xi-\:
\cdot\:}{\cos\theta}\right): \left|\xi\right|=\sin \theta,\ 0\leq\theta\leq\vartheta\right\}
$$ 
is a compact set in $\mathcal{S}(\mathbb{R}^{n})$, the
Banach-Steinhaus theorem implies that the weak-asymptotic behavior
of $\mathbf{f}$ holds uniformly when evaluated at test functions
of $\mathfrak{C}$. Then, as $r\to0^{+}$ (resp. $r\to\infty$),
$$\frac{1}{r^{\alpha}L(r)}M^{\mathbf{f}}_{\varphi}(r\xi,r\cos\theta)\to \left\langle
\mathbf{g}(t),\frac{1}{(\cos \theta)^{n}} \varphi\left(\frac{\xi-t}{\cos\theta}\right)
\right\rangle= M^{\mathbf{g}}_{\varphi}(\xi,\cos\theta),
$$
uniformly in $\left|\xi\right|=\sin \theta$ and $0\leq\theta\leq\vartheta$. Thus, we have shown (\ref{wnwaeq2}). On the other hand, if again $x=r \xi$ and $y=r \cos\theta$, where $r,\xi$ and $\theta$ are fixed, we have that, as $h\to0^{+}$ (resp. $h\to\infty$),
\begin{align*}
 M^{\mathbf{f}}_{\varphi}(hx,hy) &\sim (rh)^{\alpha}L(hr)\left\langle
\mathbf{g}(t),\frac{1}{(\cos \theta)^{n}} \varphi\left(\frac{\xi-t}{\cos\theta}\right)
\right\rangle
\\
&=h^{\alpha}L(hr)\left\langle
\mathbf{g}(rt),\frac{1}{(\cos \theta)^{n}} \varphi\left(\frac{\xi-t}{\cos\theta}\right)
\right\rangle
\\
&
\sim h^{\alpha}L(h)M^{\mathbf{g}}_{\varphi}(x,y)
, \ \ \ \text{in } E,
\end{align*}
because of the homogeneity of $\mathbf{g}$ and the fact that $L$ is slowly varying. Hence, (\ref{wnwaeq3}) has been proved.
\end{proof}

\newpage

\section{Wavelet Tauberian Characterization of  Weak-asymptotics in $\mathcal{S}_{0}'(\mathbb{R}^{n},E)$}
\label{wnwtc}
The purpose of this section is to characterize the weak-asymptotic behavior in the space $\mathcal{S}_{0}'(\mathbb{R}^{n},E)$ in terms of the asymptotic behavior of the wavelet transform with respect to non-degenerate wavelets from $\mathcal{S}_{0}(\mathbb{R}^{n})$. Our characterization is of Tauberian character and it is related to (\ref{wnwaeq1}) and (\ref{wnwaeq3}) for the wavelet transform. Notice that the results below extend those from \cite{vindas-pilipovic-rakic}. Related results in terms of orthogonal wavelet expansions have been considered in \cite{pilipovic-teofanov,saneva-vindas}.

We begin with a preliminary proposition which shows that the conditions (\ref{wnwaeq1}) and (\ref{wnwaeq3}) are equivalent to (apparently) weaker ones.

\begin{proposition}
\label{wnwp5.1}
Let $\mathbf{f}\in\mathcal{S}'(\mathbb{R}^{n},E)$, $\varphi\in\mathcal{S}(\mathbb{R}^{n})$, and let $L$ be slowly varying at the origin (resp. at infinity). Then,
\begin{enumerate}
\item [(i)] The estimate (\ref{wnwaeq1}) is equivalent to one of the form ($k$ may be a different exponent)
\begin{equation}
\label{wnweq5.1}
\limsup_{\varepsilon\rightarrow0^+}\sup_{\left|x\right|^2+y^2=1,\:y>0}\frac{y^k}{\varepsilon^{\alpha}L(\varepsilon)}\left\|M^{\mathbf{f}}_{\varphi}\left(x_0+\varepsilon
x,\varepsilon y\right)\right\|<\infty
\end{equation}
\begin{equation*}
\left(\mbox{resp. }\limsup_{\lambda\rightarrow\infty}\sup_{\left|x\right|^2+y^2=1,\:y>0}\frac{y^k}{\lambda^{\alpha}L(\lambda)}\left\|M^{\mathbf{f}}_{\varphi}\left(\lambda
x,\lambda y\right)\right\|<\infty\right).
\end{equation*}
\item[(ii)] If
\begin{equation}
\label{wnweq5.2}
\lim_{\varepsilon\to0^{+}}\frac{1}{\varepsilon^{\alpha}L(\varepsilon)}M_{\varphi}^\mathbf{f}(x_{0}+\varepsilon x,\varepsilon y)=M_{x,y}\in E
\end{equation}
$$
\left(\mbox{resp. } \lim_{\lambda\to\infty}\frac{1}{\lambda^{\alpha}L(\lambda)}M_{\varphi}^\mathbf{f}(\lambda x,\lambda y)=M_{x,y}\in E\right)
$$
exists for each $(x,y)\in \mathbb{H}^{n+1}\cap\mathbb{S}^{n}$, then it exists for all $(x,y)\in \mathbb{H}^{n+1}$.
\end{enumerate}
\end{proposition}
\begin{proof} By translating, we may assume that $x_{0}=0.$

\emph{Part (i).} We only need to show that (\ref{wnweq5.1}) implies (\ref{wnwaeq1}). Our assumption is that there are constants $C_{1},h_{0}>0$ such that
\begin{equation*}
\left\|M^{\mathbf{f}}_{\varphi}\left(h
\xi,h \cos\vartheta\right)\right\|< \frac{C_{1}}{(\cos\vartheta)^{k}}h^{\alpha}L(h),
\end{equation*}
for all $\left|\xi\right|^{2}+(\cos\vartheta)^{2}=1$ and  $0<h\leq h_{0}$ (resp. $h_{0}\leq h$).

We can assume that $1+\left|\alpha\right|\leq k$ and $h_{0}< 1$ (resp. $1<h_{0}$). Potter's estimate \cite[p. 25]{bingham} implies that we may assume that
\begin{align}
\label{wnweq5.3}
\frac{L(h r)}{L(h)}< C_{2}\frac{(1+r)^{2}}{r} ,  \ \ \  & \text{for }h,h r\in(0,h_{0}]
\\
&(\mbox{resp. } h,h r\in[h_{0},\infty)\:).\nonumber
\end{align}
In addition, since $1/L(h)=o(h^{-1})$ as $h\to0^{+}$ (resp. $1/L(h)=o(h)$, as $h\to\infty$) \cite{bingham,seneta}, we can assume
\begin{equation}
\label{wnweq5.4}
\frac{1}{L(h)}<\frac{C_{3}}{h}, \mbox{ for } 0<h\leq h_{0}
\
\left( \mbox{resp. }  \frac{1}{L(h)}<C_{3}h, \mbox{ for } h_{0}\leq h\: \right).
\end{equation}
After this preparation, we are ready to give the proof. For $(x,y)\in\mathbb{H}^{n+1}$ write $x=r\xi$ and $y=r\cos\vartheta$, with $r=\left|(x,y)\right|$. We always keep $h\leq h_{0}$ (resp. $h_{0}\leq h$). If $rh\leq h_{0}$ (resp. $h_{0}\leq rh$), we have that
\begin{align*}
\left\|M^{\mathbf{f}}_{\varphi}\left(h r
\xi,h r\cos\vartheta\right)\right\|&<\frac{C_{1}}{y^{k}}h^{\alpha} L(h r) r^{\alpha+k}
<C_{1}C_{2}h^{\alpha} L(h)\frac{(1+r)^{\alpha+k+1}}{y^{k}}
\\
&
<C_{4}h^{\alpha}L(h)\left(\frac{1}{y}+y\right)^{\alpha+2k+1} \left(1+\left|x\right|\right)^{\alpha+k+1},
\end{align*}
with $C_{4}=2^{\alpha+k+1}C_{1}C_{2}$. We now analyze the case $h_{0}<h r$ (resp. $hr<h_{0}$). Proposition \ref{wnwp2} implies the existence of $k_{1},l_{1}\in\mathbb{N}$, $k_{1}\geq k,$ and $C_{5}$ such that
\begin{align*}
\left\|M^{\mathbf{f}}_{\varphi}\left(h x,h y\right)\right\|&<C_{5}\left(\frac{1}{h y}+h y\right)^{k_{1}}\left(1+h\left|x\right|\right)^{l_{1}}
\\
&< C_{5} h^{\alpha}L(h)\left(\frac{1}{ y}+ y\right)^{k_{1}}\left(1+\left|x\right|\right)^{l_{1}} \frac{1}{h^{\alpha+k_{1}}L(h)}
\\
&\left(\mbox{resp. }< C_{5} h^{\alpha}L(h)\left(\frac{1}{ y}+ y\right)^{k_{1}}\left(1+\left|x\right|\right)^{l_{1}} \frac{h^{k_{1}+l_{1}}}{h^{\alpha}L(h)}\right)
\\
&
<C_{3}C_{5} h^{\alpha}L(h)\left(\frac{1}{ y}+ y\right)^{k_{1}}\left(1+\left|x\right|\right)^{l_{1}} \left(\frac{r}{h_{0}}\right)^{k_{1}+\alpha+1}
\\
&\left(\mbox{resp. }< C_{3}C_{5} h^{\alpha}L(h)\left(\frac{1}{ y}+ y\right)^{k_{1}}\left(1+\left|x\right|\right)^{l_{1}} \left(\frac{h_{0}}{r}\right)^{k_{1}+l_1-\alpha+1}\right)
\\
&
<C_{6}h^{\alpha}L(h)\left(\frac{1}{ y}+ y\right)^{\alpha+2k_{1}+1}\left(1+\left|x\right|\right)^{\alpha+l_{1}+k_{1}+1}
\\
&
\left(\mbox{resp. }<C_{6}h^{\alpha}L(h)\left(\frac{1}{ y}+ y\right)^{2k_{1}+l_1-\alpha+1}\left(1+\left|x\right|\right)^{l_{1}}\right),
\end{align*}
with $C_{6}=C_{3}C_{5}(2/h_{0})^{\alpha+k_{1}+1}$ (resp. $C_{6}=C_{3}C_{5}h_{0}^{k_{1}+l_1-\alpha+1}$).
Therefore, if $C=\max\left\{C_{4},C_{6}\right\}$, $k_{2}>\left|\alpha\right|+2k_{1}+l_1+1$ and $l_{2}>\alpha+l_{1}+k_{1}+1$,
$$
\left\|M^{\mathbf{f}}_{\varphi}\left(h x,h y\right)\right\|< Ch^{\alpha}L(h)\left(\frac{1}{ y}+ y\right)^{k_{2}}\left(1+\left|x\right|\right)^{l_{2}},
$$
for all $(x,y)\in\mathbb{H}^{n+1}$ and $0<h\leq h_{0}$ (resp. $h_{0}<h$).

\emph{Part (ii).} Fix $(x,y)\in\mathbb{H}^{n+1}$ and write it as $(x,y)=(r\xi,r \cos\vartheta)$, where $(\xi,\cos\vartheta)\in\mathbb{H}^{n+1}\cap\mathbb{S}^{n}$. Then, as $h\to0^{+}$ (resp. $h\to\infty$), we have
\begin{align*}\frac{1}{h^{\alpha}L(h)}
M_{\varphi}^\mathbf{f}(hr \xi,hr \cos\vartheta)&=
\frac{L(hr)}{L(h)}r^{\alpha}\left(\frac{1}{(hr)^{\alpha}L(hr)}M_{\varphi}^\mathbf{f}(hr \xi,hr \cos\vartheta)\right)
\\
&
\longrightarrow 1\cdot r^{\alpha} M_{\xi,\cos\vartheta}\:,\ \ \ \text{in } E.
\end{align*}
\end{proof}

We now state the Tauberian characterization of weak-asymptotics in the space $\mathcal{S}'_{0}(\mathbb{R}^{n},E)$. The simple proof of the following theorem is a consequence of our previous work.

\begin{theorem}
\label{wnwth5.1}
Let $\psi\in\mathcal{S}_{0}(\mathbb{R}^{n})$ be a non-degenerate wavelet and let $L$ be slowly varying at the origin (resp. at infinity).
\begin{itemize}
\item [(i)] A necessary and sufficient condition for $\mathbf{f}\in\mathcal{S}'(\mathbb{R}^{n},E)$ to be weak-asymptotically bounded of degree $\alpha$ at the point $x_{0}$ (resp. at infinity) with respect to $L$ in $\mathcal{S}'_{0}(\mathbb{R}^{n},E)$ is the existence of $k\in\mathbb{N}$ such that
\begin{equation}
\label{wnweq5.5}
\limsup_{\varepsilon\rightarrow0^+}\sup_{\left|x\right|^2+y^2=1,\:y>0}\frac{y^k}{\varepsilon^{\alpha}L(\varepsilon)}\left\|\mathcal{W}_{\psi}\mathbf{f}\left(x_0+\varepsilon
x,\varepsilon y\right)\right\|<\infty
\end{equation}
\begin{equation*}
\left(\mbox{resp. }\limsup_{\lambda\rightarrow\infty}\sup_{\left|x\right|^2+y^2=1,\:y>0}\frac{y^k}{\lambda^{\alpha}L(\lambda)}\left\|\mathcal{W}_{\psi}\mathbf{f}\left(\lambda
x,\lambda y\right)\right\|<\infty\right) .
\end{equation*}
\item [(ii)] The existence of the limits
\begin{equation}
\label{wnweq5.6}
\lim_{\varepsilon\to0^{+}}\frac{1}{\varepsilon^{\alpha}L(\varepsilon)}\mathcal{W}_{\psi}\mathbf{f}(x_{0}+\varepsilon x,\varepsilon y)=W_{x,y}\: , \ \ \ \text{for each } (x,y)\in\mathbb{H}^{n+1}\cap\mathbb{S}^{n}
\end{equation}
\begin{equation*}
\left(\mbox{resp. } \lim_{\lambda\to\infty}\frac{1}{\lambda^{\alpha}L(\lambda)}\mathcal{W}_{\psi}\mathbf{f}(\lambda x,\lambda y)=W_{x,y}\in E\right),
\end{equation*}
and the estimate (\ref{wnweq5.5}), for some $k\in\mathbb{N}$, are necessary and sufficient for $\mathbf{f}$ to have weak-asymptotic behavior of degree $\alpha$ at the point $x_{0}$ (resp. at infinity) with respect to $L$ in the space $\mathcal{S}'_{0}(\mathbb{R}^{n},E)$.
\end{itemize}
\end{theorem}
\begin{proof}
The equivalence between the weak-asymptotic boundedness and the estimate (\ref{wnweq5.5}) follows at once on combining Proposition \ref {wnwp5.1} with Proposition \ref{wnwp4} when considering the set (in $\mathcal{S}'_{0}(\mathbb{R}^{n},E)$) $$\left\{\frac{1}{\varepsilon^{\alpha}L(\varepsilon)}\mathbf{f}(x_{0}+\varepsilon\: \cdot \:):0<\varepsilon\leq 1\right\} \left(\mbox{resp. }\left\{\frac{1}{\lambda^{\alpha}L(\lambda)}\:\mathbf{f}(\lambda\: \cdot\:): 1\leq\lambda\right\}\right) ,$$
while Part (ii) follows from Proposition \ref {wnwp5.1} and Proposition \ref{wnwp5}.
\end{proof}

We will need the following proposition for future applications when studying Tauberian theorems for the non-wavelet case and wavelet transforms with respect to non-degenerate wavelets from $\mathcal{S}(\mathbb{R}^{n})\setminus\mathcal{S}_{0}(\mathbb{R}^{n})$. It tells us the weak-asymptotic properties of the projection of a tempered distribution onto $\mathcal{S}'_{0}(\mathbb{R}^{n},E)$ when its transform $M^{\mathbf{f}}_{\varphi}$ has asymptotics as in Proposition \ref{wnwp5.1}.

\begin{proposition}
\label{wnwp5.2}
Let $\varphi\in\mathcal{S}(\mathbb{R}^{n})$ be non-degenerate and let $L$ be slowly varying at the origin (resp. at infinity). Suppose that $\mathbf{f}\in\mathcal{S}'(\mathbb{R}^{n},E)$.

\begin{itemize}
\item [(i)] If there exists $k\in\mathbb{N}$ such that the estimate (\ref{wnweq5.1}) holds, then $\mathbf{f}$ is weak-asymptotically bounded of degree $\alpha$ at the point $x_{0}$ (resp. at infinity) with respect to $L$ in the space $\mathcal{S}'_{0}(\mathbb{R}^{n},E)$.
\item [(ii)] If the limit (\ref{wnweq5.2}) exists for each $(x,y)\in \mathbb{H}^{n+1}\cap\mathbb{S}^{n}$, and there is a $k\in\mathbb{N}$ such that the estimate (\ref{wnweq5.1}) is satisfied, then  $\mathbf{f}$ has weak-asymptotic behavior of degree $\alpha$ at the point $x_{0}$ (resp. at infinity) with respect to $L$ in the space $\mathcal{S}'_{0}(\mathbb{R}^{n},E)$.
\end{itemize}
\end{proposition}
\begin{proof}
Translating $\mathbf{f}$, we can assume that $x_{0}=0$. Consider the non-degenerate wavelet $\psi\in\mathcal{S}_{0}(\mathbb{R}^{n})$ given by $\hat{\psi}(u)=e^{-\left|u\right|-(1/\left|u\right|)}$. Set $\psi_{1}=\check{\bar{\varphi}}\ast\psi$, then, $\psi_1\in\mathcal{S}_{0}(\mathbb{R}^{n})$ is also a non-degenerate wavelet. Indeed, $\hat{\psi_{1}}=\widehat{\check{\bar{\varphi}}}\hat{\psi}$ and its partial derivatives of any order vanish at the origin.
First notice that $\mathcal{W}_{\psi_{1}}\mathbf{f}$ is given by
\begin{align*}
\mathcal{W}_{\psi_{1}}\mathbf{f}(x,y)&=\left\langle \mathbf{f}(x+yt),\check{\varphi}\ast\bar{\psi}(t)\right\rangle
\\
&=\left\langle \mathbf{f}(x+yt),\int_{\mathbb{R}^{n}}\bar{\psi}(u)\varphi(u-t)\mathrm{d}u \right\rangle
\\
&
=\int_{\mathbb{R}^{n}}\bar{\psi}(u) \left\langle \mathbf{f}(x+yt),\varphi(u-t) \right\rangle\mathrm{d}u
\\
&
=\int_{\mathbb{R}^{n}}\bar{\psi}(u)(\mathbf{f}\ast\varphi_{y})(x+yu)\mathrm{d}u
\\
&
=\int_{\mathbb{R}^{n}}M^{\mathbf{f}}_{\varphi}(x+yu,y)\bar{\psi}(u)\mathrm{d}u.
\end{align*}

\emph{Part (i)}.
By Proposition \ref{wnwp5.1}, (\ref{wnweq5.1}) is equivalent to an estimate (\ref{wnwaeq1}) ($k$ may be a different number). Our strategy will be to show that $\mathcal{W}_{\psi_{1}}\mathbf{f}$ satisfies (\ref{wnweq5.5}), and then the result would follow immediately from Theorem \ref{wnwth5.1}.
Indeed, for all $(x,y)\in\mathbb{H}^{n+1}\cap\mathbb{S}^{n},$ and $0<h\leq \varepsilon_{0}$ (resp. $\lambda_{0}\leq h$) we have the estimate
\begin{align}
\left\|M^{\mathbf{f}}_{\varphi}(h x+hyu,h y)\right\|& \leq \frac{2^{k}C}{y^{k}} h^{\alpha}L(h)\left(1+\left|x\right|+y\left|u\right|\right)^{l}  \nonumber
\\
& \label{wnweq5.7}
< \frac{C_{1}}{y^{k}}h^{\alpha}L(h) (1+\left|u\right|)^{l},
\end{align}
with $C_{1}=2^{k+l}C$. Therefore, $\mathcal{W}_{\psi_{1}}\mathbf{f}$ satisfies (\ref{wnweq5.5}), namely,
$$
\sup_{\left|x\right|^2+y^2=1,\:y>0}y^{k}\left\|\mathcal{W}_{\psi_{1}}\mathbf{f}(hx,hy)\right\|< C_{2}h^{\alpha}L(h),
$$
where $C_{2}=C_{1}\int_{\mathbb{R}^{n}}(1+\left|u\right|)^{l}\left|\bar{\psi}(u)\right|\mathrm{d}u$.

\emph{Part (ii)}. If the limit (\ref{wnweq5.2}) exists for each $(x,y)\in\mathbb{H}^{n+1}\cap\mathbb{S}^{n-1}$, then so does it for all $(x,y)\in\mathbb{H}^{n+1}$. The estimate (\ref{wnweq5.7}) allows us to use the dominated convergence theorem for Bochner integrals and conclude that, for each fixed $(x,y)\in\mathbb{H}^{n+1}\cap\mathbb{S}^{n}$,
\begin{align*}
\frac{1}{h^{\alpha}L(h)}\mathcal{W}_{\psi}\mathbf{f}(hx,hy)&
=\int_{\mathbb{R}^{n}}\frac{1}{h^{\alpha}L(h)}M^{\mathbf{f}}_{\varphi}(hx+hyu,hy)\bar{\psi}(u)\mathrm{d}u
\\
&
\longrightarrow\int_{\mathbb{R}^{n}}M_{x+yu,y}\:\bar{\psi}(u)\mathrm{d}u,
\end{align*}
as $h\to0^{+}$ (resp. $h\to\infty$). Thus, Theorem \ref{wnwth5.1} yields the result.
\end{proof}

\newpage

\section{Tauberian Theorems in $\mathcal{S}'(\mathbb{R}^{n},E)$}
\label{wnwtt}
We will state and prove in this section Tauberian theorems for weak-asymptotics of tempered $E$-valued distributions.

\subsection{Associate Asymptotically Homogeneous and Homogeneously Bounded Functions}
\label{aah} We need to introduce a class of functions which is of great importance in the study of asymptotic properties of distributions. They appear naturally in the statements and proofs of our Tauberian theorems. The terminology is from  \cite{vindas1,vindas2,vindas3,vindas4,vindas-pilipovic1} (see also de Haan theory in \cite{bingham}).

\begin{definition}
\label{wnwd5} Let  $\mathbf{c}:(0,A)\to E$ (resp. $(A,\infty)\to E$), $A>0$, be a continuous $E$-valued function and let $L$ be slowly varying function at the origin (resp. at infinity). We say that:
\begin{enumerate}
\item [(i)] $\mathbf{c}$ is associate asymptotically homogeneous of degree 0 with respect to  $L$ if for some $\mathbf{v}\in E$
\begin{equation*}
\mathbf{c}(a\varepsilon)=\mathbf{c}(\varepsilon)+L(\varepsilon)\log a\: \mathbf{v}+o(L(\varepsilon)) \ \ \ \text{as}\ \varepsilon\to0^{+}, \ \mbox{ for each } a>0
\end{equation*}
$$
\left(\text{resp. } \mathbf{c}(a\lambda)=\mathbf{c}(\lambda)+L(\lambda)\log a \:\mathbf{v}+o(L(\lambda))\ \ \ \text{as}\ \lambda\to\infty\: \right).
$$
\item [(ii)] $\mathbf{c}$ is asymptotically homogeneously bounded of degree 0 with respect to  $L$ if
\begin{equation*}
\mathbf{c}(a\varepsilon)=\mathbf{c}(\varepsilon)+O(L(\varepsilon)) \ \ \ \text{as}\ \varepsilon\to0^{+}, \ \mbox{ for each } a>0
\end{equation*}
$$
\left(\text{resp. } \mathbf{c}(a\lambda)=\mathbf{c}(\lambda)+O(L(\lambda))\ \ \ \text{as}\ \lambda\to\infty\: \right).
$$
\end{enumerate}
\end{definition}
If $\mathbf{c}$ satisfies either condition (i) or (ii) of Definition \ref{wnwd5}, one can show as in \cite[Prop. 2.3]{vindas1} that given any $\sigma>0$
\begin{equation*}
\left\|\mathbf{c}(\varepsilon)\right\|=o(\varepsilon^{-\sigma}) \ \ \ \mbox{as } \varepsilon\to0^{+} \ \ \ \left(\mbox{resp. } \left\|\mathbf{c}(\lambda)\right\|=o(\lambda^{\sigma})\  \ \ \mbox{as } \lambda\to\infty\: \right).
\end{equation*}
\subsection{Tauberian Theorem for $\phi-$transforms}
\label{wnwtphi}
The ensuing theorem characterizes weak-asymptotic boundedness in terms of the $\phi-$transform.
 \begin{theorem}
 \label{wnwth6.1}Let $\phi\in\mathcal{S}(\mathbb{R}^{n})$ be such that $\mu_{0}(\phi)=1$ and let $L$ be slowly varying at the origin (resp. at infinity).
A necessary and sufficient condition for $\mathbf{f}\in\mathcal{S}'(\mathbb{R}^{n},E)$
to be weak-asymptotically bounded of degree $\alpha\in\mathbb{R}$ at the point $x_0\in\mathbb{R}^{n}$ (resp. at infinity) with respect to $L$ is the existence of $k\in\mathbb{N}$ such that
\begin{equation}
\label{wnwtteq1}
\limsup_{\varepsilon\rightarrow0^+}\sup_{\left|x\right|^2+y^2=1,\:y>0}\frac{y^k}{\varepsilon^{\alpha}L(\varepsilon)}\left\|F_{\phi}\mathbf{f}\left(x_0+\varepsilon
x,\varepsilon y\right)\right\|<\infty
\end{equation}
$$
\left(\mbox{resp. } \limsup_{\lambda\rightarrow\infty}\sup_{\left|x\right|^2+y^2=1,\:y>0}\frac{y^k}{\lambda^{\alpha}L(\lambda)}\left\|F_{\phi}\mathbf{f}\left(\lambda
x,\lambda y\right)\right\|<\infty\:\right).
$$
\end{theorem}
We shall present two different proofs of this theorem. The two methods of proof are applicable to both the case of behavior at infinity and the one at finite points. We concentrate in showing the sufficiency because the necessity follows at once from the Abelian result (Proposition \ref{wnwth2}).
\begin{proof}[First proof of Theorem \ref{wnwth6.1}] We show the case of behavior at the point $x_{0}$ in this first proof. We first need to prove the following claim:

\begin{claim}
\label{wnwl8} Given a set of distinct multi-indices $\left\{m_{l}\right\}_{l=1}^{q}$, a point $u=(u_{1},\cdots,u_{q})\in\mathbb{R}^{q}$, and an arbitrary positive number $\sigma$, there exists a test function $\rho$ in the linear span of $\left\{y^{-n}\phi ((\:\cdot\:-x)/y):(x,y)\in\mathbb{H}^{n+1}\right\}$ such that
$$
\left|u_{l}-\mu_{m_{l}}(\rho)\right|<\sigma, \ \ \ l=1,\cdots,q.
$$
\end{claim}\begin{proof}[Proof of Claim \ref{wnwl8}]
The linear continuous map
$$T:\eta\in\mathcal{S}(\mathbb{R}^{n})\mapsto (\mu_{m_{1}}(\eta),\dots,\mu_{m_{q}}(\eta))\in\mathbb{R}^{q} ,$$
is clearly surjective, as can be verified directly or by using general results (e.g., Borel theorem or results from \cite{duran-estrada,estrada1}). Corollary \ref{wnwc1} implies that the image under $T$ of the linear span of 
$$
\left\{\phi \left(\frac{\:\cdot\:-x}{y}\right):(x,y)\in\mathbb{H}^{n+1}\right\}
$$ is dense in $\mathbb{R}^{q}$, from where we obtain the claimed approximation property.
\end{proof}

We now divide the proof  of Theorem \ref{wnwth6.1} into two cases.

\emph{Case $\alpha\notin\mathbb{N}$}.

Proposition \ref{wnwp5.2} and Proposition A.2 imply the existence of an $E$-valued polynomial
$$\mathbf{P}(t)=\sum_{\left|m\right|\leq d}t^{m}\mathbf{w}_{m}$$
such that
$$
\mathbf{f}(x_{0}+\varepsilon t)=\mathbf{P}(\varepsilon t)+O\left(\varepsilon^{\alpha}L(\varepsilon)\right) \ \ \ \text{in}\ \mathcal{S}'(\mathbb{R}^{n},E).
$$
We must show that $\mathbf{P}(\varepsilon t)=O\left(\varepsilon^{\alpha}L(\varepsilon)\right)$. We may assume that $d<\alpha$ because: $\varepsilon^{\nu-\alpha}=O(L(\varepsilon))$ whenever $\nu>\alpha$ \cite{bingham,seneta}. On the other hand, since $L(\varepsilon)=O(\varepsilon^{-\sigma})$, for any $\sigma>0$, we obtain that
\begin{equation}
\label{wnwtteq2}
\mathbf{f}(x_{0}+\varepsilon t)=\mathbf{P}(\varepsilon t)+O\left(\varepsilon^{d+\kappa}\right) \ \ \ \text{in}\ \mathcal{S}'(\mathbb{R}^{n},E),
\end{equation}
where $\kappa$ is chosen so that $0<\kappa<\alpha-d$. Take $\rho$ in the linear span of $\left\{y^{-n}\phi ((\:\cdot\:-x)/y):(x,y)\in\mathbb{H}^{n+1}\right\}$, this test function is fixed by the moment but its properties will be appropriately chosen later. The hypothesis (\ref{wnwtteq1}) implies that $\left\|\left\langle f(x_{0}+\varepsilon t),\rho(t)\right\rangle\right\|=O(\varepsilon^{d+\kappa})$. Evaluation of (\ref{wnwtteq2}) at $\rho$ and the last fact yield
\begin{equation*}
\sum_{\nu=0}^{d}\varepsilon^{\nu}\sum_{\left|m\right|=\nu}\mu_{m}(\rho)\mathbf{w}_{m}=O(\varepsilon^{d+\kappa}) ,
\end{equation*}
which readily implies that,
\begin{equation}
\label{wnwtteq3}
\sum_{\left|m\right|=\nu}\mu_{m}(\rho)\mathbf{w}_{m}=0 , \ \ \ \text{for}\  \nu=0,1,\cdots,d .
\end{equation}
For a fixed index $0\leq\nu\leq d$, let $q=q_{\nu}$ be the number of multi-indices such that $\left|m\right|=\nu$; moreover, index such multi-indices as $\left\{m_{l}\right\}_{l=1}^{q}$. Given an arbitrary $0<\sigma<1$, we select $\rho$ as in Claim \ref{wnwl8} with $u=e_{l}\in\mathbb{R}^{q}$, the vector with 1 in the $l$th component and zeros in the other entries. Then, (\ref{wnwtteq3}) with this $\rho$ gives
$$
\left\|\mathbf{w}_{m_{l}}\right\|<\frac{\sigma}{1-\sigma}\sum_{i=1,i\neq l}^{q}\left\|\mathbf{w}_{m_{i}}\right\|,
$$
and taking $\sigma\to0^{+}$, we conclude $\mathbf{w}_{m_l}=0$. Since the argument works for all $l$ and $\nu$, it follows that $\mathbf{w}_{m}=0$, for all $\left|m\right|\leq d$. This completes the proof in the first case.

\emph{Case $\alpha=p\in\mathbb{N}$.}

In this case, Proposition \ref{wnwp5.2} and Proposition A.2 imply the existence of $\mathbf{w}_{j}$, $\left|j\right|<p$, and asymptotically homogeneously bounded functions $\mathbf{c}_{m}$, $\left|m\right|=p$, of degree 0 with respect to $L$ such that
$$
\mathbf{f}(x_{0}+\varepsilon t)=\sum_{\left|j\right|<p}\varepsilon^{\left|j\right|}t^{j}\mathbf{w}_{j} +\varepsilon^{p}\sum_{\left|m\right|=p}t^{m}\mathbf{c}_{m}(\varepsilon)+O\left(\varepsilon^{p}L(\varepsilon)\right) \ \ \ \text{in}\ \mathcal{S}'(\mathbb{R}^{n},E) .
$$
We have that each $\mathbf{c}_{m}$ satisfies $\mathbf{c}_{m}(\varepsilon)=O(\varepsilon^{-1/2})$ (cf. Subsection \ref{aah}), and thus
$$
\mathbf{f}(x_{0}+\varepsilon t)=\sum_{\left|j\right|<p}\varepsilon^{\left|j\right|}t^{j}\mathbf{w}_{j} +O\left(\varepsilon^{p-1/2}\right) \ \ \ \text{in}\ \mathcal{S}'(\mathbb{R}^{n},E).
$$
Proceeding as in the preceding case, we conclude that each $\mathbf{w}_{j}=0$. Summarizing, we have shown so far
\begin{equation}
\label{wnwtteq4}
\mathbf{f}(x_{0}+\varepsilon t)=\varepsilon^{p}\sum_{\left|m\right|=p}t^{m}\mathbf{c}_{m}(\varepsilon)+O\left(\varepsilon^{p}L(\varepsilon)\right) \ \ \ \text{in}\ \mathcal{S}'(\mathbb{R}^{n},E).
\end{equation}
Let now $q$ be the number of multi-indices such that $\left|m\right|=p$, once again, we index those multi-indices as $\left\{m_{l}\right\}_{l=1}^{q}$, and consider the vectors $e_{l}\in\mathbb{R}^{q}$ with 1 in the $l$th component and zeros in the other entries. Let $\sigma>0$ be small enough such that if the $q\times q$ matrix $A=\left(a_{l,\nu}\right)_{l,\nu}$ satisfies $\left|a_{l,\nu}-\delta_{l,\nu}\right|<\sigma$, then $A$ is invertible ($\delta_{l,\nu}$ is the Kronecker delta). For each $1 \leq l\leq q$, find $\rho_{l}$ satisfying the conclusions of Claim \ref{wnwl8} for $\sigma$ and $e_{l}$, that is, $\left|\mu_{m_{\nu}}(\rho_{l})-\delta_{l,\nu}\right|<\sigma$. Then, the matrix $A:=\left(\mu_{m_{\nu}}(\rho_{l})\right)_{l,\nu}$ is invertible. Evaluation of (\ref{wnwtteq4}) at the $\rho_{l}$ and the hypothesis (\ref{wnwtteq1}) yield the $q\times q$ system of inequalities
$$
\sum_{\nu=1}^{q} \mu_{m_{v}}(\rho_{l}){\mathbf{c}}_{m_{\nu}}(\varepsilon)=O(L(\varepsilon)) , \ \ \ l=1,\cdots,q.
$$
Multiplication by $A^{-1}$ implies that ${\mathbf{c}}_{m}(\varepsilon)=O(L(\varepsilon))$, for each $\left|m\right|=p$, which turns out to prove
$$
\mathbf{f}(x_{0}+\varepsilon t)=O\left(\varepsilon^{p}L(\varepsilon)\right) \ \ \ \text{in}\ \mathcal{S}'(\mathbb{R}^{n},E),
$$
as required.
\end{proof}
\begin{proof}[Second proof of Theorem \ref{wnwth6.1}] In this
second proof we only consider the behavior at infinity. As in the
first proof, we can conclude the existence of an $E$-valued
polynomial, which can be assumed to have the form
$\mathbf{P}(t)=\sum_{\alpha<\left|m\right|\leq
d}t^{m}\mathbf{w}_{m}$, such that $$\mathbf{f}(\lambda
t)=\mathbf{P}(\lambda t)+O(\lambda^{\alpha}L(\lambda))$$ if
$\alpha\notin\mathbb{N}$, or $$\mathbf{f}(\lambda
t)=\mathbf{P}(\lambda t)+\lambda^{p}
\sum_{\left|m\right|=p}\mathbf{c}_{m}(\lambda
)t^{m}+O(\lambda^{p}L(\lambda))$$ if $\alpha=p\in\mathbb{N}$,
where the $\mathbf{c}_{m}$ are asymptotically homogeneously
bounded of degree 0 with respect to $L$; either asymptotic formula
holding as $\lambda\to\infty$ in the space
$\mathcal{S}'(\mathbb{R}^{n},E)$. We first show that
$\mathbf{P}=\mathbf{0}$ in both cases. Select
$\alpha<\kappa<[\alpha]+1$, since both $L(\lambda)$ and the
$\mathbf{c}_{m}$ are $O(\lambda^{\kappa-\alpha})$, we obtain in
either case that
$$
\mathbf{f}(\lambda t)=\mathbf{P}(\lambda t)+O(\lambda^{\kappa}) \ \ \ \mbox{as }\lambda\to\infty\mbox{ in }\mathcal{S}'(\mathbb{R}^{n}).
$$
If we use the estimate for $F_{\phi}\mathbf{f}(\lambda x, \lambda y)$, we have that for each (fixed) $(x,y)\in\mathbb{H}^{n+1}$,
$$F_{\phi}\mathbf{P}(\lambda x, \lambda y)=\sum_{\alpha<\left|m\right|\leq d} (-\lambda i)^{\left|m\right|}\frac{\partial^{\left|m\right|}}{\partial u^{m}}\left.\left(
e^{i x\cdot u}{\hat{\phi}}( -yu)\right)\right|_{u=0}
\mathbf{w}_{m}=O(\lambda^{\kappa}),
$$
$\lambda\to\infty$. This allows us to conclude that, for each $(x,y)\in\mathbb{H}^{n+1}$ and $\alpha<\nu\leq d$,
$$\mathbf{0}=\sum_{\left|m\right|=\nu}\frac{\partial^{\left|m\right|}}{\partial u^{m}}\left.\left(
e^{i x\cdot u}{\hat{\phi}}( -yu)\right)\right|_{u=0}
\mathbf{w}_{m}=\sum_{\left|m\right|=\nu}(ix)^{m}\mathbf{w}_{m}+\sum_{q=1}^{\nu}(iy)^{q}\mathbf{R}_{q}(x),
$$
for certain $E$-valued polynomials $\mathbf{R}_{q}$. But we can
take $y\to0^{+}$ in the above equation, which implies that
$\mathbf{w}_{m}=\mathbf{0}$ for $\left|m\right|=q$, and since the
same holds for every $\alpha<\nu\leq d$, we have just shown that
$\mathbf{P}=\mathbf{0}$. Therefore, the case
$\alpha\notin\mathbb{N}$ has been established. The case
$\alpha=p\in\mathbb{N}$ would now follow if we were able to prove
that
$$\mathbf{C}(\lambda,t):=\sum_{\left|m\right|=p}t^{m}\mathbf{c}_{m}(\lambda)=O(L(\lambda))$$
as $\lambda\to\infty$ in $\mathcal{S}'(\mathbb{R}^{n},E)$.  We
keep $(x,y)\in B(0,1)\times(0,1)$, where $B(0,1)$ is the unit ball
in $\mathbb{R}^{n}$. By Proposition \ref{wnwp2}, applied to
$$\frac{1}{\lambda^{p}L(\lambda)}\left(\mathbf{f}(\lambda t)-\lambda^{p}\sum
_{\left|m\right|=p}t^{m}\mathbf{c}_{m}(\lambda)\right),$$ and
(\ref{wnwtteq1}), there are constants $\lambda_{0},C>0$ and
$l\in\mathbb{N}$ such that
$$
\left\|\lambda^{p}\sum_{\left|m\right|=p} (-\lambda i)^{p}\frac{\partial^{\left|m\right|}}{\partial u^{m}}\left.\left(
e^{i x\cdot u}{\hat{\phi}}( -yu)\right)\right|_{u=0}
\mathbf{c}_{m}(\lambda)\right\|\leq \frac{C}{y^{l}}\lambda^{p}L(\lambda),
$$
for all $(x,y)\in B(0,1)\times(0,1)$ and $\lambda_{0}\leq\lambda$, that is,
$$
\left\|\mathbf{C}(\lambda,x)+\sum_{\nu=1}^{p}y^{\nu}\mathbf{C}_{\nu}(\lambda,x)
\right\|\leq \frac{CL(\lambda)}{y^{l}},
$$
for suitable $E$-valued functions $\mathbf{C}_{\nu}(\lambda,x)$.
If we now select $p+1$ points $0<y_{1}<y_{2}\dots<y_{p}<1$, we
obtain a system of $p+1$ inequalities with Vandermonde matrix
$A=(y^{\nu}_{j})_{j,\nu}$. Multiplying by $A^{-1}$ and setting
$x=t/(1+\left|t\right|)$, we can find a constant $C_1$ such that
$$
\left\|\mathbf{C}(\lambda,t)\right\|\leq C_{1}(1+\left|t\right|)^{p}L(\lambda), \ \ \ \mbox{for all } t\in\mathbb{R}^{n}\mbox{ and }\lambda_{0}\leq\lambda.
$$
This completes the proof.

\end{proof}

We now investigate the weak-asymptotic behavior.
\begin{theorem}
 \label{wnwth6.2} Let $\phi\in\mathcal{S}(\mathbb{R}^{n})$ be such that $\mu_{0}(\phi)=1$ and let $L$ be slowly varying at the origin (resp. at infinity).
Then, the existence of the limits
\begin{equation}
\label{wnwtteq5}
\lim_{\varepsilon\to0^{+}}\frac{1}{\varepsilon^{\alpha}L(\varepsilon)}F_{\phi}\mathbf{f}(x_{0}+\varepsilon x,\varepsilon y)=F_{x,y}\: , \ \ \ \text{for each } (x,y)\in\mathbb{H}^{n+1}\cap\mathbb{S}^{n}
\end{equation}
$$\left(\mbox{resp. } \lim_{\lambda\to\infty}\frac{1}{\lambda^{\alpha}L(\lambda)}F_{\phi}\mathbf{f}(\lambda x,\lambda y)=F_{x,y}\in E\right),
$$
and the estimate (\ref{wnwtteq1}), for some $k\in\mathbb{N}$,
are necessary and sufficient for $\mathbf{f}$ to have weak-asymptotic behavior in the space $\mathcal{S}'(\mathbb{R}^{n},E)$, namely, the existence of an $E$-valued homogeneous distribution $\mathbf{g}\in\mathcal{S}'(\mathbb{R}^{n},E)$ such that
\begin{equation}
\label{wnwtteq6}
\mathbf{f}(x_{0}+\varepsilon t)\sim \varepsilon^{\alpha}L(\varepsilon)\mathbf{g}(t)\ \ \ \mbox{as}\ \varepsilon\to0^{+}\ \ \mbox{in}\ \mathcal{S}'(\mathbb{R}^{n},E)
\end{equation}
$$\left(\mbox{resp. }\mathbf{f}\left(\lambda t\right)\sim\lambda^{\alpha}L(\lambda)\mathbf{g}(t) \ \ \ \mbox{as}\ \lambda\to\infty \ \ \mbox{in}\ \mathcal{S}'(\mathbb{R}^{n},E)\:\right).
$$
In such a case, $\mathbf{g}$ is completely determined by $F_{\phi}\mathbf{g}(x,y)=F_{x,y}$.
\end{theorem}
\begin{proof} Theorem \ref{wnwth6.1} gives the equivalence of (\ref{wnwtteq1}) with the weak-asymptotic boundedness of $\mathbf{f}$. So, by the Banach-Steinhaus theorem, (\ref{wnwtteq6}) is now equivalent to the convergence of $(\varepsilon^{-\alpha}/L(\varepsilon))\mathbf{f}(x_{0}+\varepsilon \:\cdot\:)$ (resp. $(\lambda^{-\alpha}/L(\lambda))\mathbf{f}(\lambda \:\cdot\:)$) over a dense subset of $\mathcal{S}(\mathbb{R}^{n})$. By Corollary \ref{wnwc1}, the linear span of the set $\left\{y^{-n}\phi ((\:\cdot\:-x)/y):(x,y)\in\mathbb{H}^{n+1}\right\}$ is dense in $\mathcal{S}(\mathbb{R}^{n})$, it remains only to observe that (\ref{wnwtteq5}) gives precisely convergence over such a dense subset.
\end{proof}

\begin{remark} We have stated the theorems of this subsection only for $\phi-$transforms, but they are obviously true for any non-wavelet transform $M^{\mathbf{f}}_{\varphi}$ if we just assume that $\mu_{0}=\mu_{0}(\varphi)=\int_{\mathbb{R}^{n}}\varphi(t)\mathrm{d}t\neq0$. Indeed, it follows simply by considering the $\phi-$transform with kernel $\phi=\mu_{0}^{-1}\check{\varphi}$.
\end{remark}

\subsection{Tauberian Theorems for Wavelet Transforms}\label{wnwwaw}
We now present the Tauberian theorems for wavelet transforms. We begin with weak-asymptotic boundedness.

\begin{theorem}
\label{wnwth8}
Let $\mathbf{f}\in\mathcal{S}'(\mathbb{R}^{n},E)$, let $\psi\in\mathcal{S}(\mathbb{R}^{n})$ be a non-degenerate wavelet, and let $L$ be slowly varying at the origin (resp. at infinity). The estimate (\ref{wnweq5.5}), for some $k\in\mathbb{N}$, is sufficient for the existence of an $E$-valued polynomial $\mathbf{P}$, of degree less than $\alpha$ (resp. of the form $\mathbf{P}(t)=\sum_{\alpha<\left|m\right|\leq d}t^{m}\mathbf{w}_{m}$, for some $d\in\mathbb{N}$), such that:
\begin{itemize}
\item [(i)] If $\alpha\notin\mathbb{N}$,
$\mathbf{f}-\mathbf{P}$ is weak-asymptotically bounded of degree $\alpha$ at the point $x_{0}$ (at infinity) with respect to $L$ in the space $\mathcal{S}'(\mathbb{R}^{n},E)$.
\item [(ii)] If $\alpha=p\in\mathbb{N}$, there exist asymptotically homogeneously bounded $E$-valued functions $\mathbf{c}_{m}$, $\left|m\right|=p$, of degree 0 with respect to $L$ such that $\mathbf{f}$ has the following asymptotic expansion
\begin{equation*}
\mathbf{f}(x_{0}+\varepsilon t)=\mathbf{P}(\varepsilon t)+\varepsilon^{p}\sum_{\left|m\right|=p}t^{m}\mathbf{c}_{m}(\varepsilon)+O\left(\varepsilon^{p}L(\varepsilon)\right)
\end{equation*}
$$\left(\mbox{resp. }\mathbf{f}\left(\lambda t\right)=\mathbf{P}(\lambda t)+\lambda^{p}\sum_{\left|m\right|=p}t^{m}\mathbf{c}_{m}(\lambda)+O\left(\lambda^{p}L(\lambda)\right)\:\right) ,
$$
as $\varepsilon\to0^{+}$ (resp. $\lambda\to\infty$) in the space $\mathcal{S}'(\mathbb{R}^{n},E)$.
\end{itemize}

Moreover, denote by $P_{q}$ the homogeneous terms of the Taylor
polynomials of $\hat{\psi}$ at the origin, that is,
\begin{equation}
\label{wnwtteq7}
P_{q}(u)=\sum_{\left|m\right|=q}\frac{{\hat{\psi}}^{(m)}(0)u^{m}}{m!},
\ \ \ q\in\mathbb{N}.
\end{equation}
Then, the $E$-valued polynomial $\mathbf{P}$ must satisfy
\begin{equation}
\label{wnwtteq8}
\overline{P_{q}}\left(\frac{\partial}{\partial t}\right)\mathbf{P}=\mathbf{0}, \ \ \ \mbox{for all } q\in\mathbb{N}.
\end{equation}
\end{theorem}
\begin{proof} By Proposition \ref{wnwp5.2}, (\ref{wnweq5.5}) implies that $\mathbf{f}$ is weak-asymptotically bounded in
 the space $\mathcal{S}'_{0}(\mathbb{R}^{n},E)$. The existence of the $E$-valued polynomial $\mathbf{P}$ and the $\mathbf{c}_{m}$, in case (ii),
 is then a direct consequence of Proposition A.2. The assertion about the degree of
 $\mathbf{P}$ follows from the growth properties of $L$ (in the case (ii) the terms of order $\left|m\right|=p$ can be
 assumed to be absorbed by the $\mathbf{c}_{m}$). It remains to establish that $\mathbf{P}$ satisfies the equations
 (\ref{wnwtteq8}). We show this fact only in the case of infinity; the proof of the case of behavior at finite
  points is completely analogous. Suppose the $E$-valued polynomial has the form
  $$\mathbf{P}(t)=\sum_{\alpha<\left|m\right|\leq d}t^{m}\mathbf{w}_{m}=\sum_{\nu=[\alpha]+1}^{d}\mathbf{Q}_{\nu}(t),$$
   where each $\mathbf{Q}_{\nu}$ is homogeneous of degree $\nu$. Choose $\alpha<\kappa<[\alpha]+1$.
Then, since $L(\lambda)=O(\lambda^{\kappa-\alpha})$ and $\mathbf{c}_{m}(\lambda)=O(\lambda^{\kappa-\alpha})$, we obtain that
$$
\mathbf{f}(\lambda t)=\mathbf{P}(\lambda t)+O(\lambda^{\kappa}) \ \ \ \mbox{as }\lambda\to\infty \mbox{ in }\mathcal{S}'(\mathbb{R}^{n},E).
$$
But, then, for each fixed $(x,y)\in\mathbb{H}^{n+1}$, the assumption on the size of $\mathcal{W}_{\psi}\mathbf{f}(\lambda x, \lambda y)$ implies that
$$
\mathcal{W}_{\psi}\mathbf{P}(\lambda x, \lambda y)=\sum_{\alpha<\left|m\right|\leq d} (-\lambda i)^{\left|m\right|}\frac{\partial^{\left|m\right|}}{\partial u^{m}}\left.\left(
e^{i x\cdot u}{\overline{\hat{\psi}}}( yu)\right)\right|_{u=0}
\mathbf{w}_{m}=O(\lambda^{\kappa}),
$$
$\lambda\to\infty$. Then, we infer that for each $\alpha<\nu\leq d$ and each $(x,y)\in\mathbb{H}^{n+1}$,
$$
\mathbf{0}=\sum_{\left|m\right|=\nu}\frac{\partial^{\left|m\right|}}{\partial u^{m}}\left.\left(
e^{i x\cdot u}{\overline{\hat{\psi}}}( yu)\right)\right|_{u=0}
\mathbf{w}_{m}=\sum_{q=0}^{\nu}(iy)^{q}(\overline{P_{q}\left(\partial/\partial x\right)}\mathbf{Q}_{\nu})(x).
$$
Thus
$$
\overline{P_{q}}\left(\frac{\partial}{\partial x}\right)\mathbf{Q}_{\nu}=\mathbf{0}, \ \ \ \mbox{for all } q,\nu\in\mathbb{N},
$$
as required.
\end{proof}

We now consider the weak-asymptotic behavior.
\begin{theorem}
\label{wnwth9}
Let $\mathbf{f}\in\mathcal{S}'(\mathbb{R}^{n},E)$, let $\psi\in\mathcal{S}(\mathbb{R}^{n})$ be a non-degenerate wavelet, and let $L$ be slowly varying at the origin (resp. at infinity). Suppose that the estimate (\ref{wnweq5.5}) holds for some $k\in\mathbb{N}$, and that the limits (\ref{wnweq5.6}) exist.
Then, there exist an $E$-valued tempered distribution $\mathbf{g}$, which satisfies $\mathcal{W}_{\psi}\mathbf{g}(x,y)=W_{x,y}$, and an $E$-valued polynomial $\mathbf{P}$, of degree less than $\alpha$ (resp. of the form $\mathbf{P}(t)=\sum_{\alpha<\left|m\right|\leq d}t^{m}\mathbf{w}_{m}$, for some $d\in\mathbb{N}$), such that:
\begin{itemize}
\item [(i)] If $\alpha\notin\mathbb{N}$, $\mathbf{g}$ is homogeneous of degree $\alpha$ and
\begin{equation*}
\mathbf{f}(x_{0}+\varepsilon t)-\mathbf{P}(\varepsilon t)\sim \varepsilon^{\alpha}L(\varepsilon)\mathbf{g}(t)\ \ \ \mbox{as}\ \varepsilon\to0^{+}\ \ \mbox{in}\ \mathcal{S}'(\mathbb{R}^{n},E)
\end{equation*}
$$\left(\mbox{resp. }\mathbf{f}\left(\lambda t\right)-\mathbf{P}(\lambda t)\sim\lambda^{\alpha}L(\lambda)\mathbf{g}(t) \ \ \ \mbox{as}\ \lambda\to\infty \ \ \mbox{in}\ \mathcal{S}'(\mathbb{R}^{n},E)\:\right) .
$$
\item [(ii)] If $\alpha=p\in\mathbb{N}$, there exist associate asymptotically homogeneous $E$-valued functions $\mathbf{c}_{m}$, $\left|m\right|=p$, of degree 0 with respect to $L$ such that $\mathbf{f}$ has the following weak-asymptotic expansion
\begin{equation*}
\mathbf{f}(x_{0}+\varepsilon t)=\mathbf{P}(\varepsilon t)+ \varepsilon^{p}L(\varepsilon)\mathbf{g}(t)+\varepsilon^{p}\sum_{\left|m\right|=p}t^{m}\mathbf{c}_{m}(\varepsilon)+o\left(\varepsilon^{p}L(\varepsilon)\right)
\end{equation*}
$$\left(\mbox{resp. }\mathbf{f}\left(\lambda t\right)=\mathbf{P}(\lambda t)+\lambda^{p}L(\lambda)\mathbf{g}(t)+\lambda^{p}\sum_{\left|m\right|=p}t^{m}\mathbf{c}_{m}(\lambda)+o\left(\lambda^{p}L(\lambda)\right)\:\right),
$$
as $\varepsilon\to0^{+}$ (resp. $\lambda\to\infty$) in the space $\mathcal{S}'(\mathbb{R}^{n},E)$.
\end{itemize}
Furthermore, $\mathbf{P}$ satisfies the equations (\ref{wnwtteq8}).
\end{theorem}

\begin{proof} Proposition \ref{wnwp5.2}, under the assumptions (\ref{wnweq5.5}) and (\ref{wnweq5.6}), implies that $\mathbf{f}$ has weak-asymptotic behavior in the space $\mathcal{S}' _{0}(\mathbb{R}^{n},E)$.
An application of Proposition A.1 yields now the existence of $\mathbf{g}$, $\mathbf{P}$, and the $\mathbf{c}_{m}$ in case (ii). That $\mathbf{P}$ satisfies the equations (\ref{wnwtteq8}) actually follows from  Theorem \ref{wnwth8}.
\end{proof}

When $\alpha\notin\mathbb{N}$ in Theorem \ref{wnwth9}, the condition $\mathcal{W}_{\psi}\mathbf{g}(x,y)=W_{x,y}$ uniquely determines $\mathbf{g}$, in view of its homogeneity. On the other hand, if $\alpha\in\mathbb{N}$, the prescribed values of $\mathcal{W}_\psi\mathbf{g}$ can only determine $\mathbf{g}$ modulo polynomials which are homogeneous of degree $\alpha$.

At this point it is worth to point out that the use of non-degenerate wavelets in Theorem \ref{wnwth8} and Theorem \ref{wnwth9} is absolutely imperative. Clearly, if $\hat{\psi}$ identically vanishes on a ray through the origin, then there are distributions for which $\mathcal{W}_{\psi}\mathbf{f}$ is identically zero and hence for those distributions the hypothesis (\ref{wnweq5.5}) is satisfied for all $\alpha$. However, among such distributions $\mathbf{f}$, it is easy to find explicit examples for which the conclusion of Theorem \ref{wnwth8} does not hold for a given $\alpha$.

Observe that if  $\psi\in\mathcal{S}_{0}(\mathbb{R}^{n})$ in Theorem \ref{wnwth8} and Theorem \ref{wnwth9}, then the converses are also true, as follows from the moment vanishing properties of $\psi$. This is actually the content of Theorem \ref{wnwth5.1}. The same consideration holds for the case of weak-asymptotics at finite points if we employ wavelets $\psi$ that have all vanishing moments $\mu_{m}(\psi)=0$ up to the order $[\alpha]$.

\newpage

\section{Tauberian Class Estimates}

\label{wnwce} In this section we show that the estimate of type
\begin{equation}
\label{wnwceeq1} \left\|M^{\mathbf{f}}_{\varphi}(x,y)\right\|\leq
C \frac{(1+y)^{k}\left(1+\left|x\right|\right)^{l}}{y^{k}}, \ \ \
(x,y)\in\mathbb{H}^{n+1},
\end{equation}
characterizes the space $\mathcal{S}'(\mathbb{R}^{n},E)$. We call
(\ref{wnwceeq1}) a  \emph{global class estimate}, and we
may say that it has an intrinsic Tauberian nature. Specifically, we prove
that if $\mathbf{f}$ takes values in a ``broad'' locally convex
space which contains the narrower Banach space $E$, and if
$\mathbf{f}$ satisfies (\ref{wnwceeq1}) for a non-degenerate test
function $\varphi$, then, there is a distribution $\mathbf{G}$
with values in the broad space such that
$\operatorname*{supp}\hat{\mathbf{G}}\subseteq\left\{0\right\}$
and $\mathbf{f}-\mathbf{G}\in\mathcal{S}'(\mathbb{R}^{n},E).$ In
case when the broad space is a normed one, $\mathbf{G}$ reduces
simply to a polynomial. This will be done in Subsection \ref{gce}.

We shall also investigate in Subsection \ref{lce} the consequences
of (\ref{wnwceeq1}) when it is only assumed to hold for
$(x,y)\in\mathbb{R}^{n}\times(0,1]$, we call it then a \emph{local
class estimate}. In this case the situation is slightly different
and we obtain that
$\mathbf{f}-\mathbf{G}\in\mathcal{S}'(\mathbb{R}^{n},E),$ where
$\hat{\mathbf{G}}$ has compact support but its support may not be
any longer the origin. We may take $\mathbf{G}$ with $\operatorname*{supp}\hat{\mathbf{G}}\subseteq\left\{0\right\}$ if we employ the wavelets introduced in the Example \ref{wnwex3.3} (Subsection \ref{wnwces}). For the $\phi-$transform
$\mathbf{G}$ does not occur (Subsection \ref{wnwcephi}).

We point out that the results of this section extend in several
directions those of Drozhzhinov and Zavialov from
\cite{drozhzhinov-z2,drozhzhinov-z3}.

Throughout this section, unless specified, $X$ is assumed to be a
(arbitrary) Hausdorff locally convex topological vector space such
that the Banach space $E\subset X$ and  the
embedding $E\rightarrow X$ is linear and continuous. Observe that
the transform (\ref{wnwneq1}) makes sense for $X$-valued
distributions as well. Measurability for $E$-valued functions is
meant in the sense of Bochner (i.e., a.e. pointwise limits of
$E$-valued continuous functions); likewise, integrals for
$E$-valued functions are taken in the Bochner sense.

\subsection{Global Class Estimates}
\label{gce} We begin with wavelets in
$\mathcal{S}_{0}(\mathbb{R}^{n})$.
\begin{proposition}
\label{wnwcep1} Let
$\mathbf{f}\in\mathcal{S}'_{0}(\mathbb{R}^{n},X)$ and let
$\psi\in\mathcal{S}_{0}(\mathbb{R}^{n})$ be a non-degenerate
wavelet. The following two conditions,
\begin{align}
\label{wnwceeq2} &\mathcal{W}_{\psi}\mathbf{f}(x,y)\in E, \mbox{
  for almost all value of } (x,y)\in\mathbb{H}^{n+1},
  \\
  & \nonumber
  \mbox{ and it is measurable as an } E\mbox{-valued function},
\end{align}
and there are constants $k,l\in\mathbb{N}$ and $C>0$ such that
\begin{align}
\label{wnwceeq4}
&\left\|\mathcal{W}_{\psi}\mathbf{f}(x,y)\right\|\leq C
\left(\frac{1}{y}+y\right)^{k}\left(1+\left|x\right|\right)^{l},
\\
&
 \mbox{  for almost all }(x,y)\in\mathbb{H}^{n+1},\nonumber
\end{align}
are necessary and sufficient for
$\mathbf{f}\in\mathcal{S}'_{0}(\mathbb{R}^{n},E)$.
\end{proposition}
\begin{proof}
The necessity is clear (Proposition \ref{wnwp4}). We show the
sufficiency. Let $\eta$ be a reconstruction wavelet for $\psi$. We
apply the wavelet synthesis operator to
$\mathbf{K}(x,y)=\mathcal{W}_{\psi}\mathbf{f}(x,y)$, this is valid
because our assumptions (\ref{wnwceeq2}) and (\ref{wnwceeq4}) ensure
that $\mathbf{K}\in\mathcal{S}'(\mathbb{H}^{n+1},E)$. So, set
$$\tilde{\mathbf{f}}:=\mathcal{M}_{\eta}\mathbf{K}\in\mathcal{S}'_{0}(\mathbb{R}^{n},E)
\subset\mathcal{S}'_{0}(\mathbb{R}^{n},X).$$
 We must therefore
show $\tilde{\mathbf{f}}=\mathbf{f}$. Let
$\rho\in\mathcal{S}_{0}(\mathbb{R}^{n})$. We have, by definition,
(\ref{wnwneq6}), and (\ref{wnwneq8}),
\begin{equation*}
\langle
\tilde{\mathbf{f}},\rho\rangle=\frac{1}{c_{\psi,\eta}}\int_{0}^{\infty}\int_{\mathbb{R}^{n}}\left\langle
f(t),
\frac{1}{y^{n}}\bar{\psi}\left(\frac{t-x}{y}\right)\mathcal{W}_{\bar{\eta}}\rho(x,y)\right\rangle\frac{\mathrm{d}x\mathrm{d}y}{y}
\end{equation*}
and
\begin{align*}
\left\langle
\mathbf{f},\rho\right\rangle&=\frac{1}{c_{\psi,\eta}}\left\langle
\mathbf{f},\mathcal{M}_{\bar{\psi}}\mathcal{W}_{\bar{\eta}}\rho\right\rangle
\\
& =\frac{1}{c_{\psi,\eta}}\left\langle
f(t),\int_{0}^{\infty}\int_{\mathbb{R}^{n}}
\frac{1}{y^{n}}\bar{\psi}\left(\frac{t-x}{y}\right)\mathcal{W}_{\bar{\eta}}
\rho(x,y)\right\rangle\frac{\mathrm{d}x\mathrm{d}y}{y}\:
.
\end{align*}
Thus, with
$$\Phi(x,y;t)=\frac{1}{y^{n+1}}\bar{\psi}\left(\frac{t-x}{y}\right)\mathcal{W}_{\bar{\eta}}\rho(x,y),$$
our problem reduces to justify the interchange of the integrals
with the dual pairing in
\begin{align}
\label{wnwceeq5}
&
\int_{0}^{\infty}\int_{\mathbb{R}^{n}}\left\langle
f(t),\Phi(x,y;t) \right\rangle\mathrm{d}x\mathrm{d}y=
\\
& \nonumber
\left\langle
f(t),\int_{0}^{\infty}\int_{\mathbb{R}^{n}}\Phi(x,y;t)\mathrm{d}x\mathrm{d}y\right\rangle
.
\end{align}
To show (\ref{wnwceeq5}), we verify that
\begin{align}
\label{wnwceeq6} 
& \left\langle
\mathbf{w}^\ast,\int_{0}^{\infty}\int_{\mathbb{R}^{n}}\left\langle
f(t),\Phi(x,y;t)
\right\rangle\mathrm{d}x\mathrm{d}y\right\rangle=
\\
&
\left\langle \mathbf{w}^\ast,\left\langle
f(t),\int_{0}^{\infty}\int_{\mathbb{R}^{n}}\Phi(x,y;t)\mathrm{d}x\mathrm{d}y\right\rangle
\right\rangle , \nonumber
\end{align}
 for arbitrary $\mathbf{w}^\ast\in X'$ (here is
where the local convexity of $X$ plays a role). Since the integral
involved in the left hand side of the above expression is a
Bochner integral in $E$ and the restriction of $\mathbf{w}^{\ast}$
to $E$ belongs to $E'$, we obtain at once the exchange formula
\begin{align}
\label{wnwceeq7} 
&
\left\langle
\mathbf{w}^\ast,\int_{0}^{\infty}\int_{\mathbb{R}^{n}}\left\langle
f(t),\Phi(x,y;t)
\right\rangle\mathrm{d}x\mathrm{d}y\right\rangle=
\\
&
\int_{0}^{\infty}\int_{\mathbb{R}^{n}}\left\langle
\mathbf{w}^\ast,\left\langle f(t),\Phi(x,y;t)
\right\rangle\right\rangle\mathrm{d}x\mathrm{d}y.\nonumber
\end{align}
On the other
hand, we may write
$$
\int_{0}^{\infty}\int_{\mathbb{R}^{n}}\Phi(x,y;t)\mathrm{d}x\mathrm{d}y
$$
as the limit of Riemann sums, convergent in
$\mathcal{S}_{0}(\mathbb{R}^{n}_{t})$, we then easily justify the
exchanges that yield
\begin{align}
\label{wnwceeq8}
&
\left\langle \mathbf{w}^\ast,\left\langle
f(t),\int_{0}^{\infty}\int_{\mathbb{R}^{n}}\Phi(x,y;t)\mathrm{d}x\mathrm{d}y\right\rangle
\right\rangle=
\\
& \nonumber
\int_{0}^{\infty}\int_{\mathbb{R}^{n}}\left\langle
\mathbf{w}^\ast,\left\langle f(t),\Phi(x,y;t)
\right\rangle\right\rangle\mathrm{d}x\mathrm{d}y.
\end{align}
The equality (\ref{wnwceeq6}) follows now by comparing
(\ref{wnwceeq7}) and (\ref{wnwceeq8}).
\end{proof}
Proposition \ref{wnwcep1} provides a full characterization of
$\mathcal{S}'_{0}(\mathbb{R}^{n},E)$.

We now aboard the general wavelet case. The $\phi-$transform will
be studied separately in Subsection \ref{wnwcephi} because a
stronger result holds for it.
\begin{theorem}
\label{wnwceth1} Let $\mathbf{f}\in\mathcal{S}'(\mathbb{R}^{n},X)$
and let $\psi\in\mathcal{S}(\mathbb{R}^{n})$ be a non-degenerate
wavelet. Sufficient conditions for the
existence of an $X$-valued distribution $\mathbf{G}\in
\mathcal{S}'(\mathbb{R}^{n},X)$ such that
$\mathbf{f}-\mathbf{G}\in\mathcal{S}'(\mathbb{R}^{n},E)$ and
$\operatorname*{supp} \hat{\mathbf{G}}\subseteq\left\{0\right\}$
are:
\begin{enumerate}
\item[(i)] $\mathcal{W}_{\psi}\mathbf{f}(x,y)$ takes values in $E$ for almost
all $(x,y)\in\mathbb{H}^{n+1}$ and is measurable as an $E$-valued function.
\item[(ii)] There exist constants $k,l\in\mathbb{N}$ and $C>0$ such that (\ref{wnwceeq4}) holds.
\end{enumerate}
\end{theorem}
\begin{proof} Let $\psi_{1}\in\mathcal{S}_{0}(\mathbb{R}^{n})$ be
the non-degenerate wavelet given by
$\hat{\psi}(u)=e^{-\left|u\right|-(1/\left|u\right|)}$. Set
$\psi_{2}=\psi_{1}\ast\psi$, then,
$\psi_{2}\in\mathcal{S}_{0}(\mathbb{R}^{n})$ is also a
non-degenerate wavelet. Using the same argument as in the proof of
Proposition \ref{wnwcep1}, the exchange of integral and dual
paring performed in the proof of Proposition \ref{wnwp5.2} is
valid and so we have the formula
\begin{equation*}
\mathcal{W}_{\psi_{2}}\mathbf{f}(x,y)
=\int_{\mathbb{R}^{n}}\mathcal{W}_{\psi}{\mathbf{f}}(x+yu,y)\overline{\psi}_{1}(u)\mathrm{d}u
,
\end{equation*}
where the integral is taken in the sense of Bochner. Thus, the
restriction of $\mathbf{f}$ to $\mathcal{S}_{0}(\mathbb{R}^{n})$
satisfies the hypotheses of Proposition \ref{wnwcep1}, and hence
there exists $\mathbf{g}\in\mathcal{S}'(\mathbb{R}^{n},E)$ such
that $\left\langle \mathbf{f-g},\rho\right\rangle=0$ for all
$\rho\in\mathcal{S}_{0}(\mathbb{R}^{n})$. This gives at once that
$\mathbf{G}=\mathbf{f-g}$ satisfies $\operatorname*{supp}
\hat{\mathbf{G}}\subseteq\left\{0\right\}$ and
$\mathbf{f-G}\in\mathcal{S}'(\mathbb{R}^{n},E)$.
\end{proof}

When $X$ is a normed space, we obviously have that the only
$X$-valued distributions with support at the origin are precisely
those having the form 
$$
\sum_{\left|m\right|\leq
N}\delta^{(m)}\mathbf{w}_{m}, \ \ \ \mathbf{w}_{m}\in X.
$$
Thus, we
have:

\begin{corollary}
\label{wnwcec1} Let $X$ be a normed space. Then, the conditions
(i) and (ii) of Theorem \ref{wnwceth1} imply the existence of an
$X$-valued polynomial $\mathbf{P}$ such that
$\mathbf{f-P}\in\mathcal{S}'(\mathbb{R}^{n},E)$.

Moreover, if $P_{q}$ denote the homogeneous terms of the Taylor
polynomials of $\hat{\psi}$ at the origin (cf. (\ref{wnwtteq7})),
then
\begin{equation}
\label{wnwceeq10} \overline{P_{q}}\left(\frac{\partial}{\partial
t}\right)\mathbf{f}\in \mathcal{S}'(\mathbb{R}^{n},E) , \ \ \
\mbox{for all } q\in\mathbb{N}.
\end{equation}
\end{corollary}
\begin{proof}
By Theorem \ref{wnwceth1}, one can find $\mathbf{P}\in
\mathcal{S}'(\mathbb{R}^{n},X)$ such that
$\mathbf{f}-\mathbf{P}\in\mathcal{S}'(\mathbb{R}^{n},E)$ and
$\operatorname*{supp} \hat{\mathbf{P}}\subseteq \left\{0\right\}$.
Since $X$ is normed, the point support property of
$\hat{\mathbf{P}}$ implies that $\mathbf{P}$ is a polynomial.
Next, write 
$$
\mathbf{P}(t)=\sum_{\left|m\right|\leq
N}i^{\left|m\right|} t^{m} \mathbf{w}_{m}, \ \ \ \mbox{with }\mathbf{w}_{m}\in X.
$$
The relation (\ref{wnwceeq10}) would follow
immediately if we show that
$\overline{P_{q}\left(\partial/\partial t\right)}\mathbf{P}$ is an
$E$-valued polynomial. Observe that the hypotheses imply that
$\mathcal{W}_{\psi}\mathbf{P}(x,y)\in E$, for almost all $(x,y)$.
Hence, for almost all $(x,y)$,
\begin{align*}
\mathcal{W}_{\psi}\mathbf{P}(x,y)&
=\frac{1}{(2\pi)^{n}}\left\langle \hat{\mathbf{P}}(u),e^{ix\cdot
u}{\overline{\hat{\psi}}}(yu) \right\rangle
\\
&=
\sum_{\left|m\right|\leq N}
\frac{\partial^{\left|m\right|}}{\partial u^{m}}\left.\left(
e^{ix\cdot u}{\overline{\hat{\psi}}}(yu)\right)\right|_{u=0}
\mathbf{w}_{m}
\\
& =
\sum_{q=0}^{N}y^{q}\sum_{\left|m\right|=q}\overline{{\hat{\psi}^{(m)}}(0)}\sum_{|j|\leq
N-q}\binom{m+j}{m}(ix)^{j}\mathbf{w}_{m+j}
\\
& = \sum_{q=1}^{N}(iy)^{q}(\overline{P_{q}\left(\partial/\partial
x\right)}\mathbf{P})(x)\in E.
\end{align*}
But the latter readily implies that
$(\overline{P_{q}\left(\partial/\partial
x\right)}\mathbf{P})(x)\in E,$ for all $0\leq  q\leq N$ and 
$x\in\mathbb{R}^{n}.$
\end{proof}

In general, it is not possible to replace the $\mathbf{G}$ by an
$X$-valued polynomial in Theorem \ref{wnwceth1}. However, we know
some valuable information about $\hat{\mathbf{G}}$. Since it is
supported by the origin, it is easy to show that
$$\hat{\mathbf{G}}=\sum_{m\in\mathbb{N}^{n}}\frac{(-1)^{\left|m\right|}\delta^{(m)}}{m!}\mu_{m}(\hat{\mathbf{G}})
,$$ where 
$$\mu_{m}(\hat{\mathbf{G}})=\left\langle
\hat{\mathbf{G}}(u),u^{m}\right\rangle\in X
$$ 
are its moments
and the series is convergent in $\mathcal{S}'(\mathbb{R}^{n},X)$.
This series is ``weakly finite'', in the sense that for each
$\mathbf{w}^{\ast}\in X'$ there exists
$N_{\mathbf{w}^{\ast}}\in\mathbb{N}$ such that
$$
\langle \mathbf{w}^{\ast},\langle
\hat{\mathbf{G}},\rho\rangle\rangle=\sum_{\left|m\right|\leq
N_{\mathbf{w}^{\ast}}}\frac{\rho^{(m)}(0)}{m!}\left\langle
\mathbf{w}^{\ast}, \mu_{m}(\hat{\mathbf{G}})\right\rangle, \ \ \
\mbox{for all } \rho\in\mathcal{S}(\mathbb{R}^{n}).
$$
Furthermore, given any continuous seminorm $\mathfrak{p}$ on $X$,
one can find an $N_{\mathfrak{p}}\in\mathbb{N}$ such that
$$
\mathfrak{p}\left(\langle\hat{\mathbf{G}},\rho\rangle-\sum_{\left|m\right|\leq
N_{\mathfrak{p}}}\frac{\rho^{(m)}(0)}{m!}\mu_{m}(\hat{\mathbf{G}})\right)=0,
$$
for all $\rho\in\mathcal{S}(\mathbb{R}^{n})$. Finally, we remark
that $\mathbf{G}$, its inverse Fourier transform, can be naturally
identified with an entire $X$-valued function (cf. Subsection
\ref{lce}).
\begin{example} \label{wnwceex1} We consider $X=C(\mathbb{R})$ and $E=C_{b}(\mathbb{R})$,
 the space of continuous bounded functions. Let
 $\chi_{\nu}\in C(\mathbb{R})$ be non-trivial such that
 $\operatorname*{supp }\chi_{\nu}\subset(\nu,\nu+1)$, $\nu\in\mathbb{N}$. Furthermore, for each $\nu\in\mathbb{N}$ find a harmonic homogeneous polynomial $Q_{\nu}$ of degree $\nu$, i.e., $\Delta Q_{\nu}=0$. Consider the $E$-valued distribution
$$
\mathbf{G}(t,\xi)=\sum_{\nu=0}^{\infty}Q_{\nu}(t)\chi_{\nu}(\xi)\in
\mathcal{S}'(\mathbb{R}^n_{t}, C(\mathbb{R}_{\xi}))\setminus
\mathcal{S}'(\mathbb{R}^n_{t}, C_{b}(\mathbb{R}_{\xi})) .
$$
Its Fourier transform is given by an infinite multipole series
supported at the origin, i.e.,
$$
\hat{\mathbf{G}}(u,\xi)=(2\pi)^n\sum_{\nu=0}^{\infty}
\left(Q_{\nu}\left(i\partial/\partial u\right)\delta\right)(u)\:\chi_{\nu}(\xi).
$$
Let $\mathbf{h}\in \mathcal{S}'(\mathbb{R}^{n}, C_{b}(\mathbb{R}))$ and
let $\psi\in\mathcal{S}(\mathbb{R}^{n})\setminus\mathcal{S}_{0}(\mathbb{R}^{n})$ be a non-degenerate
wavelet such that its Fourier transform satisfies $\hat{\psi}(u)=\left|u\right|^2+O(\left|u\right|^N)$ as $u\to0$, for all $N>2$. If 
$$
\mathbf{f}=\mathbf{h}+\mathbf{G}\in
\mathcal{S}'(\mathbb{R}^n, C(\mathbb{R})),
$$
then $\mathcal{W}_{\psi}\mathbf{f}(x,y)=\mathcal{W}_{\psi}\mathbf{h}(x,y)$ for all $(x,y)\in\mathbb{H}^{n+1}$. Thus,
$\mathbf{f}$ satisfies all the hypotheses of Theorem
\ref{wnwceth1}; however, there is no $C(\mathbb{R})$-valued
polynomial $\mathbf{P}$ such that $\mathbf{f-P}\in
\mathcal{S}'(\mathbb{R}^{n}, C_{b}(\mathbb{R}))$.
\end{example}
\subsection{Local Class Estimates}
\label{lce}

We now proceed to study local class estimates, namely,
(\ref{wnwceeq1}) only assumed to hold for
$(x,y)\in\mathbb{R}^{n}\times(0,1]$. Let us start by pointing out
that $M^{\mathbf{f}}_{\varphi}(x,y)$ may sometimes be trivial for
$y\in(0,1)$, this may happen even if $\varphi$ is non-degenerate:

\begin{example}
\label{wnwceex2} Let $\omega\in\mathbb{S}^{n-1}$,
$r\in\mathbb{R}_{+}$; denote $[0,r\omega]=\left\{\sigma\omega:\:
\sigma\in[0,r]\right\}$. Suppose that $\mathbf{f}\in
\mathcal{S}'(\mathbb{R}^{n},X)$ is such that
$\operatorname*{supp}\:\hat{\mathbf{f}}\subset [0,r\omega]$ and
$\psi\in \mathcal{S}_{0}(\mathbb{R}^{n})$ is any wavelet
satisfying
$\operatorname*{supp}\hat{\psi}\subset\mathbb{R}^{n}\setminus
[0,r\omega]$, then
\begin{equation*}
\mathcal{W}_{\psi}\mathbf{f}(x,y)=\frac{1}{(2\pi)^{n}}\left\langle
\hat{\mathbf{f}}(u),e^{ixu}\overline{\hat{\psi}}(yu)\right\rangle=0,
\ \ \ \mbox{for all } y\in(0,1).
\end{equation*}
\smallskip
\end{example}

Fortunately, we will show that the only distributions
$\mathbf{f}\in\mathcal{S}'(\mathbb{R}^{n},X)\setminus\mathcal{S}'(\mathbb{R}^{n},E)$
that may satisfy a local class estimate, with respect to a
non-degenerate wavelet, are those whose Fourier transforms are
compactly supported.

We need to introduce some terminology in order to move further on. We
will make use of weak integrals for $X$-valued functions as
defined, for example, in \cite[p. 77]{rudin}. We say that a
tempered $X$-valued distribution
$\mathbf{g}\in\mathcal{S}'(\mathbb{R}^{n},X)$ is \emph{weakly
regular} if there exists an $X$-valued function
$\tilde{\mathbf{g}}$ such that $\rho \tilde{\mathbf{g}}$ is weakly
integrable over $\mathbb{R}^{n}$ for all
$\rho\in\mathcal{S}(\mathbb{R}^{n})$ and
$$
\left\langle
\mathbf{g},\rho\right\rangle=\int_{\mathbb{R}^{n}}\rho(t)\tilde{\mathbf{g}}(t)\mathrm{d}t
\in X ,
$$
where the last integral is taken in the weak sense. We identify
$\mathbf{g}$ with $\tilde{\mathbf{g}}$, so, as usual, we write
$\mathbf{g}=\tilde{\mathbf{g}}$. 

Let us recall some facts about (vector valued) compactly supported
distributions \cite{schwartzv}. Let
$\mathbf{g}\in\mathcal{S}'(\mathbb{R}^{n},X)$ have support in
$\overline{B(0,r)}$, the closed ball of radius $r$. Then, the
following version of the Schwartz-Paley-Wiener theorem holds:
$$\mathbf{G}(z)=\left\langle \mathbf{g}(u),e^{-iz\cdot
u}\right\rangle, \ \ \ z\in \mathbb C^n,$$ is an $X$-valued entire
function which defines a weakly regular tempered distribution, and
$\mathbf{G}(\xi)=\hat{\mathbf{g}}(\xi),$ $\xi  \in \mathbb R^n$;
moreover, $\mathbf{G}$ is of weakly exponential type, i.e., for
all $\mathbf{w}^{\ast}\in X'$ one can find constants
$C_{\mathbf{w}^{\ast}}>0$ and $N_{\mathbf{w}^{\ast}}\in\mathbb{N}$
with
\begin{equation}
\label{wnwceeq11} \left|\left\langle
\mathbf{w}^{\ast},\mathbf{G}(z)\right\rangle\right|\leq
C_{\mathbf{w}^{\ast}}(1+\left|z\right|)^{N_{\mathbf{w}^{\ast}}}e^{r\left|\Im
m\:z\right|}, \ \ \ z\in\mathbb{C}^{n}.
\end{equation}
Conversely, if $\mathbf{G}$ is an $X$-valued entire function which
defines a weakly regular  tempered distribution and for all
$w^*\in X^*$ there exist $C_{w^*}>0$ and $ N_{w^*}\in\mathbb N$ such
that (\ref{wnwceeq11}) holds, then
$\hat{\mathbf{G}}=\mathbf{g}
,$ where
${\mathbf{g}}\in\mathcal{S}'(\mathbb R^n,X)$ and
$\operatorname*{supp}{\mathbf{g}}\subseteq \overline{B(0,r)}$.

The following concept for non-degenerate test functions is of much
relevance for the problem under consideration.

\begin{definition}
\label{wnwced1} Let $\varphi\in\mathcal{S}(\mathbb{R}^{n})$ be
non-degenerate. Given $\omega\in\mathbb{S}^{n-1}$, consider the
function of one variable $R_{\omega}(r)=\hat{\varphi}(r\omega)\in
C^{\infty}[0,\infty)$. We define the index of non-degenerateness
of $\varphi$ as the (finite) number
$$
\tau=\inf\left\{r\in\mathbb{R}_{+}:\:\operatorname*{supp}
R_{\omega}\cap[0,r]\neq \emptyset,
\forall\omega\in\mathbb{S}^{n-1}\right\} .
$$
\end{definition}
We are ready to state and prove the main Tauberian result of this
subsection. We shall consider slightly more general norm
estimates for the regularizing transform
$M^{\mathbf{f}}_{\varphi}$ in terms of functions
$\Psi:\mathbb{R}^{n}\times(0,1]\to\mathbb{R}_{+}$ which satisfy,
for some constants $C_1>0$ and $k,l\in\mathbb{N}$,
\begin{equation}
\label{wnweqgb}
\Psi(0,y)\leq \frac{C_1}{y^{k}}\ \ \mbox{ and }  \ \ \Psi(x+\xi,y)\leq C_1 \Psi(x,y)(1+\left|\xi\right|)^{l},
\end{equation}
for all $x,\xi\in\mathbb{R}^{n}$ and $y\in(0,1]$.
\begin{theorem}
\label{wnwceth2} Let $\mathbf{f}\in\mathcal{S}'(\mathbb{R}^{n},X)$
and let $\varphi\in\mathcal{S}(\mathbb{R}^{n})$ be a
non-degenerate test function with index of non-degenerateness
$\tau$. Assume:
\begin{enumerate}
\item[(i)] $M_{\varphi}^\mathbf{f}(x,y)$ takes values in $E$ for almost all $(x,y)\in\mathbb{R}^{n}\times(0,1]$ and is measurable as an $E$-valued function on $\mathbb{R}^{n}\times(0,1]$.
\item[(ii)] There is a function $\Psi:\mathbb{R}^{n}\times(0,1]$ that satisfies (\ref{wnweqgb}) and such that
$$
\left\|M^{\mathbf{f}}_{\varphi}(x,y)\right\|\leq \Psi(x,y), \ \ \  \mbox{for almost all } (x,y)\in\mathbb{R}^{n}\times(0,1].
$$
\end{enumerate}
Then, for any $r>\tau$, there exists an $X$-valued entire function
$\mathbf{G}$, which defines a weakly regular tempered $X$-valued distribution
and satisfies (\ref{wnwceeq11}), such that
$$\mathbf{f-G}\in\mathcal{S}'(\mathbb{R}^{n},E).$$
Furthermore, there exists $C$ such that
$$
\left\|M^{\mathbf{f-G}}_{\varphi}(x,y)\right\|\leq C \Psi(x,y), \ \ \  \mbox{for almost all } (x,y)\in\mathbb{R}^{n}\times(0,1],
$$
and we can choose $\mathbf{G}$ so that $\hat{\mathbf{G}}=\chi\hat{\mathbf{f}}$, where $\chi\in\mathcal{D}(\mathbb{R}^{n})$ is an arbitrary test function that satisfies $\chi(t)=1$ for $\left|t\right|\leq \tau$ and has support contained in the ball of radius $r$ and center at the origin.
\end{theorem}
\begin{proof}
Let $r_{1}$ be such that $\tau<r_{1}<r$. It is easy to find a reconstruction wavelet $\eta\in\mathcal{S}_{0}(\mathbb{R}^{n})$ for $\varphi$,
in the sense that
$$
1=\int^{\infty}_{0}\hat{\varphi}(r\omega)\hat{\eta}(r\omega)\frac{\mathrm{d}r}{r}, \ \ \ \mbox{for every }\omega\in\mathbb{S}^{n-1},
$$
with the
property $\operatorname*{supp}\hat{\eta}\subset B(0,r_{1})$. Indeed, if we choose a non-negative $\kappa\in\mathcal{D}(\mathbb{R}^{n})$ with support in $B(0,r_{1})\setminus\left\{0\right\}$ and being equal to $1$ in a neighborhood of the sphere $\tau \mathbb{S}^{n-1}=\left\{u\in\mathbb{R}^{n}: \left|u\right|=\tau\right\}$, then 
the same argument given in the proof of Proposition \ref{wnwp3} shows that 
$$
\hat{\eta}(x)= \frac{\kappa(x)\overline{\hat{\varphi}(x)}}
{\displaystyle\int^{\infty}_{0}\kappa(rx)\left|\hat{\varphi} (rx)\right| ^{2}\frac{\mathrm{d}r}{r}
}
$$
fulfills the requirements. The usual calculation \cite[p. 66]{holschneider} is valid and so, for $\rho\in\mathcal{S}_{0}(\mathbb{R}^{n})$,

$$\rho(t)=\int_{0}^{\infty}\int_{\mathbb{R}^{n}}
\frac{1}{y^{n}}\varphi\left(\frac{x-t}{y}\right)\mathcal{W}_{\bar{\eta}}
\rho(x,y)\frac{\mathrm{d}x\mathrm{d}y}{y}\:.$$
Observe now that if
$\operatorname*{supp}\hat{\rho}\subseteq\mathbb{R}^{n}\setminus
B(0,r_{1})$, then
$$
\mathcal{W}_{\bar{\eta}}\rho(x,y)=\frac{1}{(2\pi)^{n}}\int_{\mathbb{R}^{n}}e^{ix
\cdot u}\hat{\rho}(u)\hat{\eta}(-yu)\mathrm{d}u=0 , \ \ \
\mbox{for all } y\in[1,\infty).
$$
Thus, the same argument employed in Proposition \ref{wnwcep1}
applies to show
\begin{equation}
\label{wnwceeq12}
\left\langle f,\rho \right\rangle=
\int_{0}^{1}\int_{\mathbb{R}^{n}}M_{\varphi}^\mathbf{f}(x,y)\mathcal{W}_{\bar{\eta}}\rho(x,y)\:\frac{\mathrm{d}x\mathrm{d}y}{y}\:
,
\end{equation}
for all $\rho\in\mathcal{S}(\mathbb{R}^{n})$ with
$\operatorname*{supp}\hat{\rho}\subseteq\mathbb{R}^{n}\setminus
B(0,r_{1}).$ Choose $\chi_{1}\in C^{\infty}(\mathbb{R}^{n})$ such that
$\chi_1(u)=1$ for all $u\in\mathbb{R}^{n}\setminus B(0,r)$ and
$\operatorname*{supp}\chi_1\in\mathbb{R}^{n}\setminus B(0,r_{1})$.
Now, $\hat{\chi}_1\ast\mathbf{f}$ is well defined since
$\hat{\chi}_1\in \mathcal{O}'_{C}(\mathbb{R}^{n})$ (the space of
convolutors, cf. \cite{schwartz1}), and actually (\ref{wnwceeq12}) and the continuity of $\mathcal{W}_{\bar{\eta}}$ imply that
$(2\pi)^{-n}\hat{\chi}_1\ast\mathbf{f}\in\mathcal{S}'(\mathbb{R}^{n},E)$.
Therefore,
$\mathbf{G}=\mathbf{f}-(2\pi)^{-n}\hat{\chi}_1\ast\mathbf{f}$
satisfies the requirements because $\hat{\mathbf{G}}=\chi\mathbf{\hat{f}}$, where $\chi=1-\chi_1$, and so
$\operatorname*{supp}\hat{\mathbf{G}}\subseteq \overline{B(0,r)}$. Since $\hat{\chi}_1=(2\pi)^{n}\delta -\hat{\chi}$ and so
\begin{align*}
M^{\mathbf{f-G}}_{\varphi}(x,y)&=M^{\mathbf{f}}_{\varphi}(x,y)-\frac{1}{(2\pi)^{n}}\left\langle\mathbf{f}(\xi),\left\langle \frac{1}{y^{n}}\varphi\left(\frac{x-\xi-t}{y}\right)\hat{\chi}(t)\right\rangle \right\rangle
\\
&
=M^{\mathbf{f}}_{\varphi}(x,y)-\frac{1}{(2\pi)^{n}}\int_{\mathbb{R}^{n}}M^{\mathbf{f}}_{\varphi}(x-t,y)\hat{\chi}(t)\mathrm{d}t,
\end{align*}
we obtain the norm estimate for
 $M^{\mathbf{f-G}}_{\varphi}$ with the constant $C=C_{1}(1+(2\pi)^{-n}\int_{\mathbb{R}^{n}}(1+\left|t\right|)^{l}|\hat{\chi}(t)|\mathrm{d}t)$.\end{proof}

Clearly, if the estimate for $M^{\mathbf{f}}_{\varphi}(x,y)$ in Theorem \ref{wnwceth2} is satisfied for \emph{all} $(x,y)\in\mathbb{R}^{n}\times(0,1]$, so does the one for $M^{\mathbf{f-G}}_{\varphi}(x,y)$.

One may be tempted to think that in Theorem \ref{wnwceth2} it is possible to take $\mathbf{G}$ with support in $\overline{B(0,\tau)}$; however, this is not true, in general, as the following counterexample shows.
\begin{example}
\label{wnwceex3} Let $X$, $E$, and the sequence
$\left\{\chi_{\nu}\right\}_{\nu=1}^{\infty}$ be as in Example
\ref{wnwceex1}. We work in dimension $n=1$. We assume additionally
that $\sup_\xi \left|\chi_{\nu}(\xi)\right|=1$, for all
$\nu\in\mathbb{N}$. Let $\tau\geq0$, the wavelet $\psi$, given by
$$\hat{\psi}(u)=e^{-\left|u\right|-(1/(\left|u\right|-\tau))} \mbox{ for }
\left|u\right|>\tau \mbox{ and } \hat{\psi}(u)=0 \mbox{ for }
\left|u\right|\leq\tau,$$ has index of non-degenerateness $\tau$.
Consider the $C(\mathbb{R})$-valued distribution
$$\mathbf{f}(t,\xi)=\sum_{\nu=1}^{\infty}e^{\nu+i\left(\tau+\frac{1}{\nu}\right)t} \chi_{\nu}(\xi)\in\mathcal{S}'(\mathbb{R}_{t}, C(\mathbb{R}_{\xi}))\setminus
\mathcal{S}'(\mathbb{R}_{t}, C_{b}(\mathbb{R}_{\xi})) .$$
Then,
$$\mathcal{W}_{\psi}\mathbf{f}(x,y)(\xi)=\sum_{1\leq\nu<\frac{y}{\tau(1-y)}}e^{\nu+(ix-y)\left(\tau+\frac{1}{\nu}\right) -\frac{\nu}{y-\nu\tau(1-y)}}\chi_{\nu}(\xi) , \ \ \ 0<y<1.$$
and hence,
$$\left\|\mathcal{W}_{\psi}\mathbf{f}(x,y)\right\|_{C_{b}(\mathbb{R})}\leq
1, \ \ \  \mbox{for all } 0<y<1.$$ Therefore, the hypotheses of Theorem
\ref{wnwceth2} are fully satisfied, however,
$\mathbf{f}-\mathbf{G}\notin\mathcal{S}'(\mathbb{R},
C_{b}(\mathbb{R}))$, for any
$\mathbf{G}\in\mathcal{S}'(\mathbb{R}, C(\mathbb{R}))$ with
$\operatorname*{supp}\hat{\mathbf{G}}\subseteq[-\tau,\tau]$.
\end{example}

\subsection{The $\phi-$transform}
\label{wnwcephi} Theorem \ref{wnwceth2} can be greatly improved for the
$\phi-$transform. Observe that the index of non-degenerateness of
$\phi$ is now $\tau=0$. Remarkably, one gets a full
characterization of the space $\mathcal{S}'(\mathbb{R}^{n},E)$.
\begin{theorem}
\label{wnwceth3} Let $\mathbf{f}\in\mathcal{S}'(\mathbb{R}^{n},X)$
and let $\phi\in\mathcal{S}(\mathbb{R}^{n})$ be such that
$\mu_{0}(\phi)=1$. Necessary and sufficient conditions for
$\mathbf{f}$ to belong to the space
$\mathcal{S}'(\mathbb{R}^{n},E)$ are:
\begin{enumerate}
\item[(i)] $F_{\phi}\mathbf{f}(x,y)$ takes values in $E$ for almost all
$(x,y)\in\mathbb{R}^{n}\times(0,1]$ and is measurable as an $E$-valued function on $\mathbb{R}^{n}\times(0,1]$, and,
\item[(ii)] There exist constants $k,l\in\mathbb{N}$ and $C>0$ such that
\begin{equation*}
\left\|F_{\phi}\mathbf{f}(x,y)\right\|\leq
C\frac{\left(1+\left|x\right|\right)^{l}}{y^{k}}, \ \ \mbox{for
almost all }(x,y)\in\mathbb{R}^{n}\times(0,1].
\end{equation*}
\end{enumerate}
\end{theorem}
\begin{proof} By Theorem \ref{wnwceth2}, one may assume that $\operatorname*{supp}\hat{\mathbf{f}}
\subseteq\overline{B(0,1)}$. So, $\mathbf{f}$ is given by an
entire function of weakly exponential type which defines a weakly
regular $X$-valued tempered distribution. Let $\rho\in
\mathcal{S}(\mathbb{R}^{n})$ such that $\rho(u)=1$ for $u\in
B(0,3/2)$ and $\operatorname*{supp}\rho\subset B(0,2)$. Choose
$2\sigma<1$ such that $|\hat{\phi}(u)|>0$ for all $u\in
\overline{B(0,2\sigma)}$. For a fixed $x\in\mathbb R^n$, the
function 
$$
\hat{\chi}_{x}(u)=e^{ixu}\frac{\rho(u)}{\hat{\phi}(-\sigma
u)}
$$ defines an element of $\mathcal{S}(\mathbb{R}^{n})$. Thus,
\begin{align*}
\mathbf{f}(x)&=\frac{1}{(2\pi)^{n}}\left\langle\hat{\mathbf{f}}(u),e^{ix\cdot
u} \right\rangle
\\
&=\frac{1}{(2\pi)^{n}}\left\langle
\hat{\mathbf{f}}(u),\hat{\chi}_{x}(u)\hat{\phi}(-\sigma
u)\right\rangle
\\
& =\frac{1}{\sigma^{n}}\left\langle \mathbf{f}(t),
\int_{\mathbb{R}^{n}}\chi_{x}(\xi)\phi\left(\frac{t+\xi}{\sigma}\right)
\mathrm{d}\xi\right\rangle
\\
& 
=\int_{\mathbb{R}^{n}}\chi_{x}(\xi)
F_{\phi}\mathbf{f}(-\xi,\sigma)\mathrm{d}\xi \in E,
\end{align*}
where the exchange with the integral sign can be established as in
the proof of Proposition \ref{wnwcep1}. Hence the entire \emph{function}
$\mathbf{f}$ takes values in $E$.
Moreover,
\begin{align*}
\left\|\mathbf{f}(x)\right\|&< \frac{C}{\sigma^{k}}
\int_{\mathbb{R}^{n}}(1+\left|\xi\right|)^{l}\left|\chi_{x}(\xi)\right|\mathrm{d}\xi
\\
&
\leq C_{1}(1+\left|x\right|)^{N},
\ \ \ \mbox{for all }x\in\mathbb{R}^{n},
\end{align*} for some constants $C_{1}>0$ and
$N\in\mathbb{N}$. Clearly, the last $E$-norm estimate over the
growth of $\mathbf{f}$ implies that
$\mathbf{f}\in\mathcal{S}'(\mathbb{R}^{n},E)$, as required.
\end{proof}



\subsection{Strongly Non-degenerate Wavelets} \label{wnwces}
A strengthened version of both Theorem \ref{wnwceth1} and Theorem \ref{wnwceth2}
holds if we restrict the non-degenerate wavelets to those fulfilling the requirements of the following definition.
\begin{definition}
\label{wnwced2} Let $\psi\in\mathcal{S}(\mathbb{R}^{n})$ be a wavelet. We call $\psi$
strongly non-degenerate if there exist constants $N\in\mathbb{N}$,
$r>0$, and $C>0$ such that
\begin{equation}
\label{wnwceeq13}
C\left|u\right|^{N}\leq |\hat{\psi}(u)|\ , \ \ \ \mbox{for all
}\left|u\right|\leq r.
\end{equation}
\end{definition}

\begin{theorem}
\label{wnwceth4} Let $\mathbf{f}\in\mathcal{S}'(\mathbb{R}^{n},X)$
and let $\psi\in\mathcal{S}(\mathbb{R}^{n})$ be a
strongly non-degenerate wavelet. Assume:
\begin{enumerate}
\item[(i)] $\mathcal{W}_{\psi}\mathbf{f}(x,y)$ takes values in $E$
for almost all $(x,y)\in\mathbb{R}^{n}\times(0,1]$ and is measurable as an $E$-valued function on $\mathbb{R}^{n}\times(0,1]$.
\item[(ii)] There exist constants $k,l\in\mathbb{N}$ and $C>0$ such that
\begin{equation*}
\left\|\mathcal{W}_{\psi}\mathbf{f}(x,y)\right\|\leq
C\frac{\left(1+\left|x\right|\right)^{l}}{y^{k}}, \ \ \mbox{for
almost all }(x,y)\in\mathbb{R}^{n}\times(0,1].
\end{equation*}
\end{enumerate}
Then, there exists $\mathbf{G}\in \mathcal{S}'(\mathbb{R}^{n},X)$ such
that $\mathbf{f}-\mathbf{G}\in\mathcal{S}'(\mathbb{R}^{n},E)$ and
$\operatorname*{supp} \hat{\mathbf{G}}\subseteq\left\{0\right\}$.
\end{theorem}
\begin{proof}
As in Theorem \ref{wnwceth3}, we may assume that
$\operatorname*{supp}\hat{\mathbf{f}} \subseteq\overline{B(0,1)}$.
Let $\rho\in\mathcal{S}(\mathbb{R}^{n})$ be the same as in the
proof of Theorem \ref{wnwceth3}. We can find $\sigma,C_1>0$ and
$N\in\mathbb{N}$ such that $2\sigma\leq 1$ and
$C_1|u|^{N}\leq|\hat{{\psi}}(u)|$, for all $u\in
\overline{B(0,2\sigma)}$. Given
$\eta\in\mathcal{S}_{0}(\mathbb{R}^{n})$, then
$$
\hat{\chi}(u)=\hat\chi_{\eta}(u)=\rho(u)\frac{\hat{\eta}(-u)}{\overline{\hat{\psi}}(\sigma
u)}
$$ defines an element of $\mathcal{S}(\mathbb{R}^{n})$ in a
continuous fashion, consequently, the mapping
$\gamma:\mathcal{S}_{0}(\mathbb{R}^{n})\mapsto [0,\infty)$ given
by
$$\gamma(\eta)=\int_{\mathbb{R}^{n}}(1+\left|\xi\right|)^{l}\left|\chi(\xi)\right|\mathrm{d}\xi$$
is a continuous seminorm over $\mathcal{S}_{0}(\mathbb{R}^{n})$.
Now, for any $\eta\in\mathcal{S}_{0}(\mathbb{R}^{n})$,
\begin{align*}
\left\langle \mathbf{f},\eta\right\rangle & =\frac{1}{(2\pi)^{n}}\left\langle \hat{\mathbf{f}}(u),\hat{\chi}(u)\overline{\hat{\psi}}(\sigma u)\right\rangle
\\
&
=\int_{\mathbb{R}^{n}}\chi(\xi)\mathcal{W}_{\psi}\mathbf{f}(-\xi,\sigma)\mathrm{d}\xi.
\end{align*}
Therefore, 
$$\left\|\left\langle \mathbf{f},\eta\right\rangle\right\|\leq (C/\sigma^{k})\gamma(\eta), \ \ \ \mbox{for all } \eta\in\mathcal{S}_{0}(\mathbb{R}^{n}),$$
and the latter implies that the restriction of $\mathbf{f}$ to $\mathcal{S}_{0}(\mathbb{R}^{n})$ belongs to $\mathcal{S}'_{0}(\mathbb{R}^{n},E)$. The standard argument (see the proof of Theorem \ref{wnwceth1}) yields the existence of $\mathbf{G}$ satisfying all the requirements.


\end{proof}
It should be noticed that the class of strongly non-degenerate wavelets coincides with that of Drozhzhinov-Zavialov wavelets, introduced in Example \ref{wnwex3.3}. Indeed, the condition (\ref{wnwceeq13}) holds if and only if the Taylor polynomial of order $N$ at the origin of $\hat{\psi}$  is non-degenerate in the sense of Example \ref{wnwex3.3}.

In dimension $n=1$, there is no distinction between non-degenerateness and strong non-degenerateness, whenever we consider wavelets in $\mathcal{S}(\mathbb{R})\setminus\mathcal{S}_{0}(\mathbb{R})$. Actually, a stronger result than Theorem \ref{wnwceth4} holds in the one-dimensional case.

\begin{proposition}
\label{wnwcep3}
Let $\mathbf{f}\in\mathcal{S}'(\mathbb{R},X)$
and let $\psi\in\mathcal{S}(\mathbb{R})$ be a
wavelet with $\mu_{d}(\psi)\neq0$, for some $d\in\mathbb{N}$. If the conditions (i) and (ii) of Theorem \ref{wnwceth4} are satisfied, then there exists an $X$-valued polynomial $\mathbf{P}$ of degree at most $d-1$ such that $\mathbf{f}-\mathbf{P}\in\mathcal{S}'(\mathbb{R},E)$.
\end{proposition}
\begin{proof} There exists $\phi\in\mathcal{S}(\mathbb{R})$ such that $\overline{\phi^{(d)}}=(-1)^{d}\psi$, and we may assume that $\mu_{0}(\phi)=1$. Then,
$$
F_{\phi}(\mathbf{f}^{(d)})(x,y)=\frac{1}{y^{d}}\mathcal{W}_{\psi}\mathbf{f}(x,y).
$$
Hence, an application of Theorem \ref{wnwceth3} gives that $\mathbf{f}^{(d)}\in\mathcal{S}'(\mathbb{R},E)$, and this clearly implies the existence of $\mathbf{P}$ with the desired properties.
\end{proof}

Observe that the conclusion of Proposition \ref{wnwcep3} does not hold for multidimensional wavelets, in general, even if they are strongly non-degenerate. This fact is shown by Example \ref{wnwceex1}.

Naturally, if $X$ is a normed space in Theorem \ref{wnwceth4}, then $\mathbf{G}$ must be an $X$-valued polynomial, this fact is stated in the next corollary. Corollary \ref{wnwcec2} extends an important result of Drozhzhinov and Zavialov \cite[Thm. 2.1]{drozhzhinov-z3}.

\begin{corollary}
\label{wnwcec2} Let the hypotheses of Theorem \ref{wnwceth4} be satisfied. If $X$ is a normed space, then there is an
$X$-valued polynomial $\mathbf{P}$ such that
$\mathbf{f-P}\in\mathcal{S}'(\mathbb{R}^{n},E)$. Moreover, if $P_{q}$, $q\in\mathbb{N}$, are the homogeneous terms
of the Taylor polynomials of $\hat{\psi}$ at the origin (cf.
(\ref{wnwtteq7})), then $\overline{P_{q}}\left(\partial/\partial t\right)\mathbf{P}$ is an $E$-valued polynomial, for each $q\in\mathbb{N}$.
\end{corollary}
\begin{proof} The existence of the polynomial is clear. The proof of the remaining assertion is identically the same as that of Corollary \ref{wnwcec1}.
\end{proof}

In general, the degree of the the polynomial $\mathbf{P}$ occurring in Corollary \ref{wnwcec2} depends merely on $\mathbf{f}$, and not on the wavelet. However, when the Taylor polynomials of the wavelet $\hat{\psi}$ posses a rich algebraic structure, it is possible to say more about the degree of $\mathbf{P}$. This fact was already observed in \cite[Thm. 2.2]{drozhzhinov-z3} for Banach spaces $X$. We denote by $\mathbb{P}_{d}(\mathbb{R}^{n})$ the ideal of (scalar-valued) polynomials of the form $Q(t)=\sum_{d\leq\left|m\right|\leq N}a_{m}t^{m}$, for some $N\in\mathbb{N}$.

\begin{corollary}
\label{wnwcec3} Let the hypotheses of Corollary \ref{wnwcec2} be satisfied. If there exists $d\in\mathbb{N}$ such that $\mathbb{P}_{d}(\mathbb{R}^{n})$ is contained in the ideal generated by the polynomials $P_1,P_2,\dots,P_{d}$, where the $P_{q}$, $q\in\mathbb{N}$, are the homogeneous terms
of the Taylor polynomials of $\hat{\psi}$ at the origin (cf.
(\ref{wnwtteq7})), then there exists an $X$-valued polynomial $\mathbf{P}$ of degree at most $d-1$ such that $\mathbf{f}-\mathbf{P}\in\mathcal{S}'(\mathbb{R}^{n},E)$.
\end{corollary}
\begin{proof}
Corollary \ref{wnwcec2} yields the existence of an $X$-valued
polynomial $\tilde{\mathbf{P}}(t)=\mathbf{P}(t)+
\sum_{d\leq\left|m\right|\leq N}\mathbf{w}_{m}t^{m}$ such that
$\mathbf{f}-\tilde{\mathbf{P}}\in \mathcal{S}'(\mathbb{R}^{n},E)$
and $\mathbf{P}$ has degree at most $d-1$. Then, we must show that
$\mathbf{w}_{m}\in E,$ for $d\leq\left|m\right|\leq N.$ But
Corollary \ref{wnwcec2} also implies that
$\overline{P_{q}(\partial/\partial t)}\tilde{\mathbf{P}}$ is an
$E$-valued polynomial for $q=1,\dots,d$, and since
$\mathbb{P}_{d}(\mathbb{R}^{n})$ is also contained in the ideal
generated by $\overline{P}_{1},\dots,\overline{P}_{d}$, we obtain
at once that $$\mathbf{w}_{m}=m!((\partial^{|m|}/\partial
t^{m})\tilde{\mathbf{P}})(0)\in E,\ \ \ \mbox{for }
d\leq\left|m\right|\leq N.$$
\end{proof}

\newpage

\section{Regularity and Asymptotic Properties of Functions and Distributions - Several Applications}
\label{wnwap}
In this section we illustrate our ideas with several applications and examples. In Subsection \ref{wnwla} we give applications of our Tauberian theory in the analysis of pointwise and global regularity properties of distributions, where we extend or recovered a number of classical results due to Meyer, Jaffard, and Holschneider \cite{holschneider94,jaff,meyer}.
We study in Subsection \ref{wnwPDE} sufficient conditions for stabilization in time of the solution to the Cauchy problem discussed in Example \ref{wnwex3.4}. Subsection \ref{wnwCA} deals with applications to regularity theory in the setting of generalized function algebras. In Subsection \ref{wnwds}, we provide a wavelet characterization of the distributionally small distributions at infinity (cf. Example \ref{wnwex2.3}). Section \ref{wnwWRfunction} is devoted to the pointwise analysis of Riemann type distributions at the rationals. We obtain Tauberian theorems for Laplace transforms in Subsection \ref{wnwLT}. Finally, we show in Subsection \ref{wnwDS} that for tempered $E$-valued distributions the weak-asymptotics in $\mathcal{D}'(\mathbb{R}^{n},E)$ and $\mathcal{S}'(\mathbb{R}^{n},E)$ are equivalent.
\subsection{Applications to Pointwise and Local Analysis of Distributions}
\label{wnwla} In \cite{meyer} Meyer has proposed the use of
weak scaling exponents at small scale in the pointwise analysis of
functions and distributions (cf. Example \ref{wnwex2.0}).
Moreover, he introduced some useful pointwise spaces of
distributions in order to measure such scaling properties. They
are natural generalizations of the classical pointwise H\"{o}lder
spaces $C^{\alpha}(x_{0})$. Recall a function $f$ belongs to
$C^{\alpha}(x_{0})$, $\alpha>0$, if there exists a polynomial $P$
such that
$$\left|f(x_{0}+h)-P(h)\right|=O(\left|h\right|^{\alpha})$$ as
$h\to 0$. It is also worth mentioning that \cite{meyer} Meyer's spaces are
unions of the (local) 2-microlocal spaces, introduced by Bony
\cite{bony1,bony2,jaffard-m,meyer} and so useful in the study of
non-linear PDE.

We shall extend in this subsection the spaces of Meyer. Furthermore, we will consider $E$-valued distributions and make use of slowly varying functions for our scaling measurements. Interestingly, despite the fact that our spaces are essentially \emph{pointwisely} defined, they are, in turn, effective tools in the study of \emph{global} regularity as well, as we show below with several applications.

\begin{definition}
Let $\mathbf{f}\in\mathcal{S}'(\mathbb{R}^{n},E)$ and let $L$ be slowly varying at the origin. For $x_{0}\in\mathbb{R}^{n}$ and $\alpha\in\mathbb{R}$, we write:
\begin{itemize}
\item [(i)] $\mathbf{f}\in \mathcal{O}^{\alpha,L}(x_{0},E)$ if $\mathbf{f}$ is weak-asymptotically bounded of degree $\alpha$ at $x_{0}$ with respect to $L$ (in $\mathcal{S}'(\mathbb{R}^{n},E)$).
\item [(ii)] $\mathbf{f}\in C_{w}^{\alpha,L}(x_{0},E)$ if there is an $E$-valued polynomial $\mathbf{P}$ such that $\mathbf{f-P}\in\mathcal{O}^{\alpha,L}(x_{0},E)$.
\item [(iii)] $\mathbf{f}\in C_{\ast,w}^{\alpha,L}(x_{0},E)$ if $\mathbf{f}$ is weak-asymptotically bounded of degree $\alpha$ at $x_{0}$ with respect to $L$ in the space $\mathcal{S}'_{0}(\mathbb{R}^{n},E)$.
\end{itemize}
\end{definition}

We call $C_{w}^{\alpha,L}(x_{0},E)$ the \emph{pointwise weak H\"{o}lder space} of exponent $\alpha$, with respect to $L$. When dealing with scalar-valued distributions, we simply write $\mathcal{O}^{\alpha,L}(x_{0})$, $C_{w}^{\alpha,L}(x_{0})$ and $C_{\ast,w}^{\alpha,L}(x_{0})$ for these spaces; furthermore, if $L\equiv1$, we then write $C_{w}^{\alpha}(x_{0},E)=C_{w}^{\alpha,L}(x_{0},E)$, and similarly for the other two spaces. It should be mentioned that $\mathcal{O}^{\alpha}(x_{0})$ coincides with the one introduced by Meyer in \cite[p. 13] {meyer}, while $C_{\ast,w}^{\alpha}(x_{0})=\Gamma^{\alpha}(x_{0})$ in Meyer's notation.

Let us discuss some properties of our pointwise spaces. First of all, we obviously have: $$\mathcal{O}^{\alpha,L}(x_{0},E)\subseteq C_{w}^{\alpha,L}(x_{0},E)\subseteq C_{\ast,w}^{\alpha,L}(x_{0},E).$$
A more precise inclusion relation is obtained if we employ Proposition A.2, this is the content of the next proposition.
\begin{proposition}\label{wnwp9.1} If $\alpha\notin\mathbb{N}$, we then have $C_{w}^{\alpha,L}(x_{0},E)=C_{\ast,w}^{\alpha,L}(x_{0},E)$. When $\alpha<0$, we have $\mathcal{O}^{\alpha,L}(x_{0},E)=C_{w}^{\alpha,L}(x_{0},E)=C_{\ast,w}^{\alpha,L}(x_{0},E)$. If $\alpha=p\in\mathbb{N}$, then $\mathbf{f}\in C_{\ast,w}^{p,L}(x_{0},E)$ if and only if $\mathbf{f}$ has a weak-asymptotic expansion of the form (\ref{wnweqA7}) at the point $x_{0}$.
\end{proposition}

Therefore, when $\alpha\in\mathbb{N}$, the difference between $C_{\ast,w}^{p,L}(x_{0},E)$ and $C_{w}^{\alpha,L}(x_{0},E)$ lies in the occurrence of associate asymptotically homogeneously bounded functions in the expansion (\ref{wnweqA7}). Let $\mathbf{c}:(0,A]\to E$ be asymptotically homogeneously bounded of degree 0 (at the origin) with respect to $L\equiv1$ (cf. Definition \ref{wnwd5}). Since $\mathbf{c}(a\varepsilon)=\mathbf{c}(\varepsilon)+O(1)$ holds uniformly for $a$ in compacts (see for instance \cite[Chap. 10]{vindas2}), it is not difficult to prove that $\left\|\mathbf{c}(\varepsilon)\right\|=O(\log(1/\varepsilon))$. This bound and Proposition \ref{wnwp9.1} immediately show the ensuing inclusion relation.

\begin{corollary}\label{wnwfec1} Let $p\in\mathbb{N}$. Then, $C^{p}_{\ast,w}(x_{0},E)\subsetneq C^{p,\left|\log\right|}_{w}(x_{0},E)$.
\end{corollary}

Theorem \ref{wnwth8} yields the characterization of $ C_{\ast,w}^{\alpha,L}(x_{0},E)$ through the wavelet transform. We rephrase this result in the following proposition.

\begin{proposition}
\label{wnwp9.2} Let $\alpha\in\mathbb{R}$ and let $\psi\in\mathcal{S}(\mathbb{R}^{n})$ be a non-degenerate wavelet  such that its moments  satisfy $\mu_{m}(\psi)=0$ for every $\left|m\right|\leq[\alpha]$. Then, $\mathbf{f}\in C_{\ast,w}^{\alpha,L}(x_{0},E)$ if and only if there is $k\in\mathbb{N}$ such that (\ref{wnweq5.5}) holds.
\end{proposition}

We can characterize $ \mathcal{O}^{\alpha,L}(x_{0},E)$ as well, but we should then employ $\phi-$transforms: $\mathbf{f}\in \mathcal{O}^{\alpha,L}(x_{0},E)$ if and only if (\ref{wnwtteq1}) is satisfied for some $k\in\mathbb{N}$. In the one-dimensional case, it is also possible to characterize $C_{w}^{\alpha,L}(x_{0},E)$, $\alpha\in\mathbb{N}$, via an estimate of the form (\ref{wnweq5.5}), this time the wavelet must satisfy $\mu_{0}(\psi)=\dots=\mu_{\alpha-1}(\psi)=0$ and $\mu_{\alpha}(\psi)\neq0$; however, such a characterization is no longer possible in the multidimensional case.

It is worth pointing out that if $\mathbf{f}\in C_{\ast,w}^{\alpha,L}(x_{0},E)$ and $\alpha>0$, then the \L ojasiewicz point values (cf. Example \ref{wnwex2.1}) $\mathbf{f}^{(m)}(x_{0})$ exist, distributionally, for all $\left|m\right|<\alpha$. Moreover, $\mathbf{f}\in C_{w}^{\alpha,L}(x_{0},E)$ if and only if the following ``Taylor formula'' holds,
$$
\mathbf{f}(t)-\sum_{\left|m\right|<\alpha} \frac{\mathbf{f}^{(m)}(x_{0})}{m!}(t-x_{0})^{m}\in \mathcal{O}^{\alpha,L}(x_{0},E).
$$

We now move to applications. The next result characterizes global regularity of distributions in terms of the wavelet transform, it generalizes the well known wavelet criterion for global H\"{o}lder continuity \cite{holschneider-t,hormander,stein}. Let $\alpha\in\mathbb{R}_{+}\setminus\mathbb{N}$ and let $L$ be a slowly varying function such that $L$ and $1/L$ are locally bounded on $(0,1]$. We say that $f$ belongs to $C^{\alpha,L}(\mathbb{R}^{n})$ if $f\in C^{[\alpha]}(\mathbb{R}^{n})$ and
$$
\left\|f\right\|_{\alpha,L}:=\sum_{\left|j\right|\leq [\alpha]} \sup_{t\in\mathbb{R}^{n}}\left|f^{(j)}(t)\right|+ \sum_{\left|m\right|=[\alpha]}\sup_{0<\left|t-x\right|\leq 1} \frac{\left|f^{(m)}(t)-f^{(m)}(x)\right|}{\left|t-x\right|^{\alpha-[\alpha]}L(\left|t-x\right|)}<\infty.
$$
The conditions imposed over $L$ ensure that $C^{\alpha,L}(\mathbb{R}^{n})$ depends only on the behavior of $L$ near 0; therefore, it is \emph{invariant} under dilations. When $L\equiv1$, this space reduces to $C^{\alpha,L}(\mathbb{R}^{n})=C^{\alpha}(\mathbb{R}^{n})$, the usual global H\"{o}lder space \cite{hormander,meyer}. Consequently, we call $C^{\alpha,L}(\mathbb{R}^{n})$ the \emph{global H\"{o}lder space with respect to }$L$. Note also the following inclusion relations:
$$C^{\beta}(\mathbb{R}^{n})\subset C^{\alpha,L}(\mathbb{R}^{n})\subset C^{\gamma}(\mathbb{R}^{n}), \ \ \ \mbox{whenever }0<\gamma<\alpha<\beta.$$

Observe that the essential technique in the proof of the following theorem is to turn pointwise weak regularity for vector-valued distributions into global information for scalar-valued distributions.

\begin{theorem}
\label{wnwfeth9.6} Let $\alpha\in\mathbb{R}_{+}\setminus\mathbb{N}$ and let $\psi\in\mathcal{S}(\mathbb{R}^{n})$ be a non-degenerate wavelet such that its moments $\mu_{m}(\psi)=0$ for every $\left|m\right|\leq[\alpha]$. Assume that $f\in\mathcal{S}'(\mathbb{R}^{n})$ satisfies
\begin{equation}
\label{wnweq9.4}\left|\mathcal{W}_{\psi}f(x,y)\right|\leq Cy^{\alpha}L(y) \  \ \ \mbox{for all }(x,y)\in\mathbb{R}^{n}\times (0,y_0],
\end{equation}
for some constants $C>0$ and $0<y_0\leq1$, where $L\in L^{\infty}_{loc}(0,1]$ is slowly varying at the origin and $1/L\in L^{\infty}_{loc}(0,1]$. Then, there exists an entire function of exponential type $G$ such that
$$f-G\in C^{\alpha,L}(\mathbb{R}^{n}).$$
Moreover, if $\tau$ is the index of non-degenerateness of $\psi$ and $r>\tau$, then $G$ can be chosen so that, for some constants $C_1$ and $N\in\mathbb{N}$,
\begin{equation}
\label{wnweq9.5} \left|G(z)\right|\leq C_1 (1+\left|z\right|)^{N}e^{\frac{r}{y_{0}}\left|\Im m\: z\right|}, \ \ \ z\in\mathbb{C}^{n}.
\end{equation}
\begin{proof} By rescaling, we may assume that $y_{0}=1$. Define the vector-valued distribution $\mathbf{f}\in\mathcal{S}'(\mathbb{R}_{t}^{n},C(\mathbb{R}_{\xi}^{n}))$ as $\mathbf{f}(t)(\xi)=f(t+\xi)$, i.e.,
$$\left\langle \mathbf{f}(t),\varphi(t)\right\rangle(\xi)=\left\langle f(t+\xi),\varphi(t)\right\rangle=(f\ast\check{\varphi})(\xi).
$$
Observe its wavelet transform
$\mathcal{W}_{\psi}\mathbf{f}(x,y)\in C(\mathbb{R}_{\xi}^{n})$ is
given by
$$\mathcal{W}_{\psi}\mathbf{f}(x,y)(\xi)=\mathcal{W}_{\psi}f(\xi+x,y).$$
By (\ref{wnweq9.4}) and Theorem \ref{wnwceth2}, applied with
$E=C_{b}(\mathbb{R}^{n})$, the Banach space of bounded continuous
functions, there exists an entire function
$G\in\mathcal{S}'(\mathbb{R}_{t}^{n})$,
which satisfies (\ref{wnweq9.5}), such that
$$\mathbf{g}(t)(\xi):=g(t+\xi)\in\mathcal{S}'(\mathbb{R}^{n}_{t},C_{b}(\mathbb{R}^{n}_{\xi})),$$
where $g=f-G$. Furthermore, Theorem \ref{wnwceth2} also yields
that for $(x,y)\in \mathbb{R}^{n}\times (0,1]$,
$$
\left\|\mathcal{W}_{\psi}\mathbf{g}(x,y)\right\|_{C_{b}(\mathbb{R}^{n})}
\leq C_{2} y^{\alpha}L(y),
$$
for some constant $C_2>0$, and in particular,
$$
\limsup_{\varepsilon\to0^{+}}\sup_{(x,y)\in (0,1]\times\mathbb{R}^{n}}\frac{y^{\sigma}}{\varepsilon^{\alpha}L(\varepsilon )}\left\|\mathcal{W}_{\psi}\mathbf{g}(\varepsilon x,\varepsilon y)\right\|_{C_{b}(\mathbb{R}^{n})}<\infty,
$$
for any $\sigma>0$. Employing now Proposition \ref{wnwp9.2}, we
conclude that $\mathbf{g}\in
C_{w}^{\alpha,L}(0,C_{b}(\mathbb{R}^n))$. Thus, the \L ojasiewicz
point value at the origin $\mathbf{g}^{(m)}(0)=v_{m}\in
C_{b}(\mathbb{{R}}^{n})$ exists, distributionally, for each
$\left|m\right|\leq[\alpha]$. If we fix $\phi\in
\mathcal{S}(\mathbb{R}^{n})$ with $\mu_0(\phi)=1$, the definition
of distributional point values tells us that, for each
$\left|m\right|\leq [\alpha]$,
$\lim_{\varepsilon\to0^{+}}g^{(m)}\ast
\check{\phi}_{\varepsilon}=\lim_{\varepsilon\to0^{+}}\left\langle
\mathbf{g}^{(m)}(\varepsilon t),\phi(t)\right\rangle= v_{m}$ in
$C_{b}(\mathbb{R}^{n})$, consequently $g^{(m)}=v_{m}$, is a
bounded continuous function. Finally, let
$\left|m\right|=[\alpha]$, observe that
$\mathbf{g}^{(m)}(t)(\xi)=g^{(m)}(t+\xi)\in
C^{\alpha-[\alpha],L}_{w}(0,C_{b}(\mathbb{R}_{\xi}^{n}))$, i.e.,
$$\mu_{0}(\varphi)g^{(m)}(\xi)-\int_{\mathbb{R}^{n}}g^{(m)}(\xi+\varepsilon
t)\varphi(t)\mathrm{d}t=O(\varepsilon^{\alpha-[\alpha]}L(\varepsilon))$$
in $C_{b}(\mathbb{R}_{\xi}^{n})$, for each test function
$\varphi\in\mathcal{S}(\mathbb{R}^{n})$. Hence, if
$0<\left|h\right|\leq1$, $\phi$ is as before ($\mu_{0}(\phi)=1$),
and we use the fact that
$\left\{\phi-\phi(\:\cdot\:-\omega):\omega\in\mathbb{S}^{n-1}\right\}$
is compact in $\mathcal{S}(\mathbb{R}^{n})$, we have
\begin{align*}
&\sup_{\xi\in\mathbb{R}^{n}}\left|g^{(m)}(\xi+h)-g^{(m)}(\xi)\right|
\leq2\sup_{\xi\in\mathbb{R}^{n}}\left|g^{(m)}(\xi)-\int_{\mathbb{R}^{n}} g^{(m)}(\left|h\right|t+\xi)\phi(t)\mathrm{d}t\right|
\\
&+\sup_{\xi\in\mathbb{R}^{n}}\left|\int_{\mathbb{R}^{n}}g^{(m)}(\xi+\left|h\right|t)(\phi(t)-\phi\left(t-|h\right|^{-1}h))\mathrm{d}t\right|
=O(\left|h\right|^{\alpha-[\alpha]}L(\left|h\right|)),
\end{align*}
and this completes the proof.
\end{proof}
\end{theorem}

As a second application of our pointwise spaces, we obtain some Tauberian criteria which allow us to compute pointwise H\"{o}lder exponents in terms of size estimates on the wavelet transform. The first part of the next theorem is originally due to Jaffard \cite{jaff}.

\begin{theorem}\label{wnw2mth1} Let $\alpha>0$ and let $\psi\in\mathcal{S}(\mathbb{R}^{n})$ be a non-degenerate wavelet with moments $\mu_{m}(\psi)=0$ for every $\left|m\right|\leq[\alpha]$. Assume that $f\in\mathcal{S}'(\mathbb{R}^{n})$ satisfies
\begin{equation}
\label{wnweqpw9.3}
\left|\mathcal{W}_{\psi}f(x_{0}+x,y)\right|\leq C(\left|x\right|+y)^{\alpha}, \ \ \ \mbox{for all } \left|x\right|+y<\sigma,
\end{equation}
for some constants $C,\sigma>0$.
\begin{itemize}
\item [(i)] If there are constants $C_{1},\gamma>0$ and $y_{0}\in(0,1)$ such that
\begin{equation}
\label{wnweqpw9.4}
\left|\mathcal{W}_{\psi}f(x_{0}+x,y)\right|\leq C_{1}y^{\gamma}, \ \ \ \mbox{for all } (x,y)\in \mathbb{R}^{n}\times(0,y_{0}],
\end{equation}
then there exists a polynomial $P$ such that
\begin{equation*}
f(x_{0}+t)=P(t)+O\left(\left|t\right|^{\alpha}\log(1/\left|t\right|)\right) \ \ \ \mbox{as } t\to0,
\end{equation*}
in the ordinary sense. In particular, $f\in C^{\alpha_{1}}(x_{0})$ for all $\alpha_{1}<\alpha$.

\item [(ii)] If there are constants, $N>0,$ $C_{2}>0$ and $y_{0}\in(0,1)$such that
\begin{equation}
\label{wnweqpw9.6}
\left|\mathcal{W}_{\psi}f(x_{0}+x,y)\right|\leq \frac{C_{2}}{(\log(1/y))^{N+1}}, \ \ \ \mbox{for all } (x,y)\in \mathbb{R}^{n}\times(0,y_{0}],
\end{equation}
then $f\in C(\mathbb{R}^{n})$ and $f\in C^{\alpha-\frac{\alpha}{N}}(x_{0})$.
\end{itemize}
\end{theorem}
We shall use a continuous Littlewood-Paley decomposition of the unity in the proof of Theorem \ref{wnw2mth1}. Let $\phi_1\in\mathcal{S}(\mathbb{R}^{n})$ be radial such that $\hat{\phi}_1$ is real-valued, $\hat{\phi}_{1}(u)=0$ for $\left|u\right|\geq 1$ and $\hat{\phi}_{1}(u)=1$ for $\left|u\right|\leq 1/2$. Furthermore, define $\psi_{1}\in \mathcal{S}_{0}(\mathbb{R}^{n})$ so that 
$$
\hat{\psi}_{1}(u)=-\displaystyle \frac{d}{dr}\hat{\phi}_{1}(ru)\left.\right|_{r=1}=- u\cdot \nabla \hat{\phi}_{1}(u).
$$ 
Then, for each fix $b>0$, in the sense of convergence in $\mathcal{S}'(\mathbb{R}^{n})$,
\begin{align*}
F_{\phi_1}f(\:\cdot\:,b)+\int_{0}^{b}\mathcal{W}_{\psi_{1}}f(\:\cdot\:,r)\frac{\mathrm{d}r}{r}
:&=\lim_{a\to0^{+}}F_{\phi_1}f(\:\cdot\:,b)+\int_{a}^{b}\mathcal{W}_{\psi_{1}}f(\:\cdot\:,r)\frac{\mathrm{d}r}{r}
\\
&=\lim_{a\to0^{+}} F_{\phi_1}f(\:\cdot\:,a),
\end{align*}
which is precisely the meaning of the formula
\begin{equation}
\label{wnweqpw9.7}
f(x)=F_{\phi_1}f(x,b)+\int_{0}^{b}\mathcal{W}_{\psi_{1}}f(x,r)\frac{\mathrm{d}r}{r}, \ \ \ \mbox{in }\mathcal{S}'(\mathbb{R}^{n}_{x}).
\end{equation}
\begin{proof}We may assume $x_{0}=0$. Let $\phi_1$ and $\psi_1$ be as defined above.
Observe first that (\ref{wnweqpw9.3}) gives at once $f\in
C_{w}^{\alpha,\left|\log\right|}(x_{0})$, by Proposition
\ref{wnwp9.2} and Corollary \ref{wnwfec1} (indeed, $f\in
C_{w}^{\alpha}(x_{0})$ if $\alpha\notin\mathbb{N}$). So, for some
polynomial $P$,
$$F_{\phi_1}f(x,\left|x\right|)=F_{\phi_1}P(x,\left|x\right|)+O(\left|x\right|^{\alpha}\log(1/\left|x\right|)),
\ \ \ \mbox{i.e.},$$
\begin{equation}
\label{wnweqpw9.8}
F_{\phi_1}f(x,\left|x\right|)=P(x)+O(\left|x\right|^{\alpha}\log(1/\left|x\right|)) \ \ \ \mbox{as }x\to0;
\end{equation}
as follows directly from the Abelian result (cf. Proposition \ref{wnwth2}).

(i) We show that we may assume $\psi=\psi_1$. In view of  Theorem \ref{wnwfeth9.6}, we may assume that $f\in C^{\gamma}(\mathbb{R}^{n})$; otherwise, we subtract an entire function and replace $f$ by the so newly obtained function. Thus, the assumption  $f\in C^{\gamma}(\mathbb{R}^{n})$ implies that we can suppose (\ref{wnweqpw9.3}) to hold actually for all $(x,y)\in\mathbb{H}^{n+1}$ and $\psi=\psi_1$ in (\ref{wnweqpw9.4}). If we consider $\mathbf{f}(t)(\xi,s)=f(\xi+st)\in\mathcal{S}'(\mathbb{R}_{t}^{n},C(\mathbb{H}^{n+1}_{(\xi,s)}))$, Theorem \ref{wnwceth1}, applied with the Banach space $E\subset C(\mathbb{H}^{n+1}_{(\xi,s)})$ of functions having finite norm
$$
\left\|v\right\|:=\sup_{(\xi,s)\in\mathbb{H}^{n+1}}\frac{\left|v(\xi,s)\right|}{\left(\left|\xi\right|+s\right)^{\alpha}}<\infty,
$$ 
and (\ref{wnweqpw9.3}) yield that $f(\xi+st)\in\mathcal{S}'_{0}(\mathbb{R}_{t}^{n},E)$; thus, we may replace $\psi$ in (\ref{wnweqpw9.3}) by any wavelet from $\mathcal{S}_{0}(\mathbb{R}^{n})$, in particular, by $\psi_1$. Obviously, we can also suppose $\gamma<\alpha$. Finally, because of (\ref{wnweqpw9.4}), we obtain that the improper integral in (\ref{wnweqpw9.7}) is actually absolutely convergent in $C_{b}(\mathbb{R})$, and hence, taking $b=\left|x\right|$, we conclude
\begin{align*}
f(x)&=F_{\phi_1}f(x,\left|x\right|)+\int_{0}^{1}\mathcal{W}_{\psi_{1}}f(x,r\left|x\right|)\frac{\mathrm{d}r}{r}
\\
&
=F_{\phi_1}f(x,\left|x\right|)+O\left(\left|x\right|^{\alpha}\int_{\left|x\right|^{\frac{\alpha}{\gamma}}}^{1}\frac{1+r^{\alpha}}{r}\mathrm{d}r+\int^{\left|x\right|^{\frac{\alpha}{\gamma}}}_{0}r^{\gamma-1}\mathrm{d}r\right)
\\
&
=P(x)+O(\left|x\right|^{\alpha}\log(1/\left|x\right|)).
\end{align*}

(ii) Suppose that we were able to show that it is possible to assume $\psi=\psi_1$ in (\ref{wnweqpw9.3}) and (\ref{wnweqpw9.6}), then the Littlewood-Paley decomposition (\ref{wnweqpw9.7}) would be an equality between continuous functions and so $f\in C^{\alpha-\frac{\alpha}{N}}(x_0)$ would be a consequence of (\ref{wnweqpw9.8}) and the estimate
\begin{align*}
\int_{0}^{1}\mathcal{W}_{\psi_{1}}f(x,r\left|x\right|)\frac{\mathrm{d}r}{r}&=O\left(\left|x\right|^{\alpha}\int_{e^{-\left|x\right|^{-\frac{\alpha}{N}}}}^{1}\frac{1+r^{\alpha}}{r}\mathrm{d}r+\int^{e^{-\left|x\right|^{-\frac{\alpha}{N}}}}_0 \frac{\mathrm{d}r}{r\left|\log r\right|^{N+1}}\right)
\\
&
=O(\left|x\right|^{\alpha-\frac{\alpha}{N}})+O(\left|x\right|^{\alpha})=O(\left|x\right|^{\alpha-\frac{\alpha}{N}}).
\end{align*}
Let us then show first that (\ref{wnweqpw9.6}) holds for $\psi_1$ in place of $\psi$. By Theorem \ref{wnwceth2} applied to $\mathbf{f}_1(t)(\xi)=f(t+\xi)\in\mathcal{S}'(\mathbb{R}^{n}_t,C(\mathbb{R}^{n}_\xi))$, we may assume that $f(t+\xi)\in\mathcal{S}'(\mathbb{R}^{n}_t,C_{b}(\mathbb{R}^{n}_\xi))$, because subtraction of an entire function does not change the hypotheses nor the conclusion. Thus, by Proposition \ref{wnwp9.2}, $\mathbf{f}_1\in C^{0,\left|\log\right|^{-N-1}}_{\ast,w}(0, C_{b}(\mathbb{R}))$, and hence, evaluation at $\psi_1$ gives precisely $\mathcal{W}_{\psi_1}f(\xi,\varepsilon)=O((\log(1/\varepsilon))^{-N-1})$ uniformly in $\xi\in\mathbb{R}^{n}$. Finally, the last relation implies that the integral in (\ref{wnweqpw9.7}) is absolutely convergent and so $f\in C_b(\mathbb{R}^{n})$; therefore, (\ref{wnweqpw9.3}) holds for all $(x,y)\in\mathbb{H}^{n+1}$, and exactly the same argument used in the proof of (i) yields (\ref{wnweqpw9.3}) for $\psi=\psi_1$.
\end{proof}
\begin{remark} Theorem \ref{wnw2mth1} may be restated in terms of $2$-microlocal spaces. Indeed, the condition (\ref{wnweqpw9.3}) is equivalent to say that $f$ coincides near $x_{0}$ with an element of the $2$-microlocal space $C^{\alpha,-\alpha}_{x_0}$, see \cite{meyer}.
\end{remark}

Concerning Part (ii) of Theorem \ref{wnw2mth1}, Meyer has considered the Tauberian condition $f\in C(\mathbb{R}^{n})$ and, for every $N>0$,
\begin{equation}
\label{wnwpweq8}
\left|f(t+h)-f(t)\right|\leq O((\log(1/\left|h\right|))^{-N}),\ \ \ \mbox{for all }t\in\mathbb{R}^{n} \mbox{ and } \left|h\right|\leq \frac{1}{2},
\end{equation}
Since (\ref{wnwpweq8}) readily implies that (\ref{wnweqpw9.6}) holds for all $N$ and suitable constants $C_{2}=C_{N}$, we have recovered the following result of Meyer \cite[Thm. 3.13]{meyer}.

\begin{corollary} If $f\in C(\mathbb{R}^{n})$ satisfies (\ref{wnweqpw9.3}) and the Tauberian condition (\ref{wnwpweq8}) for every $N\in\mathbb{N}$, then $f\in C^{s}(x_0)$ for all $s<\alpha$.
\label{wnwpwc1}
\end{corollary}

We end this subsection by comparing Theorem \ref{wnwfeth9.6} with Holschneider's inverse theorem for global regularity \cite{holschneider94,holschneider}. His theorem is in terms of submultiplicative functions. Recall a positive function $R\in L^{\infty}_{loc}(\mathbb{R}_{+})$ is called \emph{submultiplicative} if there exists a constant $\tilde{C}>0$ such that
\begin{equation}
\label{wnwfeeq8.3}
R(ab)\leq \tilde{C} R(a)R(b), \ \ \ \mbox{for all }a,b\in\mathbb{R}_{+}.
\end{equation}
Holschneider imposes the following technical conditions on the ``growth'' of $R$ at $0$ and $\infty$,
\begin{equation}\label{holeq1}\int_{0}^{1}\frac{R(t)}{t^{k+1}}\:\mathrm{d}t<\infty \ \ \ \mbox{and} \ \ \ \int_{1}^{\infty}\frac{R(t)}{t^{k+2}}\:\mathrm{d}t<\infty,
\end{equation}
for some $k\in\mathbb{N}$. In addition, he also assumed that $R$ is monotonous, but this assumption plays no role in our discussion. Interestingly, such functions do not stay too far from regularly varying functions, more precisely:
\begin{proposition}\label{wnwap8.6} Let $R\in L^{\infty}_{loc}(\mathbb{R}_{+})$ be a positive submultiplicative function that satisfies (\ref{holeq1}). Then, there exist two numbers $k<\alpha\leq\beta<k+1$ and functions $L_1\in C(0,1]$, slowly varying at the origin, and $L_2\in C[1,\infty)$, slowly varying at infinity, such that
\begin{equation}
\label{wnweq9.8}
\tilde{C}^{-1}\varepsilon^{\alpha}\leq R(\varepsilon)\leq \varepsilon^{\alpha}L_{1}(\varepsilon),\ \ \ \mbox{for all } \varepsilon\in (0,1],
\end{equation}
\begin{equation*}
\tilde{C}^{-1}\lambda^{\beta}\leq R(\lambda)\leq \lambda^{\beta}L_{2}(\lambda),\ \ \ \mbox{for all } \lambda\in [1,\infty),
\end{equation*}
and
\begin{equation}
\label{wnweqR9.12} \limsup_{\varepsilon\to0^{+}}\frac{R(\varepsilon)}{\varepsilon^{\alpha}L_1(\varepsilon)}=1 \ \ \ \mbox{and} \ \ \ \limsup_{\lambda\to\infty}\frac{R(\lambda)}{\lambda^{\alpha}L_2(\lambda)}=1.
\end{equation}
 \end{proposition}

\begin{proof}Since $T(t)=\log(\tilde{C}R(e^{t}))$ is subadditive, an application of the well known limit Theorem for subadditive functions \cite[Thm. 7.6.2]{hille-phillips}(cf. \cite{beurling}) gives the existence of
\begin{equation*}
 \alpha=\lim_{\varepsilon\to0^{+}}\frac{\log R(\varepsilon)}{\log \varepsilon}=\sup_{\varepsilon<1}\frac{\log (\tilde{C}R(\varepsilon))}{\log \varepsilon}
\end{equation*}
and
\begin{equation*}\beta=\lim_{\lambda\to\infty} \frac{\log R(\lambda)}{\log\lambda}=\inf_{1<\lambda }\frac{\log(\tilde{C} R(\lambda))}{\log\lambda},
\end{equation*}
and also the fact $\alpha\leq\beta$. Now, Phillips's theorem
\cite[Thm. 7.4.4]{hille-phillips} (cf. \cite{phillips}), applied
to $$T_{1}(t)=\log (\tilde{C}R(e^{t}))-kt \ \ \mbox{ and } \ \ 
T_{2}(t)=\log (\tilde{C}R(e^{-t}))+kt,$$ and the relations
(\ref{holeq1}) imply the growth estimates
$$
R(\lambda)=O(\lambda^{k+1}/(\log\lambda)^{2}) \ \ \ \mbox{as }
\lambda\to\infty$$
and
$$
R(1/\lambda)=O(\lambda^{-k}/(\log\lambda)^{2}) \ \ \ \mbox{as }
\lambda\to\infty,$$
and so we infer that $k<\alpha\leq\beta<k+1$.
The existence of $L_1$ and $L_2$ with the required properties
follows from the approximation theorem by regularly varying
functions \cite[Thm. 2.3.11, p. 81]{bingham}.
\end{proof}

Therefore, Theorem \ref{wnwfeth9.6} also allows us to obtain information about the regularity of distributions when their wavelet transforms are bounded by submultiplicative functions.

\begin{corollary}
\label{wnwfec2}
Let $R\in L^{\infty}_{loc}(\mathbb{R}_{+})$ be a positive submultiplicative function that satisfies satisfies (\ref{holeq1}) for some $k\in\mathbb{N}$ and let $\psi\in\mathcal{S}(\mathbb{R}^{n})$ be a non-degenerate wavelet with moments $\mu_{m}(\psi)=0$ for $\left|m\right|\leq k$. If $f\in\mathcal{S}'(\mathbb{R}^{n})$ satisfies
\begin{equation*}
\mathcal{W}_{\psi}f(x,y)\leq CR(y) \  \ \ \mbox{for all }(x,y)\in\mathbb{R}^{n}\times(0,y_{0}],
\end{equation*}
for some constants $C>0$ and $0<y_0\leq1$, then there exists an entire function $G$ that satisfies (\ref{wnweq9.5}) such that 
$$f-G\in C^{\alpha,L_1}(\mathbb{R}^{n}),$$ 
where
$$
\alpha=\lim_{\varepsilon\to0^{+}} \frac{\log R(\varepsilon)}{\log \varepsilon}
$$
and $L_1\in C(0,1]$ is any slowly varying function that fulfills (\ref{wnweq9.8}) and (\ref{wnweqR9.12}).
\end{corollary}

\subsection{Asymptotic Stabilization in Time for Cauchy Problems}
\label{wnwPDE} We retain in this subsection the notation from Example \ref{wnwex3.4}, that is, $U$ is the unique solution to the Cauchy problem (\ref{wnweq3.5}) and  $\phi=(2\pi)^{-n}\hat{\eta}$, where $\eta\in\mathcal{S}(\mathbb{R}^{n})$ satisfies $\eta(u)=e^{P(iu)}$, $u\in\Gamma$; thus, $U$ is given by (\ref{wnweq3.8}). We apply Theorem \ref{wnwth6.2} to find sufficient geometric conditions for the stabilization in time of the solution to the Cauchy problem (\ref{wnweq3.5}), namely, we study conditions which ensure the existence of a function $T:(A,\infty)\to\mathbb{R}_{+}$ and a constant $\ell\in\mathbb{C}$ such that the following limits exist
\begin{equation}
\label{wnwPDEeq1}
\lim_{t\to\infty}\frac{U(x,t)}{T(t)}=\ell, \ \ \ \mbox{for each }x\in\mathbb{R}^{n}.
\end{equation}

Let $L$ be slowly varying at infinity and $\alpha\in\mathbb{R}$. We shall say that $U$ \emph{stabilizes along $d$-curves} (at infinity), relative to $\lambda^{\alpha}L(\lambda)$, if the following two conditions hold:
\begin{enumerate}
\item There exist the limits
\begin{equation}
\label{wnwPDEeq2}
\lim_{\lambda\to\infty}\frac{U(\lambda x, \lambda^{d} t)}{\lambda^{\alpha}L(\lambda)}=U_{0}(x,t), \ \ \ (x,t)\in\mathbb{H}^{n+1}\cap\mathbb{S}^{n};
\end{equation}
\item There are constants $M\in\mathbb{R}_{+}$ and $l\in\mathbb{N}$ such that
\begin{equation}
\label{wnwPDEeq3}
\left|\frac{U(\lambda x, \lambda^{d} t)}{\lambda^{\alpha}L(\lambda)}\right|\leq \frac{M}{t^{l}}, \ \ \ (x,t)\in\mathbb{H}^{n+1}\cap\mathbb{S}^{n}.
\end{equation}
\end{enumerate}
\begin{theorem}
\label{wnwth7.1} The solution $U$ to the Cauchy problem (\ref{wnweq3.5}) stabilizes along $d$-curves, relative to $\lambda^\alpha L(\lambda)$, if and only if $f$ has weak-asymptotic behavior of degree $\alpha$ at infinity with respect to $L$.
\begin{proof} We have that $U(x,t)=F_{\phi}f(x,y)$, with $y=t^{1/d}$, then, conditions (\ref{wnwPDEeq2}) and (\ref{wnwPDEeq3}) translate directly into conditions (\ref{wnwtteq1}) and (\ref{wnwtteq5}), with $F_{x,y}=U_{0}(x,t^{1/d})$ and $k=dl$. Therefore, Theorem \ref{wnwth6.2} yields the desired equivalence.
\end{proof}
\end{theorem}

\begin{corollary}
\label{wnwPDEc1} If $U$ stabilizes along $d$-curves, relative to $\lambda^{\alpha}L(\lambda)$, then $U$ stabilizes in time with respect to $T(t)=t^{\alpha /d}L(t^{1/d})$. Moreover, the limit (\ref{wnwPDEeq1}) holds uniformly for $x$ in compacts of $\mathbb{R}^{n}$.
\end{corollary}
\begin{proof} By Theorem \ref{wnwth7.1}, there exists $g\in\mathcal{S}'(\mathbb{R}^{n})$ such that
$$
f(\lambda \xi)\sim \lambda^{\alpha}L(\lambda)g(\xi) \ \ \ \mbox{as }\lambda\to\infty\ \mbox{in }\mathcal{S}'(\mathbb{R}^{n}).
$$
If $K\subset\mathbb{R}^{n}$ is compact, then,
\begin{align*}
\lim_{t\to\infty}\frac{U(x,t)}{T(t)}&=\lim_{t\to\infty}\frac{1}{t^{\alpha /d}L(t^{1/d})}\left\langle f(t^{1/d}\xi),\phi\left(\xi-\frac{x}{t^{1/d}}\right)\right\rangle
\\
&
=\left\langle g(\xi),\phi(\xi)\right\rangle,
\end{align*}
uniformly for $x\in K$ because 
$$\phi\left(\xi-x/t^{1/d}\right)\to\phi(\xi)\ \ \ \mbox{in } \mathcal{S}(\mathbb{R}^{n}), \mbox{ as }t\to\infty.$$ 
\end{proof}
\begin{example} \emph{The heat equation.} When $\Gamma=\mathbb{R}^{n}$ and $P(\partial/\partial x)=\Delta$, we obtain that stabilization along parabolas (i.e., $d=2$) is sufficient for stabilization in time of the solution to the Cauchy problem for the heat equation. This particular case of Corollary \ref{wnwPDEc1} was studied in \cite{drozhzhinov-z2,drozhzhinov-z3,drozhzhinov-z5}.
\end{example}

\subsection{Applications to Regularity Theory in Algebras of Generalized Functions}
\label{wnwCA}
In this section we show how the Tauberian theorems for the wavelet transform can be used as a standard device to derive results in the regularity theory for algebras of generalized functions. 

First, we consider the  algebra of tempered generalized functions which contains $\mathcal S'(\mathbb R^n)$ as a proper subspace. Let $\mathcal{O}_{M}(\mathbb{R}^{n})$ be the space of multipliers of $\mathcal{S}(\mathbb{R}^{n})$ \cite{schwartz1}, that is, the space of smooth functions whose derivatives are bounded by polynomials, of possible different degrees. Colombeau \cite{colombeau} defined
the algebra of tempered generalized functions  as the quotient  $\mathcal{G}_{\tau}(\mathbb{R}^{n})=\mathcal{E}_{M,\tau}(\mathbb{R}^{n})/\mathcal{N}_{\tau}(\mathbb{R}^{n})$, where $\mathcal{E}_{M,\tau}(\mathbb{R}^{n})$  is the algebra of nets $(f_{\varepsilon})_{\varepsilon}\in\mathcal{O}_{M}(\mathbb{R}^{n})^{(0,1)}$
$$
(\forall m\in\mathbb{N}^{n})(\exists N\in\mathbb{N})(\sup_{x\in\mathbb{R}^{n}}(1+\left|x\right|)^{-N}|f^{(m)}_{\varepsilon}(x)|=O(\varepsilon^{-N}))
$$
while its ideal $\mathcal{N}_{\tau}(\mathbb{R}^{n})$ consists of those such that
$$
(\forall m\in\mathbb{N}^{n})(\exists N\in\mathbb{N})(\forall b>0)(\sup_{x\in\mathbb{R}^{n}}(1+\left|x\right|)^{-N}|f^{(m)}_{\varepsilon}(x)|=O(\varepsilon^{b})).
$$
We can embed $\mathcal{S}'(\mathbb{R}^{n})$ into $\mathcal{G}_\tau(\mathbb
R^n)$ via $\iota(f)=[(f\ast\phi_{\varepsilon})_{\varepsilon}]$, where $\phi$ satisfies the condition (\ref{wnweq3.3}).

The algebra of regular tempered generalized functions $\mathcal{G}_\tau^\infty(\mathbb
R^n)$ consists of those nets in $\mathcal{O}_{M}(\mathbb{R}^{n})^{(0,1)}$ such that
\begin{equation}
\label{wnwcoleq1}
(\exists a\in\mathbb{R})(\forall m\in\mathbb{N}^{n})(\exists N\in\mathbb{N})(\sup_{x\in\mathbb{R}^{n}}(1+\left|x\right|)^{-N}|f^{(m)}_{\varepsilon}(x)|=O(\varepsilon^{-a})).
\end{equation}
We will show the regularity theorem for $\mathcal{G}_\tau^\infty(\mathbb
R^n)$; it originally appeared in \cite{hor}. The proof is similar to that of Theorem \ref{wnwfeth9.6}.

\begin{theorem}
$\mathcal S'(\mathbb R^n)\cap\mathcal G_\tau^\infty(\mathbb
R^n)=\mathcal{O}_M(\mathbb R^n).$
\end{theorem}
This equality  means that if $f\in\mathcal S'(\mathbb R^n)$ and $f_\varepsilon=f*\phi_\varepsilon,$ $\varepsilon\in(0,1),$ determines an element of
 $\mathcal G_\tau^\infty(\mathbb
R^n)$, then $f\in\mathcal{O}_M(\mathbb R^n).$
\begin{proof} The inclusion $\mathcal{O}_M(\mathbb R^n)\subseteq \mathcal S'(\mathbb {R}^n)\cap\mathcal G_\tau^\infty(\mathbb
R^n)$ is obvious. Let $f\in\mathcal{S}'(\mathbb{R}^{n})$ such that $\iota(f)\in\mathcal G_\tau^\infty(\mathbb
R^n)$, that is, the net $f_{\varepsilon}=f\ast\phi_{\varepsilon}$ satisfies (\ref{wnwcoleq1}). We should show that $f^{(m)}$ is continuous of polynomial growth for each $m\in\mathbb{N}^{n}$. Let $\nu\in\mathbb{N}$ be such that $\beta=2\nu-a>0$. Then, there exists $N_{0}\in\mathbb{N}$ such that
\begin{equation}
\label{wnwcoleq2}
\sup_{x\in\mathbb{R}^{n}}(1+\left|x\right|)^{-N_{0}}\left|\mathcal{W}_{\psi}f^{(m)}(x,y)\right|=O(y^{\beta}), \ \ \ 0<y<1,
\end{equation}
where $\psi=\overline{\Delta^{\nu}\check{\phi}}$, a non-degenerate
wavelet.  Define $\mathbf{h}$ by $$\left\langle
\mathbf{h},\rho\right\rangle=f^{(m)}\ast \check{\rho},$$ for
$\rho\in\mathcal{S}(\mathbb{R}^{n})$. Then there exists $N>N_{0}$
such that $\mathbf{h}\in\mathcal{S}'(\mathbb{R}^{n},E)$, where $E$
is the Banach space of continuous functions $v\in
C(\mathbb{R}^{n})$ such that
$$\left\|v\right\|:=\sup_{\xi\in\mathbb{R}^{n}}(1+\left|\xi\right|)^{-N}\left|v(\xi)\right|<\infty,$$
provided with the norm $\left\|\:\cdot\:\right\|$. Since
$\mathcal{W}_{\psi}\mathbf{h}(x,y)(\xi)=\mathcal{W}_{\psi}f^{(m)}(\xi+x,y)$,
the estimate (\ref{wnwcoleq2}) gives now
$$
\limsup_{\varepsilon\to0^{+}}\sup_{\left|x\right|\leq 1,\: 0<y<1}\varepsilon^{-\beta}\left\|\mathcal{W}_{\psi}\mathbf{h}(\varepsilon x,\varepsilon y)\right\|<\infty.
$$
Theorem \ref{wnwth8} implies, in particular, that $\mathbf{h}$ has
a distributional point value at the origin (cf. Example
\ref{wnwex2.1}), say $\mathbf{h}(0)=v\in E$, distributionally,
i.e., for each test function $\rho$,
$$\lim_{\varepsilon\to0^{+}}f^{(m)}\ast\check{\rho}_{\varepsilon}=\lim_{\varepsilon\to0^{+}}\left\langle
\mathbf{h}(\varepsilon t),\rho(t)\right\rangle= v
\int_{\mathbb{R}^{n}}\rho(t)\mathrm{d}t,$$ where the limit holds
in $E$. But if we take $\rho=\check{\phi}$, we obtain in
particular that
$\lim_{\varepsilon\to0^{+}}(f^{(m)}\ast\phi_{\varepsilon})(\xi)=v(\xi)$
uniformly for $\xi$ in compacts of $\mathbb{R}^{n}$, and this
means exactly that $f^{(m)}=v$ is a continuous function of at most
polynomial growth.
\end{proof}

\begin{remark}
Recall \cite{colombeau} that the Colombeau algebra of generalized functions is defined  as
$\mathcal{G}(\Omega)=\mathcal E_M(\Omega)/\mathcal N(\Omega),$ where $\mathcal E_M(\Omega)$,  $\mathcal N(\Omega)$,
consist of nets of smooth functions in $\Omega$, $(f_\varepsilon)_{\varepsilon\in(0,1)}$, with the properties
$$
(\forall \omega\subset\subset \Omega) (\forall \nu\in\mathbb{N}) (\exists N\in\mathbb{N})
 (\sup_{|m|\leq
\nu,x\in\omega}|f^{(m)}_{\varepsilon}(x)|=O(\varepsilon^{-N})),
$$
$$
(\forall \omega\subset\subset \Omega) (\forall b\in\mathbb{R})
(\forall \nu\in\mathbb{N}) (\sup_{|m|\leq
\nu,x\in\omega}|f^{(m)}_{\varepsilon}(x)|=O(\varepsilon^{b})).
$$
The embedding of the Schwartz distribution space
$\mathcal{E}^{\prime}(\Omega)$ is realized through the sheaf
homomorphism 
$$ \mathcal{E}^{\prime}(\Omega)\ni f\mapsto\iota(f)=
[(f\ast\phi_{\varepsilon}|{\Omega})_\varepsilon]\in\mathcal{G}(\Omega),$$
where $\phi\in\mathcal{S}(\mathbb{R}^{n})$ is as before.
This sheaf homomorphism, extended over $\mathcal{D}^{\prime}$,
gives the embedding of $\mathcal{D}^{\prime}(\Omega)$ into
$\mathcal{G}(\Omega)$. The embedding respects the multiplication of smooth functions.

The generalized algebra of ``smooth generalized functions"
$\mathcal{G}^\infty(\Omega)$ is defined in \cite{ober 001} as the
quotient of the algebras $\mathcal{E}^\infty_{M}(\Omega)$ and
$\mathcal{N}(\Omega),$ where $\mathcal{E}^\infty_{M}(\Omega),$
consists of nets of smooth functions in $\Omega$
 with the property
$$
(\forall \omega\subset\subset \Omega) (\exists a\in\mathbb{R})
(\forall \nu\in\mathbb{N}) (\sup_{|m|\leq
\nu,x\in\omega}|f^{(m)}_{\varepsilon}(x)|=O(\varepsilon^{-a})),
$$

Note that $\mathcal{G}^\infty$ is a subsheaf of  $\mathcal{G}$.
Roughly speaking, it has the same role as $C^\infty$ in $\mathcal
D'.$

Similarly as above, one can prove the following well known assertion  \cite{ober 001}:
\begin{theorem}
\label{col} $\mathcal D'(\Omega)\cap\mathcal
G^\infty(\Omega)=C^\infty(\Omega).$
\end{theorem}

In fact, Theorem \ref{col} is also a direct consequence of Theorem \ref{wnwfeth9.6}.
\end{remark}

\subsection{Distributionally Small Distributions at Infinity}
\label{wnwds} Estrada \cite{estrada98,estrada-kanwal2} has characterized the class of distributions which are distributionally small at infinity (cf. Example \ref{wnwex2.3}), that is, the ones which have a weak-asymptotic expansion
\begin{equation}
\label{wnwdseq1}
\mathbf{f}(\lambda t)\sim \sum_{\left|m\right|=0}^{\infty}\frac{(-1)^{\left|m\right|}}{m!\lambda^{\left|m\right|+n}} \delta^{(m)}(t)\mathbf{w}_{m}  \ \ \ \mbox{as}\ \lambda\to\infty \ \ \mbox{in}\
\mathcal{S}'(\mathbb{R}^{n},E),
\end{equation}
for some multi-sequence $\left\{\mathbf{w}_{m}\right\}_{m\in\mathbb{N}^{n}}$ in $E$.

The distributionally small distributions are precisely the elements of the space $\mathcal{K}'(\mathbb{R}^{n},E)$, where $\mathcal{K}(\mathbb{R}^{n})$ is the test function space of the so called GLS symbols \cite{gls}, defined as follows. Given $\beta\in\mathbb{R}$, the space $\mathcal{K}_{\beta}(\mathbb{R}^{n})$ is formed by those smooth functions $\rho$ such that for each $m\in\mathbb{N}^{n}$
$$\rho^{(m)}(t)=O(\left|t\right|^{\beta-\left|m\right|})\ \ \ \mbox{as} \left|t\right|\to\infty,$$
 provided with the topology generated by the seminorms
$$
\max\{\sup_{\left|t\right|\leq1}|\rho^{(m)}(t)|, \sup_{\left|t\right|\geq1}\left|t\right|^{\left|m\right|-\beta}|\rho^{(m)}(t)|\}.
$$
Then, $\mathcal{K}(\mathbb{R}^{n})=\mbox{ind lim}_{\beta\to\infty}\mathcal{K}_{\beta}(\mathbb{R}^{n})$. Observe that the elements of $\mathcal{K}(\mathbb{R}^{n})$ are indeed symbols of pseudodifferential operators.

It is important to emphasize that if (\ref{wnwdseq1}) holds, then \cite{estrada98,estrada-kanwal2} it actually holds in the space $\mathcal{K}'(\mathbb{R}^{n},E)$.

Fourier transforming (\ref{wnwdseq1}), we have that
$\mathbf{f}\in\mathcal{K}'(\mathbb{R}^{n},E)$ if and only if $\hat{\mathbf{f}}\in C_{w}^{\infty}(0,E)$, where we set, cf. Section \ref{wnwla},
\begin{equation}
\label{Cweq}C_{w}^{\infty}(x_0,E):=\bigcap_{\alpha\in\mathbb{R}}C_{\ast,w}^{\alpha}(x_0,E)=\bigcap_{\alpha\in\mathbb{R}}C_{w}^{\alpha}(x_0,E), \ \ \ \mbox{for }x_{0}\in\mathbb{R}^{n}.
\end{equation}
Therefore, using Theorem \ref{wnwth8}, we have obtained the following wavelet characterization of the space $\mathcal{K}'(\mathbb{R}^{n},E)$. Let $\psi\in\mathcal{S}_{0}(\mathbb{R}^{n})$ be a non-degenerate wavelet. \emph{Then, an $E$-valued tempered distribution $\mathbf{f}$ belongs to $\mathcal{K}'(\mathbb{R}^{n},E)$ if and only if there exists a sequence $\left\{\nu_{p}\right\}_{p=0}^{\infty}$ of non-negative integers such that for each $p\in\mathbb{N}$}
$$
\limsup_{\varepsilon\to0^{+}}\sup_{\left|x\right|^{2}+y^{2}=1,\: y>0}\frac{y^{\nu_{p}}}{\varepsilon^{p}}\left\|\mathcal{W}_{\psi}\hat{\mathbf{f}}(\varepsilon x,\varepsilon y)\right\|<\infty .
$$
\subsection{Pointwise Analysis of Riemann Type Distributions at the Rationals}\label{wnwWRfunction}
We will investigate the pointwise weak-asymptotic expansion of the family of \emph{Riemann distributions}
$$R_{\beta}(t)=\sum_{n=1}^\infty \frac{e^{i\pi n^2t}}{n^{2\beta}}\in \mathcal{S}'(\mathbb{R}_{t}), \ \ \ \beta\in\mathbb{C},$$
at points of $\mathbb{Q}$. We split $\mathbb{Q}$ into two disjoint subsets $S_{0}$ and $S_{1}$ where
$$
S_{0}=\left\{\frac{2\nu+1}{2j}:\: \nu,j\in\mathbb{Z}\right\}\cup\left\{\frac{2j}{2\nu+1}:\: \nu,j\in\mathbb{Z}\right\}
$$
and
$$
S_{1}=\left\{\frac{2\nu+1}{2j+1}:\: \nu,j\in\mathbb{Z}\right\}.
$$

When $\beta>1/2$, $R_{\beta}$ is a continuous function. The imaginary part of $R_1$ is the classical Riemann ``non-differentiable'' function. It is well known \cite{gerver,holschneider-t} that  if $\beta>3/4$, then $R_{\beta}$ is differentiable at the points of $S_{1}$ and has local cusps with differentiable remainder at points of $S_0$; for $\beta\in(1/2,5/4)$, $R_{\beta}$ is not differentiable at any irrational point, as shown essentially by Hardy and Littlewood \cite{hardy1916,hardy-littlewood1914,holschneider-t}. Jaffard and Meyer \cite{jaffard-m} showed that $\Im mR_1$ has trigonometric chirps at the points of $S_1$. Consult \cite{duistermaat,jaffard96,jaffard-m} for properties of $R_1$ at the irrationals.

We will exhibit more precise information concerning the scaling weak-asymptotic properties of $R_{\beta}$ at the rationals, in fact, we will show that $R_{\beta}$ admits a full weak-asymptotic series at points of $\mathbb{Q}$, no matter the value of $\beta$. In particular, our analysis reveals that $R_{\beta}$ has weak scaling exponent (cf. Example \ref{wnwex2.0}) equal to $\infty$ at points of $S_1$; at points of $S_{0}$, it has infinite weak scaling exponent after subtraction of an adequate term.  To this end, we will be led to the study of the analytic continuation of the zeta-type function
\begin{equation}
\label{wnwzetaeq}
\zeta_{r}(z):=\sum_{n=1}^{\infty}\frac{e^{i\pi r n^{2}}}{n^{z}}, \ \ \ \Re e\:z>1,
\end{equation}
where $r\in\mathbb{Q}$. If $r=0$, (\ref{wnwzetaeq}) reduces to $\zeta_{0}=\zeta$, the familiar Riemann zeta function. The main results of this subsection are Theorems \ref{wnwRWfunctionth1} and Theorem \ref{wnwrwth2}. The results of Dur\'{an} and Estrada \cite{duran-estrada2,duran-estrada3,estrada98} will play an important role in our arguments.

We first study the properties of $R_{0}$. We begin with the expansion at $0$. Consider the tempered distribution
$$f(t)=R_{0}(t)-\frac{\sqrt{i}}{2}(t+i0)^{-\frac{1}{2}}=\sum_{n=1}^{\infty}e^{i\pi n^{2}t}-\frac{\sqrt{i}}{2}(t+i0)^{-\frac{1}{2}},$$
where we denote by $(t+i0)^{\alpha}$ the boundary value of the analytic function $z^{\alpha}$, $\Im m\:z>0$. Its Fourier transform is given by
$$
\hat{f}(u)=2\pi\sum_{n=1}^{\infty}\delta(u-\pi n^{2})-\sqrt{\pi}u^{-\frac{1}{2}}_{+}.
$$
Then, if $\psi\in\mathcal{S}_{0}(\mathbb{R})$ and $\nu\in\mathbb{N}$,
\begin{align*}
\left|\mathcal{W}_{\psi}f(\varepsilon x, \varepsilon y)\right|&= \left|\sum_{n=1}^{\infty}e^{i\varepsilon x\pi n^{2}}\overline{\hat{\psi}}\left(\varepsilon y \pi n^{2}\right)- \int_{0}^{\infty}e^{i\varepsilon x\pi u^{2}}\overline{\hat{\psi}}\left(\varepsilon y\pi u^{2}\right)\mathrm{d}u\right|
\\
&
\leq \frac{\varepsilon^{\nu}}{y^{\nu+1}}O(1), \ \ \ \mbox{uniformly in }y,\varepsilon\in(0,1) \mbox{ and }\left|x\right|\leq 1,
\end{align*}
as shown by the Euler-Maclaurin summation formula \cite{estrada-kanwal2}.
This implies that $\hat{f}\in\mathcal{K}'(\mathbb{R})$ and thus satisfies the Estrada-Kanwal moment asymptotic expansion,
$$
\hat{f}(\lambda u)\sim \sum_{m=0}^{\infty}\frac{(-1)^{m}\mu_{m}}{m!\lambda^{m+1}} \delta^{(m)}(u)  \ \ \ \mbox{as}\ \lambda\to\infty \mbox{ in } \mathcal{K}'(\mathbb{R}_u).
$$
The moments of $\hat{f}$ can actually be evaluated in the Ces\`{a}ro sense \cite{estrada98,estrada-kanwal2}. If $H$ denotes the Heaviside function and $(\mathrm{C})$ stands for limits in the Ces\`{a}ro sense, then \cite{duran-estrada3}
\begin{align*}
\frac{1}{2\pi}\mu_{m}
&=\frac{1}{2\pi}\left\langle \hat{f}(u),u^{m}\right\rangle
=\pi^{m}\left\langle \sum_{n=1}^{\infty}\delta(\xi-n)- H(\xi),\xi^{2m}\right\rangle
\\
&
=\pi^{m} \lim_{x\to\infty} \left(\sum_{1\leq n\leq x}n^{2m}-\int_{0}^{x}\xi^{2m}\mathrm{d}\xi\right)
= \pi^{m}\zeta(-2m) \ \ \ (\mathrm{C}),
\end{align*}
and hence $\mu_{0}=2\pi\zeta(0)=-\pi$, and $\mu_{m}=\zeta(-2m)=0$ for every $m\geq1$. Consequently,
$$
\hat{f}(\lambda u)= -\pi\frac{\delta( u)}{\lambda}+o\left(\frac{1}{\lambda^{\infty}}\right) \ \ \ \mbox{as }\lambda\to\infty \mbox{ in } \mathcal{S}'(\mathbb{R}_u),
$$
where $o\left(1/\lambda^{\infty}\right)$ means $o\left(1/\lambda^{N}\right)$ for every $N\in\mathbb N.$
Taking Fourier inverse transform, we have the weak-asymptotic expansion
\begin{equation}
\label{wnwdseq8.14}
R_{0}(\varepsilon t)=\sum_{n=1}^{\infty}e^{i \pi n^{2}\varepsilon t}=\frac{\sqrt{i}}{2}\varepsilon^{-\frac{1}{2}} (t+i0)^{-\frac{1}{2}}-\frac{1}{2}+o(\varepsilon^{\infty}) \ \ \ \mbox{in }\mathcal{S}'(\mathbb{R}_t)
\end{equation}
as $\varepsilon\to0^{+}$.

We now determine the weak-asymptotic expansion of $R_0$ at $1$. Observe that $R_0(1+t)=2R_0(4t)-R_0(t)$, thus the behavior at origin implies that
\begin{equation}
\label{wnwdseq8.15}
R_{0}(1+\varepsilon t)=-\frac{1}{2}+o(\varepsilon^{\infty}) \ \ \ \mbox{ as }\varepsilon\to0^{+} \mbox{ in }\mathcal{S}'(\mathbb{R}_t).
\end{equation}

 We return to the general case. Consider the two complex transformations
$$Kz=z+1 \ \ \ \mbox{and} \ \ \ Uz=-1/z, \ \ \ \mbox{for }z\in\mathbb{C},$$
they generate the well know modular group which leaves invariant the upper half-plane and the real line. We are more interested in the \emph{theta group}, namely, the subgroup $G_{\vartheta}$ of modular transformations generated by $K^{2}$ and $U$. Then, one readily verifies that

$$
G_{\vartheta}\cdot 0=S_{0} \ \ \ \mbox{ and } \ \ \ G_{\vartheta}\cdot 1=S_{1},
$$
that is, $S_{0}$ is the orbit of 0 under $G_{\vartheta}$ while $S_{1}$ that of 1.
Let $\vartheta$ be the \emph{Jacobi theta} function given by
$$
\vartheta(z):=1+2\sum_{n=1}^{\infty}e^{i\pi n^{2}z}, \ \ \ \Im m\: z>0.
$$
We then have the following transformation laws
$$\vartheta (K^{2}z)=\vartheta(z) \ \ \ \mbox{ and } \ \ \ \vartheta (Uz)=\sqrt{-iz}\:\vartheta(z);$$
the first of them is completely obvious, while the second one follows easily from the Poisson summation formula (cf. \cite[p. 304]{holschneider}). Observe that $\vartheta$ admits a boundary tempered distribution on the real line, which we also denote by $\vartheta$, or $\vartheta(t)$.

We recall that the pointwise space $C^{\infty}_{w}(x_{0})$ was introduced in (\ref{Cweq}).

The ensuing lemma describes the scaling weak-asymptotic properties of $R_{0}$ at points of the orbit $S_1=G_{\vartheta}\cdot 1$.
\begin{lemma}
\label{wnwdsl1} Let $r\in G_{\vartheta}\cdot 1$. Then, $R_{0}\in C_{w}^{\infty}(r)$. Furthermore, at those points, $R_{0}(r)=-1/2$ and $R^{(m)}_{0}(r)=0$, distributionally, for each $m\geq1$.
\end{lemma}
\begin{proof} Since $R_0=(\vartheta-1)/2$, it is enough to show that $\vartheta \in C^{\infty}_{w}(r)$ whenever $r\in G_{\vartheta}\cdot 1$, and $\vartheta^{(m)}(r)=0$, distributionally, for all $m\in\mathbb{N}$. If $r=1$, (\ref{wnwdseq8.15}) shows that
$$
\vartheta(1+\varepsilon t)=o(\varepsilon^{\infty})  \ \ \ \mbox{as } \varepsilon\to0^{+} \mbox{ in } \mathcal{S}'(\mathbb{R}_{t}),
$$
and hence $\vartheta\in C^{\infty}_{w}(1)$ and $\vartheta^{(m)}(1)=0$, distributionally, for all $m\in\mathbb{N}$. If the conclusion of the theorem holds at $r$, then clearly it holds at $K^{2}r=r+2$. Therefore, it remains to show that $\vartheta\in C^{\infty}_{w}(r)$ and $\vartheta^{(m)}(r)=0$, distributionally, $m\in\mathbb{N}$, implies $\vartheta\in C^{\infty}_{w}(Ur)$ and $\vartheta^{(m)}(Ur)=0$, distributionally, $m\in\mathbb{N}$. So assume that $\vartheta$ has the desired property at $r$, hence $\vartheta$ satisfies
$$
\vartheta(r+\varepsilon t)=o(\varepsilon^{\infty})\ \ \ \mbox{as }\varepsilon\to0^{+} \mbox{ in } \mathcal{S}'(\mathbb{R}_{t}),
$$
As in Example \ref{wnwex3.5}, we find $\phi\in\mathcal{S}(\mathbb{R})$ such that $\hat{\phi}(-u)=e^{-u}$ for $u\in [0,\infty)$, then $F_{\phi}\vartheta(x,y)=\vartheta(x+iy)$, because
\begin{align*}
F_{\phi}\vartheta(x,y)&=\frac{1}{2\pi}\left\langle e^{ixu}\hat{\vartheta}(u),e^{-yu}\right\rangle 
\\
&
=\left\langle  \delta(u)+2\sum_{n=1}^\infty e^{i\pi x n^2 }\delta(u-\pi n^2),e^{-yu}\right\rangle.
\end{align*}
 By Proposition \ref{wnwth2}, we can find sequences of positive reals $\left\{C_{\nu}\right\}_{\nu=0}^{\infty}$ and positive integers $\left\{k_{\nu}\right\}_{\nu=0}^{\infty}$ such that for each $\nu\in\mathbb{N}$
$$
\left|\vartheta(r+\varepsilon z)\right|\leq \frac{C_{\nu}\varepsilon^{\nu+1}}{(\Im m\: z)^{k_{\nu}}}, \ \ \ \mbox{for all } \varepsilon\in(0,1] \mbox{ and } 0<\left|z\right|\leq 1.$$
But since $\vartheta(U z)=\sqrt{-iz}\:\vartheta (z)$, we have that
\begin{equation}\label{wnwdseq8.16}
\vartheta\left(-\frac{1}{r}+\varepsilon z\right)=\sqrt{\frac{-ir}{1-\varepsilon rz}}\:\vartheta\left( r+ \frac{\varepsilon zr^{2}}{1-\varepsilon zr}\right),
\end{equation}
from where we obtain that
$$
\left|\vartheta\left(-\frac{1}{r}+\varepsilon z\right)\right|\leq \frac{M_{\nu}\varepsilon^{\nu+1}}{(\Im m\: z)^{k_{\nu}}} ,\  \mbox{for all } \nu\in\mathbb{N},\ \varepsilon\in(0,\varepsilon_0] \mbox{ and } 0<\left|z\right|\leq 1,
$$
for suitable $\varepsilon_{0}>0$ and $\left\{M_{\nu}\right\}_{\nu=0}^{\infty}$. Theorem \ref{wnwth6.1} yields immediately $\vartheta\in C^{\infty}_{w}(-1/r)$ and actually $\vartheta^{(m)}(-1/r)=0$, distributionally, for all $m\in\mathbb{N}$. This completes the proof.
\end{proof}
If we introduce the pointwise space
\begin{equation}
\label{wnwspaceO}
\mathcal{O}^{\infty}(x_{0}):=\bigcap_{\alpha\in\mathbb{R}}\mathcal{O}^{\alpha}(x_{0}), \ \ \ x_{0}\in\mathbb{R},
\end{equation}
where $\mathcal{O}^{\alpha}(x_{0})$ was defined in Subsection \ref{wnwla}, we may rephrase Lemma \ref{wnwdsl1} by saying that $R_{0}+1/2\in\mathcal{O}^{\infty}(r)$ for each $r\in G_{\vartheta}\cdot 1$.

Observe that Lemma \ref{wnwdsl1} gives then the full weak-asymptotic expansion of $R_0$ at $r=(2j+1)/(2\nu+1)$, $j,\nu\in\mathbb{Z}$, namely,
$$
R_0 \left(r+\varepsilon t\right)=-\frac{1}{2}+o(\varepsilon^{\infty}) \ \ \ \mbox{as }\varepsilon\to0^{+} \mbox{ in } \mathcal{S}'(\mathbb{R}_{t}).
$$
It then follows by Fourier transforming that
$$
\sum_{n=1}^{\infty}e^{i\pi rn^{2}}\delta(\lambda u -\pi n^{2})=-\frac{\delta(u)}{2\lambda}+o\left(\frac{1}{\lambda^{\infty}}\right)  \ \ \ \mbox{as }\lambda\to\infty \mbox{ in } \mathcal{S}'(\mathbb{R}_{u}).
$$
The above expansion yields that $$\sum_{n=1}^{\infty}e^{i\pi rn^{2}}\delta(u -\pi n^{2})\in\mathcal{K}'(\mathbb{R}_{u}).$$ Its moment series are Ces\`{a}ro summable and its moment function defines an entire function \cite{duran-estrada2,estrada98} (cf. \cite[Thm. 6.7.2 and Thm. 6.11.1]{estrada-kanwal2}). Indeed, since the support of $\sum_{n=1}^{\infty}e^{i\pi rn^{2}}\delta(u -\pi n^{2})$ is $[\pi,\infty)$, we can multiply it by $u^{-\beta}$ and so
$$\left(\frac{\pi}{u}\right)^{\beta}\sum_{n=1}^{\infty}e^{i\pi rn^{2}}\delta(u -\pi n^{2})=\sum_{n=1}^{\infty}n^{-2\beta}e^{i\pi r n^{2}}\delta(u -\pi n^{2})\in\mathcal{K}'(\mathbb{R}_{u}),$$
for each $\beta\in\mathbb{C}$. Therefore, the above distribution admits a moment asymptotic expansion at $\infty$, and, by taking inverse Fourier transform, we readily verify that its $m$-th moment is given by $(i\pi)^{-m}R^{(m)}_{\beta}(r)$, the point values interpreted naturally in the distributional sense. Summarizing, we have obtained the complete pointwise behavior of $R_{\beta}$ at the points of the orbit $G_{\vartheta}\cdot 1$.
\begin{theorem}
\label{wnwRWfunctionth1} Let $r\in G_{\vartheta}\cdot 1$. Then $R_{\beta}\in C^{\infty}_{w}(r)$ for any $\beta\in\mathbb{C}$. Moreover, the Dirichlet series
\begin{equation}
\label{wnwRWfunctioneq8.17}
\zeta_{r}(z)=\sum_{n=1}^{\infty}\frac{e^{i\pi rn^{2}}}{n^{z}} \ \ \ (\mathrm{C}), \ \ \ z\in\mathbb{C},
\end{equation}
defines an entire function in $z$, where the sums of series for $\Re e\:z<1$ are taken in the Ces\`{a}ro sense, and they are convergent on the closed half-plane $\Re e\:z\geq 1$. In particular, the \L ojasiewicz point values of the derivatives of $R_{\beta}$ at points of the orbit $G_{\vartheta}\cdot 1$ are given by
\begin{equation}
\label{wnwRWfunctioneq8.18}
R^{(m)}_{\beta}\left(r\right)= (i\pi)^{m} \zeta_{r}(2\beta-2m), \ \ \mbox{distributionally}, \ \ \ \mbox{for all } m\in\mathbb{N}.
\end{equation}
\end{theorem}
\begin{proof} We have already shown everything except the convergence of the series (\ref{wnwRWfunctioneq8.17}) for $\Re e\:z=1$. But since $\zeta_{r}$ admits an analytic continuation beyond $\Re e\:z=1$, the convergence of
\begin{equation}
\label{wnwserieseq8.19}
\sum_{n=1}^{\infty}\frac{e^{i\pi r n^{2}}}{n^{1+yi}}
\end{equation}
is then a consequence of the Newman-Ingham Tauberian theorem for Dirichlet series \cite{ingham,korevaarbook,newman}.
\end{proof}

It is implicit in Theorem \ref{wnwRWfunctionth1} that $R_{\beta}$ admits a weak-asymptotic series at the points of $G_{\vartheta}\cdot 1$.
\begin{corollary}
\label{wnwRWcexp} Let $r\in G_{\vartheta}\cdot 1$. Then, for any $\beta\in\mathbb{C}$,
\begin{equation*}
R_\beta \left(r+\varepsilon t\right)\sim \sum_{m=0}^{\infty} \frac{\zeta_{r}(2\beta-2m)}{m!}(i\varepsilon\pi t)^{m} \ \ \ \mbox{as } \varepsilon\to0^{+} \mbox{ in } \mathcal{S}'(\mathbb{R}_{t}).
\end{equation*}
\end{corollary}


We now proceed to study the pointwise properties of $R_{\beta}$ on the orbit $G_{\vartheta}\cdot 0$. As usual, we start with $R_{0}$. We use the pointwise space defined by (\ref{wnwspaceO}) in the next proposition.

\begin{proposition}\label{wnwrwl2} At any point $r\in G_{\vartheta}\cdot 0$, there exists a constant $\mathfrak{p}_r\in\mathbb{C}$ such that
$$
R_0(t)-\frac{\sqrt{i}}{2}\mathfrak{p}_r(t-r+i0)^{-\frac{1}{2}}+1/2\in \mathcal{O}^{\infty}(r).
$$
Moreover, the constants $\mathfrak{p}_r$ are completely determined by the transformation equations:
\begin{equation}
\label{reseq}
\mathfrak{p}_0=1, \ \ \ \mathfrak{p}_{K^{2}r}=\mathfrak{p}_r,  \ \ \ \mbox{and}\ \ \ \mathfrak{p}_{Ur}=\sqrt{-\frac{i}{r}}\:\mathfrak{p}_r.
\end{equation}
\end{proposition}
Proposition \ref{wnwrwl2} means that, at any point of the orbit $G_{\vartheta}\cdot 0$, we have the weak-asymptotic expansion
\begin{equation}
\label{wnwrweq8.23}
R_0(r+\varepsilon t)= \frac{\sqrt{i}}{2}\mathfrak{p}_r\varepsilon^{-\frac{1}{2}}( t+i0)^{-\frac{1}{2}}-\frac{1}{2}+ o(\varepsilon^{\infty}) \ \ \ \mbox{as }\varepsilon\to0^{+} \mbox{ in }\mathcal{S}'(\mathbb{R}_{t}).
\end{equation}
\begin{proof} Invoking again $R_0=(\vartheta-1)/2$, it is enough to show that $\vartheta(t)-b_r(t-r+i0)^{-\frac{1}{2}}\in \mathcal{O}^{\infty}(r)$, $r\in G_{\vartheta}\cdot 0$, where $
b_r=\sqrt{i}\:\mathfrak{p}_r$. The property is satisfied at $r=0$ because of (\ref{wnwdseq8.14}). It then suffices to show that if $\vartheta(t)-b_r(t-r+i0)^{-\frac{1}{2}}\in \mathcal{O}^{\infty}(r)$, $r\neq0$, then $\vartheta(t)-b_{Ur}(t-Ur+i0)^{-\frac{1}{2}}\in \mathcal{O}^{\infty}(Ur)$, with
$b_{Ur}=(-i/r)^{1/2}b_r$. A similar argument to the one applied in the proof of Lemma \ref{wnwdsl1} shows that the latter assumption gives the existence of a sequence of positive numbers $\left\{C_{\nu}\right\}_{\nu=0}^{\infty}$ and a sequence of positive integers $\left\{k_{\nu}\right\}_{\nu=0}^{\infty}$ such that, for all $\nu\in\mathbb{N}$, $\varepsilon\in(0,1]$ and $0<\left|z\right|\leq1$,
$$
\left|\vartheta(r+\varepsilon z)-b_r(\varepsilon z)^{-\frac{1}{2}}\right|\leq  \frac{C_{\nu}\varepsilon^{\nu+1}}{(\Im m\:z)^{k_{\nu}}}.
$$
Consequently, because of (\ref{wnwdseq8.16}), for suitable $\varepsilon_{0}>0$ and $\left\{M_{\nu}\right\}_{\nu=0}^{\infty}$, we obtain, for all $\nu\in\mathbb{N}$, $\varepsilon\in(0,\varepsilon_0]$ and $0<\left|z\right|\leq1$,
$$
\left|\vartheta\left(-\frac{1}{r}+\varepsilon z\right)- \sqrt{-\frac{i}{r}}\:b_r(\varepsilon z)^{-\frac{1}{2}}\right|\leq  \frac{M_{\nu}\varepsilon^{\nu+1}}{(\Im m\:z)^{k_{\nu}}},
$$
and, by Theorem \ref{wnwth6.1}, we conclude $\vartheta(t)-b_{Ur}(t-Ur+i0)^{-\frac{1}{2}}\in \mathcal{O}^{\infty}(Ur)$, as required.
\end{proof}

Depending on whether $\beta=1/2$ or $\beta\neq1/2$, the distributions $R_{\beta}$ will behave differently on the orbit of $0$ under the theta group. This fact is intimately connected with the analytic continuation of $\zeta_{r}$ for $r\in G_{\vartheta}\cdot 0$, which is obtained in the next proposition.

\begin{proposition}
\label{wnwrwp1}
Let $r\in G_{\vartheta}\cdot 0$. Then, $\zeta_{r}$ admits an analytic continuation to $\mathbb{C}\setminus\left\{1\right\}$. Furthermore, $\zeta_{r}$ has a simple pole at $z=1$ with residue $\mathfrak{p}_{r}$, determined by (\ref{reseq}), and the entire function
\begin{equation}
\label{wnwrweq8.24} A_{r}(z)=\zeta_{r}(z)-\frac{\mathfrak{p}_{r}}{z-1}
\end{equation}
can be expressed as the Ces\`{a}ro limit
\begin{equation}
\label{wnwrweq8.25} A_{r}(z)= \lim_{x\to\infty} \sum_{1\leq n<x} \frac{ e^{i\pi r n^{2}}}{n^{z}}-\mathfrak{p}_{r}\int_{1}^{x}\frac{\mathrm{d}\xi}{\xi^{z}}\ \ \ \mathrm{(C)}.
\end{equation}
\end{proposition}
\begin{proof}
Observe that the formula (\ref{wnwrweq8.24}) holds when $\Re e\:z>1$, so it suffices to prove that (\ref{wnwrweq8.25}) defines truly an entire function. Fourier transforming (\ref{wnwrweq8.23}),
we obtain the moment asymptotic expansion
$$
\sum_{n=1}^{\infty}e^{i\pi rn^{2}}\delta(\lambda u-\pi n^{2})- \frac{\mathfrak{p}_{r}}{2}(\lambda\pi u)_{+}^{-\frac{1}{2}}= -\frac{\delta(x)}{2\lambda}+o\left(\frac{1}{\lambda^{\infty}}\right) \ \ \ \mbox{as }\lambda\to\infty
$$
in $\mathcal{S}'(\mathbb{R}_{u})$, and so,
\begin{equation*}
f(u)=\sum_{n=1}^{\infty}e^{i\pi rn^{2}}\delta(\lambda u-\pi n^{2})- \frac{\mathfrak{p}_{r}}{2}(\pi u)_{+}^{-\frac{1}{2}}\in\mathcal{K}'(\mathbb{R}_{u}).
\end{equation*}
Set
$$g(u)=\sum_{n=1}^{\infty}e^{i\pi r n^{2}}\delta(u-\pi n^{2})-\frac{\mathfrak{p}_{r}}{2}(\pi u)^{-\frac{1}{2}}H(u-\pi),$$
where $H$ is the Heaviside function. We show that $g\in \mathcal{K}'(\mathbb{R})$. Indeed, $g-f\in \mathcal{E}'(\mathbb{R})\subset\mathcal{K}'(\mathbb{R})$, where $\mathcal{E}'(\mathbb{R})$ is the space of compactly supported distributions. Now, $\operatorname*{supp}g=[\pi,\infty)$, then its moment function $\left\langle g(u),u^{z}\right\rangle$ is entire and can be computed in the Ces\`{a}ro sense \cite[Chap. 6]{estrada-kanwal2}, hence the function
$$
A_{r}(z)=\pi^{-z}\left\langle g(u),u^{-\frac{z}{2}}\right\rangle$$
is entire. It coincides with (\ref{wnwrweq8.25}) because
\begin{align*}
\pi^{-z}\left\langle g(u),u^{-\frac{z}{2}}\right\rangle&= \left\langle \sum_{n=1}^{\infty}e^{i\pi rn^{2}}\delta(\xi-n)- \mathfrak{p}_{r}H(\xi-1),\xi^{-z}\right\rangle
\\
&
=\lim_{x\to\infty} \sum_{1\leq n<x} \frac{ e^{i\pi r n^{2}}}{n^{z}}-\mathfrak{p}_{r}\int_{1}^{x}\frac{\mathrm{d}\xi}{\xi^{z}}\ \ \ \mathrm{(C)}.
\end{align*}
\end{proof}
We are ready to describe the pointwise behavior of $R_{\beta}$ on $G_{\vartheta}\cdot 0$. We define the \emph{generalized gamma constant} as
$$\gamma_{r}:=A_{r}(1).$$
Observe that in fact $\gamma_{0}=\gamma$, the familiar Euler gamma constant because $\zeta_{0}=\zeta$ is the Riemann zeta function.
\begin{theorem}
\label{wnwrwth2} Let $r\in G_{\vartheta}\cdot 0$.
\begin{itemize}
\item [(i)] If $\beta\in\mathbb{C}\setminus\left\{1/2\right\}$, then
$$
R_{\beta}(r+\varepsilon t)\sim \frac{(-i\pi)^{\beta-\frac{1}{2}}\Gamma\left(\frac{1}{2}-\beta\right)\mathfrak{p}_{r}}{2} (\varepsilon t+i0)^{\beta-\frac{1}{2}}+\sum_{m=0}^{\infty} \frac{\zeta_{r}(2\beta-2m)}{m!}(i\varepsilon\pi t)^{m},
$$
as $\varepsilon\to 0^{+}$ in $\mathcal{S}'(\mathbb{R}_{t})$.
\item [(ii)] When $\beta=1/2$, we have
$$
R_{\frac{1}{2}}(r+\varepsilon t)\sim \gamma_{r}+\frac{\mathfrak{p}_{r}}{2}\left(-\log\left(\frac{\varepsilon\left| t\right|}{\pi}\right)+\frac{i\pi}{2}\mathrm{sgn}t-\gamma\right)+\sum_{m=1}^{\infty} \frac{\zeta_{r}(1-2m)}{m!}(i\varepsilon\pi t)^{m},
$$
as $\varepsilon\to 0^{+}$ in $\mathcal{S}'(\mathbb{R}_{t})$.
\end{itemize}
\end{theorem}
\begin{proof}
Let $g\in\mathcal{K}'(\mathbb{R}_{u})$ be as in the proof of Proposition \ref{wnwrwp1}. Then it can be multiplied by $(\pi/u)^{\beta}$ and actually 
$$g_{\beta}(u):=(\pi/u)^{\beta}g(u)\in\mathcal{K}'(\mathbb{R}_{u}).$$
The moments of $g_{\beta}$ are given by $\pi^{m}A_{r}(2\beta-2m)$, $m\in\mathbb{N}.$

Assume first that $\beta\neq1/2$. The distribution $$h_{\beta}(u)=\frac{\pi^{\beta-\frac{1}{2}}}{2} u_{+}^{-\beta-1/2}(H(u)-H(u-\pi))$$ has moments
$$
\left\langle
h_{\beta}(u),u^{m}\right\rangle=\frac{\pi^{\beta-\frac{1}{2}}}{2}\mbox{F.p.}\int_{0}^{\pi}u^{m-\beta-\frac{1}{2}}\mathrm{d}u=-\frac{\pi^{m}}{2\beta-2m-1},
$$
where F.p. stands for the Hadamard finite part \cite{estrada-kanwal2}. So $g_{\beta}-\mathfrak{p}_{r}h_{\beta}$ has moments $\pi^{m}\zeta_{r}(2\beta-2m)$ which implies the moment asymptotic expansion
$$
\sum_{n=1}^{\infty}\frac{e^{i\pi rn^{2}}}{n^{2\beta}}\delta(\lambda-un^{2})\sim \frac{\pi^{\beta}\mathfrak{p}_{r}}{2\sqrt{\pi}} ( \lambda u)_{+}^{-\beta-\frac{1}{2}}+ \sum_{m=0}^{\infty} \frac{(-\pi)^{m}\zeta_{r}(2\beta-2m)} {m!\lambda^{m+1}} \delta^{(m)}(u)
$$
as $\lambda\to\infty$ in $\mathcal{K}'(\mathbb{R}_{u})$. Part (i) follows now by taking Fourier inverse transform.

Suppose now that $\beta=1/2$ and set
$$
h_{\frac{1}{2}}(u)=\frac{1}{2}\left(\mathrm{Pf}\left(\frac{H(u)}{u}\right)- \frac{H(u-\pi)}{u}\right),
$$ 
where
$\mathrm{Pf}(H(u)/u)$ is the usual \cite{estrada-kanwal2}
regularization of $H(u)/u$ by finite part. The moments of
$h_{1/2}$ are given by
$$
\left\langle
h_{\frac{1}{2}}(u),u^{m}\right\rangle=
\left\{ \begin{array}{ll}
\pi^{m}/(2m), & m=1,2,3,\dots ,\\
(\log \pi)/2, & m=0.\\
\end{array}\right.
$$
The moment asymptotic expansion of $g_{1/2}-\mathfrak{p}_{r}h_{1/2}$ yields the weak-asymptotic expansion, as $\lambda\to\infty$ in $\mathcal{K}'(\mathbb{R}_u)$,
$$
\frac{e^{i\lambda r u}}{2\pi}\hat{R}_{\frac{1}{2}}(\lambda u)
~\sim \frac{\mathfrak{p}_{r}}{2}\mathrm{Pf}\left(\frac{H(\lambda u)}{\lambda u}\right)+\frac{b\delta(u)}{\lambda} +\sum_{m=1}^{\infty} \frac{(-\pi)^{m}\zeta_{r}(1-2m)} {m!\lambda^{m+1}} \delta^{(m)}(u),
$$
where $b=(\mathfrak{p}_{r}/2)\log\pi+\gamma_{r}$,
which in turn proves Part (ii) after taking Fourier inverse transform.
\end{proof}

In the rest of this subsection we discuss some useful formulas which can be derived from our previous analysis. The next corollary provides formulas for the constants $\mathfrak{p}_{r}$.

\begin{corollary}\label{wnwrwc1}
Let $r\in G_{\vartheta}\cdot 0$. Then
\begin{equation}
\label{formulap}
\lim_{N\to\infty}\frac{1}{N}\sum_{n=1}^{N}e^{i\pi rn^{2}}=\mathfrak{p}_{r}.
\end{equation}
\end{corollary}
\begin{proof}
The relation (\ref{formulap}) follows directly from the Ikehara theorem \cite[p. 122]{korevaarbook} and Proposition \ref{wnwrwp1}.
\end{proof}
We now give a formula for $\gamma_{r}$.

\begin{corollary}\label{wnwrwc2}
Let $r\in G_{\vartheta}\cdot 0$. The series
\begin{equation}
\label{wnwconvs1}
\sum_{n=0}^{\infty}\frac{e^{i\pi rn^{2}}-\mathfrak{p}_{r}}{n^{1+iy}}\:,
\end{equation}
is convergent for any $y\in\mathbb{R}$. In particular,
\begin{equation}
\label{formulagamma}
\sum_{n=0}^{\infty}\frac{e^{i\pi rn^{2}}-\mathfrak{p}_{r}}{n}=\gamma_{r}-\mathfrak{p}_{r}\gamma
\end{equation}
or equivalently,
\begin{equation}
\label{formulagamma2}
\lim_{x\to\infty}\sum_{n=1}^{N}\frac{e^{i\pi rn^{2}}}{n}-\mathfrak{p}_{r}\log N=\gamma_{r}.
\end{equation}
\end{corollary}
\begin{proof}
By Proposition \ref{wnwrwp1}, $\zeta_{r}-\mathfrak{p}_{r}\zeta$ is an entire function. Thus, the convergence of (\ref{wnwconvs1}) is implied by the Newman-Ingham Tauberian theorem for Dirichlet series \cite{ingham,korevaarbook,newman}. The equality (\ref{formulagamma}) holds because $(\zeta_{r}-\zeta)(1)=\gamma_{r}-\gamma$ (Newman-Ingham theorem again), while (\ref{formulagamma2}) is an easy consequence of (\ref{formulagamma}) and the well known relation $\sum_{n=1}^{N}1/n=\log N+\gamma+o(1)$.
\end{proof}

The convergence of (\ref{wnwserieseq8.19}) is interesting by itself, so we state it in the following corollary.

\begin{corollary}
\label{wnwRWc3} The series
$$\sum_{n=1}^{\infty}\frac{e^{i\pi n^{2}\frac{2j+1}{2\nu+1}}}{n^{1+yi}}$$
is convergent for any $j,\nu\in\mathbb{Z}$ and $y\in \mathbb{R}$.
\end{corollary}

The pointwise behavior of $R_{0}$  can be used to calculate some Ces\`{a}ro sums and limits which apparently have not been given elsewhere before. That is the context of the next corollary, whose proof is obtained immediately by comparing Corollary \ref{wnwRWcexp} with Lemma \ref{wnwdsl1} and the expansion from Theorem \ref{wnwrwth2} with (\ref{wnwrweq8.23}).

\begin{corollary} For any $j,\nu\in\mathbb{Z}$,
\label{wnwRWc4}
\begin{equation*}
\sum_{n=1}^{\infty} n^{2m} e^{i\pi n^{2}\frac{2j+1}{2\nu+1}}=0 \ \ \ \mathrm{(C)}, \ \ \ m=1,2,3,\dots ,
\end{equation*}
and
\begin{equation*}
\sum_{n=1}^{\infty}  e^{i\pi n^{2}\frac{2j+1}{2\nu+1}}=-\frac{1}{2} \ \ \ \mathrm{(C)}.
\end{equation*}

If $r\in G_{\vartheta}\cdot 0$, then
\begin{equation*}
\lim_{x\to\infty}\left(\sum_{1\leq n<x} n^{2m} e^{i\pi r n^{2}}-\mathfrak{p}_{r}\int_{0}^{x}\xi^{2m}\mathrm{d}\xi\right)=0 \ \ \ \mathrm{(C)}, \ \ \ m=1,2,3,\dots ,
\end{equation*}
and
\begin{equation*}
\lim_{x\to\infty}\left(\sum_{1\leq n <x} e^{i\pi r n^{2}}- \mathfrak{p}_{r}x\right)=-\frac{1}{2} \ \ \ \mathrm{(C)}.
\end{equation*}
\end{corollary}

So, Corollary \ref{wnwRWc4} tells us that the values of all generalized zeta functions coincide at the nonpositive even integers, and actually $\zeta_{r}(0)=\zeta(0)=-1/2$ and $\zeta_{r}(-2m)=\zeta(-2m)=0$ for $m=1,2,3,\dots\ $.

\subsection{Tauberian Theorems for Laplace Transforms}
\label{wnwLT} We now apply the results from Subsection \ref{wnwtphi} to Laplace transforms. As in Example \ref{wnwex3.5}, $\Gamma$ is assumed to be a closed convex acute cone with vertex at the origin; we set $C_{\Gamma}=\operatorname*{int} \Gamma^{\ast}$ and $T^{C_{\Gamma}}=\mathbb{R}^{n}+iC_{\Gamma}$. The following Tauberian theorems for the Laplace transform were originally obtained in \cite{drozhzhinov-z0,vladimirov-d-z1} under the additional assumption that $\Gamma$ is a regular cone (i.e., its Cauchy-Szeg\"{o} kernel is a divisor of the unity in the Vladimirov algebra $H(T^{C_{\Gamma}})$ \cite{vladimirovbook,vladimirov-d-z1}); we will not make use of such a hypothesis over the cone $\Gamma$.

Given $\kappa\geq0$, we denote by $\Omega^{\kappa}\subset \mathbb{H}^{n+1}$ the set
\begin{equation}
\label{wnwLeq1}\Omega^{\kappa}=\left\{(x,\sigma)\in\mathbb{H}^{n+1}:\:\left|x\right|\leq\sigma^{\kappa}\mbox{ and }0<\sigma\leq1 \right\}.
\end{equation}
\begin{theorem}
\label{wnwLth1} Let $\mathbf{h}\in\mathcal{S}'_{\Gamma}(E)$ and let $L$ be slowly varying at infinity. Then, $\mathbf{h}$ is weak-asymptotically bounded of degree $\alpha$ at infinity with respect to $L$ if and only if there exist numbers $k\in\mathbb{N}$ and $0\leq\kappa<1$ and a vector $\omega\in C_{\Gamma}$ such that
\begin{equation}
\label{wnwLeq2}
\limsup_{\varepsilon\to 0^{+}}\sup_{(x,\sigma)\in\partial\Omega^{\kappa},\: \sigma>0} \frac{\sigma^{k}\varepsilon^{n+\alpha}}{L(1/\varepsilon)}\left\|\mathcal{L}\left\{\mathbf{h};\varepsilon\left(x+i\sigma\omega\right)\right\}\right\|<\infty.
\end{equation}
\end{theorem}
\begin{proof} Set $\hat{\mathbf{f}}=(2\pi)^{n}\mathbf{h}$ and keep the notation from Example \ref{wnwex3.5}.  Clearly, $\mathbf{h}$ is weak-asymptotically bounded of degree $\alpha$ at infinity with respect to $L$ if and only if $\mathbf{f}$ is weak-asymptotically bounded of degree $-\alpha-n$ at the origin with respect to $L(1/\varepsilon)$. The latter holds, by (\ref{wnweq3.9}) and Theorem \ref{wnwth6.1}, if and only if there exists $k_{1}\in\mathbb{N}$ such that
\begin{equation}
\label{wnwLeq3}
\limsup_{\varepsilon\to 0^{+}}\sup_{\stackrel{\left|x\right|^{2}+(\cos \vartheta)^2=1}{\vartheta\in[0,\pi/2)}} \frac{\left(\cos\vartheta\right)^{k_1}\varepsilon^{n+\alpha}}{L(1/\varepsilon)}\left\|F_{\phi_{\omega}}\mathbf{f}(\varepsilon x,\varepsilon \cos\vartheta)\right\|<\infty.
\end{equation}
Thus, we shall show the equivalence between (\ref{wnwLeq2}) and (\ref{wnwLeq3}). By part (i) of Proposition \ref{wnwp5.1}, (\ref{wnwLeq3}) implies (\ref{wnwLeq2}). Assume now (\ref{wnwLeq2}), namely, there exist $C_{1}$ and $0<\varepsilon_{0}<1$ such that
\begin{equation}
\label{wnwLeq4}
\left\|F_{\phi_{\omega}}\mathbf{f}(\varepsilon x',\varepsilon \sigma)\right\|< \frac{C_{1}}{\sigma^{k}}\varepsilon^{-\alpha-n}L\left(1/\varepsilon\right), \ \ \ \varepsilon\leq\varepsilon_{0},\ (x',\sigma)\in\Omega^{\kappa}.
\end{equation}
We may assume that $k\geq\alpha+n+1$ and $L$ satisfies
(\ref{wnweq5.3}) and (\ref{wnweq5.4}) (the case at infinity). We
keep arbitrary $\varepsilon<\varepsilon_{0}$,
$\vartheta\in(0,\pi/2)$ and $x\in\mathbb{R}^{n}$ with
$\left|x\right|^2+(\cos\vartheta)^{2}=1$. Set
$$r=\left|x\right|^{\frac{1}{1-\kappa}}/(\cos\vartheta)^{\frac{\kappa}{1-\kappa}},\ 
x'=x/r \mbox{ and } \sigma=(\cos\vartheta)/r.$$ Observe that
$(x',\sigma)\in\partial\Omega^{\kappa}$. Assume first that
$r\varepsilon\leq\varepsilon_{0}$, then, in view of
(\ref{wnwLeq4}) and (\ref{wnweq5.3}),
\begin{align*}
\left\|F_{\phi_{\omega}}\mathbf{f}(\varepsilon x,\varepsilon \cos\vartheta)\right\|
&<\frac{C_{1}}{(\cos\vartheta/r)^{k}}(r\varepsilon)^{-\alpha-n}L\left(1/(r\varepsilon)\right)
\\
&
\leq4C_{1}C_{2} \varepsilon^{-\alpha-n}L\left(1/\varepsilon\right)
(\cos\vartheta)^{ -k-\frac{\kappa}{1-\kappa}(k-\alpha-n+1)  };
\end{align*}
on the other hand, if now $\varepsilon_{0}<r\varepsilon$, Proposition \ref{wnwp2}  implies that for some $k_{2}\in\mathbb{N}$, $k_2\leq k$ and $C_{4}>0$,
\begin{align*}
\left\|F_{\phi_{\omega}}\mathbf{f}(\varepsilon x,\varepsilon \cos\vartheta)\right\|
&<\frac{C_{4}}{(\varepsilon\cos\vartheta)^{k_{2}}}
=\frac{C_{4}}{(\cos\vartheta)^{k_2}}\varepsilon^{-\alpha-n}L\left(1/\varepsilon\right)\frac{(1/\varepsilon)^{k_{2}-\alpha-n}}{L(1/\varepsilon)}
\\
&
<\frac{C_{4}C_{3}}{(\cos\vartheta)^{k_2}}\varepsilon^{-\alpha-n}L\left(1/\varepsilon\right)\left(\frac{r}{\varepsilon_{0}}\right)^{k_{2}+1-\alpha-n}
\\
&
<\frac{C_{4}C_{3}}{\varepsilon_{0}^{k_{2}+1-\alpha-n}}\varepsilon^{-\alpha-n}L\left(1/\varepsilon\right)(\cos\vartheta)^{ -k_2-\frac{\kappa}{1-\kappa}(k_2-\alpha-n+1)  },
\end{align*}
where we have used (\ref{wnweq5.4}). Therefore, (\ref{wnwLeq3}) is satisfied with $k_{1}\geq k_2+\kappa(k_2-\alpha-n+1)/(1-\kappa)$.
\end{proof}

We obtain as a corollary the so called general Tauberian theorem for Laplace transforms \cite[p. 84]{vladimirov-d-z1}.
\begin{corollary}\label{wnwLc1}
Let $\mathbf{h}\in\mathcal{S}'_{\Gamma}(E)$ and let $L$ be slowly varying at infinity. Then, an estimate (\ref{wnwLeq2}), for some $k\in\mathbb{N}$ and $\omega\in C_{\Gamma}$, and the existence of a solid cone $C'\subset C_{\Gamma}$ (i.e., $\operatorname*{int}C'\neq \emptyset$) such that
\begin{equation}
\label{wnwLeq5}
\lim_{\varepsilon\to0^{+}} \frac{\varepsilon^{\alpha+n}}{L(1/\varepsilon)}\mathcal{L}\left\{\mathbf{h};i\varepsilon \xi\right\}=\mathbf{G}(i\xi), \mbox{ in } E, \ \ \ \mbox{for each }\xi\in C',
\end{equation}
are necessary and sufficient for $\mathbf{h}$ to have weak-asymptotic behavior at infinity of degree $\alpha$, i.e.,
$$
\mathbf{h}(\lambda u)\sim \lambda^{\alpha}L(\lambda)\mathbf{g}(u)\ \ \ \mbox{as }\lambda\to\infty \mbox{ in } \mathcal{S'}(\mathbb{R}^{n},E), \ \  \mbox{for some }\mathbf{g}\in\mathcal{S}'_{\Gamma}(E).
$$
In such a case, $\mathbf{G}(z)=\mathcal{L}\left\{\mathbf{g};z\right\}$, $z\in T^{C_{\Gamma}}$.
\end{corollary}
\begin{proof} Recall \cite{vladimirovbook} that $\mathcal{S}'_{\Gamma}(E)$ is canonically isomorphic to $\mathcal{S}'(\Gamma,E)=L_{b}(\mathcal{S}(\Gamma),E)$. By the injectivity of the Laplace transform and the uniqueness property of holomorphic functions, the linear span of $$\left\{e^{i \xi\cdot u}: \xi\in C'\right\}$$ is dense in $\mathcal{S}(\Gamma)$; observe that (\ref{wnwLeq5}) gives precisely convergence of $$\frac{1}{\lambda^{\alpha}L(\lambda)}\mathbf{h}(\lambda\: \cdot)$$
over such a dense subset. To conclude the proof, it suffices to apply Theorem \ref{wnwLth1} and the Banach-Steinhaus theorem.
\end{proof}
\begin{example}\label{wnwapexl} \emph{Littlewood's Tauberian theorem.} The classical Tauberian theorem of Littlewood \cite{hardy,korevaarbook,littlewood} states that if

\begin{equation} \label{wnwLeq6} \lim_{\varepsilon \to 0^+} \sum_{n = 0}^{\infty}
c_n e^{-\varepsilon n} = \beta
\end{equation}
and if the Tauberian hypothesis $ c_n = O(1/n) $ is satisfied,
then the numerical series is convergent, i.e.,
$\sum_{n = 0}^{\infty}
c_n = \beta.$

We give a quick proof of this theorem based on Corollary \ref{wnwLc1}. We first show that the distribution $h(u)=\sum_{n=0}^{\infty}c_{n}\delta(u-n)$ has the weak-asymptotic behavior
\begin{equation}\label{wnwLeq7}
h(\lambda u)=\sum_{n=0}^{\infty}c_{n}\delta(\lambda u-n)\sim \beta \frac{\delta(u)}{\lambda} \ \ \ \mbox{as }\lambda\to\infty\ \mbox{in }\mathcal{S}'(\mathbb{R}_{u}).
\end{equation}
Observe that (\ref{wnwLeq5}) is an immediate consequence of (\ref{wnwLeq6}) (here $n=1$, $\alpha=-1$, $L\equiv1$). We verify (\ref{wnwLeq2}) with $\kappa=0$, actually, on the rectangle $\Omega^{0}=[-1,1]\times(0,1]$. Indeed, (\ref{wnwLeq6}) and the Tauberian hypothesis imply that for suitable constants $C_{1},C_{2},C_{3},C_{4}>0$, independent of $(x,\sigma)\in\Omega^{0}$,
\begin{align*}
\left|\mathcal{L}\left\{h;\varepsilon(x+i\sigma)\right\}\right|& =\left|\sum_{n=0}^{\infty}c_{n}e^{-\varepsilon \sigma n}e^{i\varepsilon x n}\right|\leq C_1+C_2\sum _{n=1}^{\infty} \frac{e^{-\varepsilon \sigma n}}{n}\left|e^{i\varepsilon x n}-1\right|
\\
& <C_{1}+C_{3} \varepsilon \sum _{n=1}^{\infty}e^{-\varepsilon \sigma n}
<\frac{C_{4}}{\sigma}, \ (x,\sigma)\in\Omega^{0},\ 0<\varepsilon\leq1.
\end{align*}
Consequently, Corollary \ref{wnwLc1} yields (\ref{wnwLeq7}). Finally, it is well known that (\ref{wnwLeq7}) and $c_{n}=O(1/n)$ imply the convergence of the series; in fact, this is true under more general Tauberian hypotheses (cf. \cite[Sec. 3]{vindas-estrada2}). We can proceed as follows. Let $\sigma>1$ be arbitrary. Choose $\rho\in\mathcal{D}(\mathbb{R})$ such that $0\leq\rho\leq1$, $\rho(u)=1$ for $u\in[0,1]$, and $\operatorname*{supp}\rho\subset[-1,\sigma]$, then, evaluation of (\ref{wnwLeq7}) at $\rho$ gives, for some constant $C_{5}$,
\begin{align*}\limsup_{\lambda\to\infty}\left|\sum_{0\leq n\leq\lambda}c_{n}-\beta\right| & \leq \limsup_{\lambda\to\infty}\left|\sum_{\lambda\leq n}c_{n}\rho\left(\frac{n}{\lambda}\right)\right|
\\
&
 < C_{5}\limsup_{\lambda\to\infty}\frac{1}{\lambda}\sum_{1<\frac{n}{\lambda}<\sigma}\frac{\lambda}{n}\:\rho\left(\frac{n}{\lambda}\right)
\\
&
=C_{5}\int_{1}^{\sigma}\frac{\rho(x)}{x}\:\mathrm{d}x<C_{5}(\sigma-1),
\end{align*}
and so, taking $\sigma\to1^{+}$, we conclude $\sum_{n=0}^{\infty}c_{n}=\beta$.
\end{example}
\begin{remark} We refer to the monograph \cite{vladimirov-d-z1} (and references therein) for the numerous applications of Corollary \ref{wnwLc1} in mathematical physics, especially in quantum field theory. Corollary \ref{wnwLc1} can also be  used to easily recover Vladimirov multidimensional generalization \cite{vladimirov1} of the Hardy-Littlewood-Karamata Tauberian theorem (cf. \cite{drozhzhinov-z0,vladimirov-d-z1}).
\end{remark}

\subsection{Relation between Weak-asymptotics in the Spaces $\mathcal{D}'(\mathbb{R}^{n},E)$ and $\mathcal{S}'(\mathbb{R}^{n},E)$} \label{wnwDS}

If a tempered $E$-valued distribution has weak-asymptotic behavior in the space $\mathcal{S}'(\mathbb{R}^{n},E)$ then, clearly, it has the same weak-asymptotic behavior in $\mathcal{D}'(\mathbb{R}^{n},E)$. The converse is also well known in the case of scalar-valued distributions, but the true of this result is less obvious. There have been several proofs of such a converse result and, remarkably, none of them is simple (cf. \cite{meyer,pilipovic3,vindas-estrada6,vindas-pilipovic1} and especially \cite[Lem. 6]{zavialov88} for the general case). We provide a new proof of this fact, which will actually be derived as an easy consequence of the results from Subsection \ref{wnwtphi}.

We begin with weak-asymptotic boundedness. Let $L$ be slowly varying at the origin (resp. at infinity).

\begin{proposition}
\label{wnwDSp1} Let $\mathbf{f}\in\mathcal{S}'(\mathbb{R}^{n},E)$. If $\mathbf{f}$ is weak-asymptotically bounded of degree $\alpha$ at the point $x_{0}$ (resp. at infinity) with respect to $L$ in the space $\mathcal{D}'(\mathbb{R}^{n},E)$, so is $\mathbf{f}$ in the space $\mathbf{f}\in\mathcal{S}'(\mathbb{R}^{n},E)$.
\end{proposition}
\begin{proof} We may assume that $x_{0}=0$. We will show both assertions at 0 and $\infty$ at the same time. The Banach-Steinhaus theorem implies the existence of $\nu\in\mathbb{N}$, $C>0$, and $h_{0}>0$ such that
$$
\left|\left\langle \mathbf{f}(ht),\rho(t)\right\rangle\right|\leq C h^{\alpha}L(h)\sup_{\left|t\right|\leq 1,\: \left|m\right|\leq \nu}\left|\rho^{(m)}(t)\right|, \ \ \ \mbox{for all }\rho\in\mathcal{D}(B(0,3))
$$
and all $0<h<h_{0}$ (resp.  $h_{0}<h$), where $B(0,3)$ is the ball of radius 3. Let now $\phi\in\mathcal{D}(B(0,1))$ be so that $\mu_{0}(\phi)=1$. If we take $\rho(t)=y^{-n}\phi(y^{-1}(t-x))$ in the above estimate, where $0<y<1$ and $\left|x\right|\leq1$, we then obtain at once that (\ref{wnwtteq1}) is satisfied with $k=\nu+n$, and consequently the Theorem \ref{wnwth6.1} implies the result.
\end{proof}
Proposition \ref{wnwDSp1}, the Banach-Steinhaus theorem, and the density of $\mathcal{D}(\mathbb{R}^{n})$ in $\mathcal{S}(\mathbb{R}^{n})$ immediately yield what we wanted:
\begin{corollary}
\label{wnwDSc1}
If $\mathbf{f}\in\mathcal{S}'(\mathbb{R}^{n},E)$ has  weak-asymptotic behavior in the space $\mathcal{D}'(\mathbb{R}^{n},E)$, so does $\mathbf{f}$ have the same weak-asymptotic behavior in the space $\mathcal{S}'(\mathbb{R}^{n},E)$.
\end{corollary}

Corollary \ref{wnwDSc1} tells us then that the weak-asymptotics at finite points in $\mathcal{S}'(\mathbb{R}^{n},E)$ are local properties. Indeed, if $\mathbf{f_{1}}=\mathbf{f_2}$ in a neighborhood of $x_{0}\in\mathbb{R}^{n}$, then we easily deduce from Corollary \ref{wnwDSc1} that they have exactly the same weak-asymptotic properties at the point $x_{0}$ in the space $\mathcal{S}'(\mathbb{R}^{n},E)$.

\newpage

\section{Further Extensions}\label{wnwe}
We indicate in this section some useful extensions and variants of the Tauberian results from the previous sections.

\subsection{Other Tauberian Conditions} \label{wnwOT}The Tauberian conditions (\ref{wnweq5.5}) and (\ref{wnwtteq1}), occurring in Theorems \ref{wnwth5.1} -- \ref{wnwth9}, can be replaced by estimates of the form (\ref{wnwLeq2}), that is, one may use the boundary of some set $\Omega^{\kappa}$, $0\leq\kappa<1$ (cf. (\ref{wnwLeq1})), instead of the upper half sphere $\mathbb{H}^{n+1}\cap\mathbb{S}^{n}$. Specifically, the same argument given in proof of Theorem \ref{wnwLth1} applies to show that (\ref{wnweq5.1}) (and hence (\ref{wnwaeq1})) is equivalent to the estimate
$$
\limsup_{\varepsilon\rightarrow0^+}\sup_{(x,y)\in\partial\Omega^{\kappa}, \: y>0}\frac{y^k}{\varepsilon^{\alpha}L(\varepsilon)}\left\|M^{\mathbf{f}}_{\varphi}\left(x_0+\varepsilon
x,\varepsilon y\right)\right\|<\infty
$$
$$
\left(\mbox{resp. }\limsup_{\lambda\rightarrow\infty}\sup_{(x,y)\in\partial\Omega^{\kappa}, \: y>0}\frac{y^k}{\lambda^{\alpha}L(\lambda)}\left\|M^{\mathbf{f}}_{\varphi}\left(\lambda
x,\lambda y\right)\right\|<\infty\right)
$$
for some $0\leq\kappa<1$ and $k\in\mathbb{N}$ (the $k$ may be different numbers).
\subsection{Distributions with Values in DFS Spaces}\label{wnwDFS}
All the results from Sections \ref{wnwa}--\ref{wnwce} hold if we
replace the Banach space $E$ by a Silva \cite{silvas} inductive
limit of Banach spaces $E_n, n\in\mathbb N$, that is,
$$
E=\bigcup_{n=1}^{\infty}E_{n}=\mbox{ind}\lim_{n\to\infty}(E_n,||\cdot||_n) \:,
$$ where $E_{1}\subset
E_{2}\subset\dots$ and each injection $E_{n}\to E_{n+1}$ is compact.
These spaces are actually the DFS spaces (strong duals of
Fr\'{e}chet-Schwartz spaces). Particular examples are
$E=\mathcal{S}'(\mathbb{R}^{n}),\mathcal{S}'_{0}(\mathbb{R}^{n}),\mathcal{D}'(Y)$,
where $Y$ is a compact manifold, among many other important spaces
arising in applications.

In this situation $E$ is regular, namely, for any bounded set $\mathfrak{B}$ there exists
$n_0\in\mathbb N$ such that $\mathfrak{B}$ is bounded in $E_{n_{0}}$. Thus, our Abelian and Tauberian theorems from Sections \ref{wnwa}--\ref{wnwtt} for $E$-valued distributions are valid if we replace the norm estimates  by memberships in bounded sets of $E$. For instance, a condition such as (\ref{wnweq5.1}) should be replaced by one of the form: There exist $k\in\mathbb{N}$, $\varepsilon_{0}>0$, and a bounded set $\mathfrak{B}\subset E$ such that
\begin{equation}
\label{wnwDFSeq1}
\frac{y^{k}}{\varepsilon^{\alpha}L(\varepsilon)}M^{\mathbf{f}}_{\varphi}\left(x_0+\varepsilon
x,\varepsilon y\right)\in\mathfrak{B}, \ \ \ \mbox{for all } 0<\varepsilon\leq\varepsilon_{0}\mbox{ and } \left|x\right|^{2}+y^{2}=1\ ;
\end{equation}
and similarly for all other conditions occurring within these sections. As already observed, (\ref{wnwDFSeq1}) is equivalent to an estimate of the form (\ref{wnweq5.1}) in some norm $\left\|\:\cdot\:\right\|_{n_0}$, but the existence of the $n_{0}$ would be extremely hard to verify in applications and thus such a Tauberian condition would have no value in concrete situations. It is therefore desirable to have more realistic Tauberian conditions. We can achieve this if we use the Mackey theorem \cite[Thm. 36.2]{treves}, because the condition (\ref{wnwDFSeq1}) is then equivalent to the following one: There exists $k\in\mathbb{N}$ such that for each $e^{\ast}\in E'$
\begin{equation}
\label{wnwDFSeq2}
\limsup_{\varepsilon\rightarrow0^+}\sup_{\left|x\right|^2+y^2=1,\:y>0}\frac{y^k}{\varepsilon^{\alpha}L(\varepsilon)}\left|\left\langle e^{\ast},M^{\mathbf{f}}_{\varphi}\left(x_0+\varepsilon
x,\varepsilon y\right)\right\rangle\right|<\infty.
\end{equation}

 Since $E$ is a Montel space \cite{silvas,treves}, the limit condition (\ref{wnweq5.2}) can be replaced be the equivalent one: There exist the limits
\begin{equation}
\label{wnwDFSeq3}
\lim_{\varepsilon^{\alpha}\to0^{+}} \frac{1}{\varepsilon L(\varepsilon)}\left\langle e^{\ast},M^{\mathbf{f}}_{\varphi}(x_{0}+\varepsilon x,\varepsilon y)\right\rangle\in \mathbb{C},
\end{equation}
for all $e^{\ast}\in E'$ and $(x,y)\in\mathbb{H}^{n+1}$, and likewise for all other limit conditions.

Furthermore, the results from Section \ref{wnwce} are also valid in this context, if we use suitable hypotheses. For example, Theorem \ref{wnwceth1} remains true if we replace the hypotheses (i) and (ii) by:
\begin{itemize}
\item [(i)$'$] $\mathcal{W}_{\psi}\mathbf{f}(x,y)\in E$ \emph{for all} $(x,y)\in\mathbb{H}^{n+1}$ and it is continuous as an $E$-valued function.
\item [(ii)$'$] There exist $k,l\in\mathbb{N}$ such that for each $e^{\ast}\in E'$
\begin{equation*}
\sup_{(x,y)\in\mathbb{H}^{n+1}}\left(\frac{1}{y}+y\right)^{-k}\left(1+\left|x\right|\right)^{-l}\left|\left\langle e^{\ast},\mathcal{W}_{\psi}\mathbf{f}(x,y)\right\rangle\right|<\infty.
\end{equation*}
\end{itemize}
The other results are true under similar considerations.

Note that one can find in \cite{kom} an overview of results concerning regular inductive limits of Banach spaces and several conditions (extensions of Silva's results)
which ensure that they have the Montel property. Since the regularity and the Montel property of DFS spaces were the only two crucial facts used above, the comments of this subsection are also valid for more general locally convex spaces.

Let us discuss an example in order to illustrate the ideas of this subsection.
\begin{example} \label{wnwex9.1}\emph{Fixation of variables in tempered distributions.} Let $f\in\mathcal{S}'(\mathbb{R}^{n}_{t}\times\mathbb{R}^{m}_{\xi})$ and $t_{0}\in\mathbb{R}^{n}$. Following \L ojasiewicz \cite{lojasiewicz2}, we say that the variable $t=t_{0}\in\mathbb{R}^{n}$ can be fixed in $f(t,\xi)$ if there exists $g\in\mathcal{S}'(\mathbb{R}^{m}_{\xi})$ such that for each $\eta\in\mathcal{S}(\mathbb{R}^{n}_{t}\times\mathbb{R}^{m}_{\xi})$
$$
\lim_{\varepsilon\to0^{+}}\left\langle f(t_{0}+\varepsilon t,\xi),\eta(t,\xi)\right\rangle=\int_{\mathbb{R}^{n}}\left\langle g(\xi),\eta(t,\xi)\right\rangle\mathrm{d}t.
$$
We write $f(t_{0},\xi)=g(\xi)$, \emph{distributionally}. The nuclearity of the Schwartz spaces \cite{treves,silva} implies that $\mathcal{S}'(\mathbb{R}^{n}_{t}\times\mathbb{R}^{m}_{\xi})$ is isomorphic to $\mathcal{S'}(\mathbb{R}_{t}^{n},E)$, where $E=\mathcal{S}'(\mathbb{R}^{m}_{\xi})$, a DFS space. Actually, the latter tells us that fixation of variables is nothing but the notion of \L ojasiewicz point values itself for $E$-valued distributions (cf. Example \ref{wnwex2.1}). Therefore, the DSF space-valued version of Theorem \ref{wnwth6.2} implies that if $\phi\in\mathcal{S}(\mathbb{R}^{n}_{t})$ with $\mu_{0}(\phi)=1$, then the variable $t=t_{0}$ can be fixed in $f(t,\xi)$ if and only if there exists $k$ such that for each $\rho\in\mathcal{S}'(\mathbb{R}^{m}_{\xi})$
$$
\limsup_{\varepsilon\rightarrow0^+}\sup_{\stackrel{(x,y)\in\mathbb{H}^{n+1}}{\left|x\right|^2+y^2=1}}
y^k \left|\left\langle f\left(t_0+\varepsilon
x+\varepsilon yt,\xi\right),\phi(t)\rho(\xi)\right\rangle\right|<\infty,
$$
and
$$\lim_{\varepsilon\to0^{+}}\left\langle  f\left(t_0+\varepsilon
x+\varepsilon yt,\xi\right),\phi(t)\rho(\xi)\right\rangle \mbox{ exists for all }(x,y)\in\mathbb{H}^{n+1}\cap\mathbb{S}^{n}.$$
\end{example}
\begin{remark}It is well known \cite{hormander1} that  the projection
$\pi: \mathbb{R}^{n}_{t}\times\mathbb{R}^{m}\rightarrow
\left\{t_0\right\}\times\mathbb{R}^{m}$, $\pi(t,\xi)=(t_0,\xi),$
defines the pull-back
$$\mathcal{S}'(\mathbb{R}^{n}_{t}\times\mathbb{R}^{m}_{\xi})\ni
f(t,\xi)\rightarrow f(t_{0},\xi):=\pi^*f(\xi)\in
\mathcal{S}'(\mathbb{R}^{m}_{\xi})$$ if the wave front set of $f$
 satisfies $$WF(f)\cap\{(t_0,\xi,\eta,0):\:\xi\in\mathbb R^{m}, \eta\in\mathbb R^{n}\}=\emptyset.$$ Thus the result given in Example \ref{wnwex9.1} is interesting since we give a necessary and sufficient condition for the existence of this
  pull-back.
\end{remark}

\newpage

\section*{A. Appendix\\ Relation between Weak-asymptotics in $\mathcal{S}'_{0}(\mathbb{R}^{n},E)$ and $\mathcal{S}'(\mathbb{R}^{n},E)$}

\label{ap}
The purpose of this Appendix is to show two propositions which establish the precise connection between weak-asymptotics in the spaces $\mathcal{S}'_{0}(\mathbb{R}^{n},E)$ and $\mathcal{S}'(\mathbb{R}^{n},E)$. Observe that such a relation was crucial for the arguments given in Section \ref{wnwtt}.

Propositions A.1 and A.2 below are multidimensional generalizations of the results from \cite[Sec. 4]{vindas-pilipovic-rakic} and their proofs are based on recent structural theorems from \cite{vindas4}. We assume again that $E$ is a Banach space.

\begin{propositionA1}
\label{wnwpA1} Let $L$ be slowly varying at the origin (resp. at infinity) and let $\mathbf{f}\in\mathcal{S}'(\mathbb{R}^{n},E)$ have weak-asymptotic behavior of degree $\alpha$ at the point $x_0$ (resp. at infinity) with respect to $L$ in $\mathcal{S}'_{0}(\mathbb{R}^{n},E)$, i.e., for each $\varphi\in\mathcal{S}_{0}(\mathbb{R}^{n})$ the following limit exists
\begin{equation*}
\tag{A.1}\label{wnweqA1}
\lim_{\varepsilon\to0^{+}}\frac{1}{ \varepsilon^{\alpha}L(\varepsilon)}\left\langle \mathbf{f}(x_{0}+\varepsilon t),\varphi(t)\right\rangle\ \ \  \mbox{in } E
\end{equation*}
$$\left(\mbox{resp. }\lim_{\lambda\to\infty}\frac{1}{ \lambda^{\alpha}L(\lambda)}\left\langle \mathbf{f}(\lambda t),\varphi(t)\right\rangle\:\right).
$$
Then, there is $\mathbf{g}\in\mathcal{S}'(\mathbb{R}^{n},E)$ such that:
\begin{itemize}
\item [(i)] If $\alpha\notin\mathbb{N}$, $\mathbf{g}$ is homogeneous of degree $\alpha$ and there
exists an $E$-valued polynomial $\mathbf{P}$ such that
\begin{equation*}\tag{A.2}
\label{wnweqA2}
\mathbf{f}(x_{0}+\varepsilon t)-\mathbf{P}(\varepsilon t)\sim \varepsilon^{\alpha}L(\varepsilon)\mathbf{g}(t)\ \ \ \mbox{as}\ \varepsilon\to0^{+}\ \ \mbox{in}\ \mathcal{S}'(\mathbb{R}^{n},E)
\end{equation*}
$$\left(\mbox{resp. }\mathbf{f}\left(\lambda t\right)-\mathbf{P}(\lambda t)\sim\lambda^{\alpha}L(\lambda)\mathbf{g}(t) \ \ \ \mbox{as}\ \lambda\to\infty \ \ \mbox{in}\ \mathcal{S}'(\mathbb{R}^{n},E)\:\right).
$$
\item [(ii)] If $\alpha=p\in\mathbb{N}$, $\mathbf{g}$ is associate
homogeneous of order 1 and degree $p$ (cf. \cite[p. 74]{estrada-kanwal2}, \cite{shelkovich}) satisfying
\begin{equation*}
\tag{A.3}
\label{wnweqA3}
\mathbf{g}(at)= a^{p}\mathbf{g}(t)+a^{p}\log a \sum _{\left|m\right|=p}t^{m}\mathbf{v}_{m} , \ \ \ \mbox{for each }a>0,
\end{equation*}
for some vectors $\mathbf{v}_{m}\in E$, $\left|m\right|=p$, and
there exist an $E$-valued polynomial $\mathbf{P}$ and associate asymptotically homogeneous $E$-valued functions $\mathbf{c}_{m}$, $\left|m\right|=p$, of degree 0 with respect to $L$ such that for each $a>0$
\begin{equation*}
\tag{A.4}
\label{wnweqA4}
\mathbf{c}_{m}(a \varepsilon)=\mathbf{c}(\varepsilon)+ L(\varepsilon)\log a \:\mathbf{v}_{m}+o(L(\varepsilon)) \ \ \ \mbox{as }\varepsilon\to0^{+}
\end{equation*}
$$
\left(\mbox{resp. }
\mathbf{c}_{m}(a \lambda)=\mathbf{c}(\lambda)+ L(\lambda)\log a \:\mathbf{v}_{m}+o(L(\lambda)) \ \ \ \mbox{as }\lambda\to\infty\:\right)
$$
and $\mathbf{f}$ has the following weak-asymptotic expansion
\begin{equation*}
\tag{A.5}
\label{wnweqA5}
\mathbf{f}(x_{0}+\varepsilon t)=\mathbf{P}(\varepsilon t)+ \varepsilon^{p}L(\varepsilon)\mathbf{g}(t)+\varepsilon^{p}\sum_{\left|m\right|=p}t^{m}\mathbf{c}_{m}(\varepsilon)+o\left(\varepsilon^{p}L(\varepsilon)\right)
\end{equation*}
$$\left(\mbox{resp. }\mathbf{f}\left(\lambda t\right)=\mathbf{P}(\lambda t)+\lambda^{p}L(\lambda)\mathbf{g}(t)+\lambda^{p}\sum_{\left|m\right|=p}t^{m}\mathbf{c}_{m}(\lambda)+o\left(\lambda^{p}L(\lambda)\right)\:\right)
$$
as $\varepsilon\to0^{+}$ (resp. $\lambda\to\infty$) in the space $\mathcal{S}'(\mathbb{R}^{n},E)$.
\end{itemize}
\end{propositionA1}
\begin{proof}

Let $\mathcal{S}^0(\mathbb{R}^{n})$ be the image under Fourier transform of $\mathcal{S}_{0}(\mathbb{R}^{n})$. Then, $\mathcal{S}^0(\mathbb{R}^{n})$ is precisely the closed subspace of $\mathcal{S}(\mathbb{R}^{n})$ consisting of test functions which vanish at the origin together with their partial derivatives of any order. Thus, if we Fourier transform (\ref{wnweqA1}) and employ the Banach-Steinhaus theorem, we obtain the existence of $\mathbf{h}_{0}\in {\mathcal{S}^0}'(\mathbb{R}^{n},E)$ such that the restriction of $\mathbf{f}$ to $\mathcal{S}^0(\mathbb{R}^{n})$ satisfies
$$
\exp({i\varepsilon^{-1}u\cdot x_{0}})\hat{\mathbf{f}}(\varepsilon^{-1}u)\sim \varepsilon^{n+\alpha}L(\varepsilon)\mathbf{h}_{0}(u)\ \ \ \text{as}\ \varepsilon\to0^{+}\ \ \text{in} \ {\mathcal{S}^{0}}'(\mathbb{R}^{n},E)
$$
$$
\left(\text{resp. } \hat{\mathbf{f}}(\lambda^{-1}u)\sim \lambda^{n+\alpha}L(\lambda)\mathbf{h}_{0}(u)\ \ \ \text{as}\ \lambda\to\infty\ \ \text{in} \ {\mathcal{S}^{0}}'(\mathbb{R}^{n},E)\:\right).
$$
Setting $\tilde{\mathbf{f}}(u)= e^{i\:u\cdot x_{0}}\hat{\mathbf{f}}(u)$ (resp. $\tilde{\mathbf{f}}(u)=\hat{\mathbf{f}}(u)\:$), $\tilde{L}(y)=L(1/y)$, $\beta=-n-\alpha$ and replacing $\varepsilon$ by $\lambda^{-1}$, we have that the restriction of $\tilde{\mathbf{f}}$ to $\mathcal{S}^0(\mathbb{R}^{n})$ has the weak-asymptotic behavior
\begin{equation*}
\tag{A.6}
\label{wnweqA6}
\tilde{\mathbf{f}}(\lambda u)\sim \lambda^{\beta}\tilde{L}(\lambda)\mathbf{h}_{0}(u)\ \ \ \text{as}\ \lambda\to\infty\ \ \text{in} \ {\mathcal{S}^{0}}'(\mathbb{R}^{n},E)
\end{equation*}
$$
\left(\text{resp. } \tilde{\mathbf{f}}(\varepsilon u)\sim \varepsilon^{\beta}\tilde{L}(\varepsilon)\mathbf{h}_{0}(u)\ \ \ \text{as}\ \varepsilon\to0^{+}\ \ \text{in} \ {\mathcal{S}^{0}}'(\mathbb{R}^{n},E) \:\right),
$$
for some $\mathbf{h}_{0}\in{\mathcal{S}^{0}}'(\mathbb{R}^{n},E)$. We now apply the results from \cite{vindas4}.

\emph{Case (i): $\alpha\notin \mathbb{N}$.} By (\ref{wnweqA6}) and \cite[Part (i) of Thm. 3.1]{vindas4}, there are an $E$-valued distribution $\mathbf{h}\in \mathcal{S}'(\mathbb{R}^{n},E)$, which is homogeneous of degree $\beta=-n-\alpha$, an natural number $d\in\mathbb{N}$, and $\mathbf{w}_{m}\in E$, $\left|m\right|\leq d$, such that

$$
\tilde{\mathbf{f}}(\lambda u)=\lambda^{\beta}\tilde{L}(\lambda)\mathbf{h}(u)+\sum_{\left|m\right|\leq d}\frac{\delta^{(m)}(u) }{\lambda^{n+\left|m\right|}}\mathbf{w}_{m}+o\left(\lambda^{\beta}\tilde{L}(\lambda)\right)\ \ \ \text{as}\ \lambda\to\infty
$$
$$
\left(\mbox{resp. }\tilde{\mathbf{f}}(\varepsilon u)=\varepsilon^{\beta}\tilde{L}(\varepsilon)\mathbf{h}(u)+\sum_{\left|m\right|\leq d}\frac{\delta^{(m)}(u) }{\varepsilon^{n+\left|m\right|}}\mathbf{w}_{m}+o\left(\varepsilon^{\beta}\tilde{L}(\varepsilon)\right)\ \ \text{as}\ \varepsilon\to0^{+}\:\right)
$$
in $\mathcal{S}'(\mathbb{R}^{n},E)$.
Finally, by setting $\hat{\mathbf{g}}=\mathbf{h}$, taking Fourier inverse transform and replacing $\lambda$ by $\varepsilon^{-1}$ (resp. $\varepsilon$ by $\lambda^{-1}$), the last relation shows that $\mathbf{f}$ satisfies (\ref{wnweqA2}) with $\mathbf{P}(t)=(1/2\pi)^{n}\sum_{\left|m\right|\leq d}(-it)^{m}\mathbf{w}_{m}$.

\emph{Case (ii): $\beta=-n-p$, $p\in\mathbb{N}$.} The weak-asymptotics (\ref{wnweqA6}) and \cite[Part (ii) of Thm. 3.1]{vindas4}
yield the existence of $d\in\mathbb{N}$, $\mathbf{w}_{m}\in E$ (for $\left|m\right|\leq d$), $\tilde{\mathbf{v}}_{m}\in E$ (for $\left|m\right|=p$), continuous functions $\tilde{\mathbf{c}}_{m}:\mathbb{R}_{+}\to E$ (for $\left|m\right|=p$), and a tempered $E$-valued distribution $\mathbf{h}\in\mathcal{S}'(\mathbb{R}^{n},E)$ such that $\tilde{\mathbf{f}}$ has the following asymptotic expansion in $\mathcal{S}'(\mathbb{R}^{n},E)$ as $\lambda\to\infty$ (resp. $\varepsilon\to0^{+}$)
\begin{equation*}
\tilde{\mathbf{f}}(\lambda u)=\frac{\tilde{L}(\lambda)}{\lambda^{n+p}}\mathbf{h}(u)
+\sum_{\left|m\right|\leq d}
\frac{\delta^{(m)}(u)}{\lambda^{n+\left|m\right|}}{\mathbf{w}}_{m}
+\sum_{\left|m\right|= p}\frac{\delta^{(m)}(u)}{\lambda^{n+p}}\tilde{\mathbf{c}}_{m}(\lambda)+o\left(\frac{\tilde{L}(\lambda)}{\lambda^{n+p}}\right),
\end{equation*}
respectively
\begin{equation*}
\tilde{\mathbf{f}}(\varepsilon u)=\frac{\tilde{L}(\varepsilon)}{\varepsilon^{n+p}}\mathbf{h}(u)
+\sum_{\left|m\right|\leq d}
\frac{\delta^{(m)}(u)}{\varepsilon^{n+\left|m\right|}}{\mathbf{w}}_{m}
+\sum_{\left|m\right|= p}\frac{\delta^{(m)}(u)}{\varepsilon^{n+p}}\tilde{\mathbf{c}}_{m}(\varepsilon)+o\left(\frac{\tilde{L}(\varepsilon)}{\varepsilon^{n+p}}\right),
\end{equation*}
where $\mathbf{h}$ satisfies
$$
\mathbf{h}(au)= a^{-n-p}\mathbf{h}(u)+a^{-n-p}\log a\sum _{\left|m\right|=p}\delta^{(m)}(u)\tilde{\mathbf{v}}_{m},
$$
for each $a>0$, while the $\tilde{\mathbf{c}}_{m}$ fulfill
\begin{equation*}
\tilde{\mathbf{c}}_{m}(a\lambda)= \tilde{\mathbf{c}}_{m}(\lambda) +\tilde{L}(\lambda)\log a\:\tilde{\mathbf{v}}_{m}+o\left(\tilde{L}(\lambda)\right), \ \ \ \left|m\right|=p,
\end{equation*}
\begin{equation*}
\left(\mbox{resp. }\tilde{\mathbf{c}}_{m}(a\varepsilon)= \tilde{\mathbf{c}}_{m}(\varepsilon) +\tilde{L}(\varepsilon)\log a\:\tilde{\mathbf{v}}_{m}+o\left(\tilde{L}(\varepsilon)\right)\:\right).
\end{equation*}
Then, Fourier inverse transforming the weak-asymptotic expansion of $\tilde{\mathbf{f}}$, we convince ourselves that $\mathbf{f}$ satisfies (\ref{wnweqA5}) with the polynomial $\mathbf{P}(t)=(2\pi)^{-n}\sum_{\left|m\right|\leq d}(-it)^{m}\mathbf{w}_{m}$, the functions $\mathbf{c}_m(y)= (-i)^{p}(2\pi)^{-n}\tilde{\mathbf{c}}_{m}(y^{-1})$ and $\mathbf{g}$ given by $\hat{\mathbf{g}}=\mathbf{h}$. In addition, the relations (\ref{wnweqA3}) and (\ref{wnweqA4}) hold with $\mathbf{v}_{m}=-(-i)^{p}(2\pi)^{-n}\tilde{\mathbf{v}}_{m}$ , $\left|m\right|=p$.

\end{proof}

The proof of the following proposition is completely analogous to that of Proposition A.1, but now making use of \cite[Thm. 3.2]{vindas4} instead of \cite[Thm. 3.1]{vindas4}; we therefore omit it.
\begin{propositionA2}
Let $L$ be slowly varying at the origin (resp. at infinity) and let $\mathbf{f}\in\mathcal{S}'(\mathbb{R}^{n},E)$ be weak-asymptotically bounded of degree $\alpha$ at the point $x_{0}$ (resp. at infinity) with respect to $L$ in $\mathcal{S}'_{0}(\mathbb{R}^{n},E)$.
Then:
\begin{itemize}
\item [(i)] If $\alpha\notin\mathbb{N}$, there exists an $E$-valued polynomial $\mathbf{P}$ such that
$\mathbf{f}-\mathbf{P}$ is weak-asymptotically bounded of degree $\alpha$ at the point $x_{0}$ (resp. at infinity) with respect to $L$ in the space $\mathcal{S}'(\mathbb{R}^{n},E)$.
\item [(ii)] If $\alpha=p\in\mathbb{N}$, there exist an $E$-valued polynomial $\mathbf{P}$ and asymptotically homogeneously bounded $E$-valued functions $\mathbf{c}_{m}$, $\left|m\right|=p$, of degree 0 with respect to $L$ such that $\mathbf{f}$ has the following weak-asymptotic expansion
\begin{equation*}
\tag{A.7}
\label{wnweqA7}
\mathbf{f}(x_{0}+\varepsilon t)=\mathbf{P}(\varepsilon t)+\varepsilon^{p}\sum_{\left|m\right|=p}t^{m}\mathbf{c}_{m}(\varepsilon)+O\left(\varepsilon^{p}L(\varepsilon)\right)
\end{equation*}
$$\left(\mbox{resp. }\mathbf{f}\left(\lambda t\right)=\mathbf{P}(\lambda t)+\lambda^{p}\sum_{\left|m\right|=p}t^{m}\mathbf{c}_{m}(\lambda)+O\left(\lambda^{p}L(\lambda)\right)\:\right) ,
$$
as $\varepsilon\to0^{+}$ (resp. $\lambda\to\infty$) in the space $\mathcal{S}'(\mathbb{R}^{n},E)$.
\end{itemize}
\end{propositionA2}

\end{document}